\documentclass[12pt,oneside,british,a4wide]{amsart}
\usepackage{amsmath, amssymb,verbatim,appendix}
\usepackage{hyperref}
\usepackage[mathscr]{eucal}
\usepackage{amscd}
\usepackage{amsthm}
\usepackage{stmaryrd}
\usepackage{enumerate}
\usepackage{comment}
\usepackage[latin1]{inputenc} 
\usepackage{tikz}
\usetikzlibrary{shapes,arrows}
\usepackage{cite}
\usepackage{url}

\makeatletter
\newcommand{\dotminus}{\mathbin{\text{\@dotminus}}}

\newcommand{\@dotminus}{%
  \ooalign{\hidewidth\raise1ex\hbox{.}\hidewidth\cr$\m@th-$\cr}%
}
\makeatother

\usepackage{setspace,a4wide,color,xcolor,graphicx}
\def\showauthornotes{1}

\ifnum\showauthornotes=1
\newcommand{\Authornote}[2]{{\sf\small\color{red}{[#1: #2]}}}
\else
\newcommand{\Authornote}[2]{}
\fi

\newtheorem{theorem}{Theorem}[section]
\newtheorem{lemma}[theorem]{Lemma}
\newtheorem{corollary}[theorem]{Corollary}
\newtheorem{proposition}[theorem]{Proposition}

\newtheorem{conjecture}[theorem]{Conjecture}
\newtheorem{problem}[theorem]{Problem}
\newtheorem{observation}[theorem]{Observation}

\newtheorem{claim}[theorem]{Claim}
\newtheorem{fact}[theorem]{Fact}

\theoremstyle{definition}

\newtheorem{definition}[theorem]{Definition}

\def\e{\epsilon}

\def\trip{\text{Trip}}

\newcommand{\VC}{\textnormal{VC}}

\newcommand{\disc}{\textnormal{disc}}

\newcommand{\dev}{\textnormal{dev}}

\newcommand{\calH}{\mathcal{H}}
\newcommand{\calC}{\mathcal{C}}

\newcommand{\calN}{\mathcal{N}}

\newcommand{\calB}{\mathcal{B}}

\newcommand{\calP}{\mathcal{P}}

\newcommand{\calE}{\mathcal{E}}
\newcommand{\calG}{\mathcal{G}}

\newcommand{\calQ}{\mathcal{Q}}

\newcommand{\calS}{\mathcal{S}}

\newcommand{\calU}{\mathcal{U}}

\newcommand{\calW}{\mathcal{W}}

\newcommand{\calZ}{\mathcal{Z}}

\def\triads{\operatorname{Triads}}

\begin{document}
\title[]{Growth of regular partitions 4: strong regularity and the pairs partition}
\author{C. Terry}
\thanks{The author was partially supported by NSF CAREER Award DMS-2115518 and a Sloan Research Fellowship}
\address{Department of Mathematics, Statistics, and Computer Science, University of Illinois at Chicago, Chicago, IL 60607, USA}
\email{caterry@uic.edu}
\maketitle

\begin{abstract}
This paper studies bounds in a strong form of regularity for $3$-uniform hypergraphs which was developed by Frankl, Gowers, Kohayakawa, Nagle,  R\"{o}dl, Skokan, and Schacht.  Regular decompositions of this type involve two structural components: a partition on the vertex set and a partition on the pairs of vertices.  The regularity of such decompositions are measured by  two parameters: an $\e_1>0$ and a function $\e_2:\mathbb{N}\rightarrow (0,1]$.  To each hereditary property $\calH$ of $3$-uniform hypergraphs, we associate two corresponding growth functions: $T_{\calH}(\e_1,\e_2)$ for the size of the vertex component, and $L_{\calH}(\e_1,\e_2)$ for the size of the pairs component.  The problem of understanding the asymptotic growth of such functions was introduced in a companion paper, which also proved several results about $T_{\calH}$. In this paper we study the possible asymptotic behavior of $L_{\calH}$.  We show any such function is either constant, bounded above and below by a polynomial, or bounded below by an exponential.  All results require only reasonable growth rates for $\e_2$ (namely polynomial). 
\end{abstract}

\section{Introduction}

This is the fourth in a series of papers about the possible growth rates of regular partitions of $3$-uniform hypergraphs (see \cite{Terry.2024a,Terry.2024b,Terry.2024c}).  Part 3 of the series \cite{Terry.2024c} contains extensive overlap with this paper in terms of motivation  and technical preliminaries.  For this reason, we give an  abridged  introduction here, and refer the reader to \cite{Terry.2024c} for  more extensive  background.

Szemer\'{e}di's reguarlity lemma states that for all $\e>0$, there is an integer $M=M(\e)$ so that any finite graph can be partitioned into at most $M$ parts, so that most pairs of parts are $\e$-regular.  This allows us  associate the following growth function to a hereditary graph property. 

\begin{definition}\label{def:M}
Given a hereditary graph property $\calH$, define $M_{\calH}:(0,1]\rightarrow \mathbb{N}$ as follows.  For all $\e>0$, let $M_{\calH}(\e)$ be the smallest integer so that any sufficiently large  graph  in $\calH$ has an $\e$-regular partition with at most $M_{\calH}(\e)$ parts.  
\end{definition}

  Theorem \ref{thm:alljump} summarizes the possible assymptotic growth rates for such a function: constant, polynomial, or tower.  It combines theorems due to Alon-Fischer-Newman, Lov\'{a}sz-Szegedy, and Fox-Lov\'{a}sz \cite{Fox.2014, Alon.2007, Lovasz.2010} and results of the author from \cite{Terry.2024b}.  For more details, we refer the reader to \cite{Terry.2024b}.

\begin{theorem}\label{thm:alljump}
Suppose $\calH$ is a hereditary graph property.  Then one of the following hold. 
\begin{enumerate}
\item (Tower) For some $C>0$, $Tw(\e^{-C})\leq M_{\calH}(\e)\leq O(Tw(\e^{-2}))$,
\item (Exponential) For some $C>0$, $\e^{-1+o(1)}\leq M_{\calH}(\e)\leq O(\e^{-C})$, or
\item (Constant) For some $C\geq 1$, $M_{\calH}(\e)=C$.
\end{enumerate}
\end{theorem}

The goal of this paper and its companions is to prove analogues of Theorem \ref{thm:alljump} for hereditary properties of $3$-uniform hypergraphs.  The first and second papers in the series dealt with the weak regularity of \cite{Chung.1991, Haviland.1989, ChungGrahamHG}.  Part 3 of the series and the current paper  focus on a version of hypergraph regularity developed by Frankl, Gowers, Kohayakawa, Nagle,  R\"{o}dl, Skokan, and Schacht \cite{Frankl.2002, Gowers.20063gk, Gowers.2007, Rodl.2005, Rodl.2004, Nagle.2013}.  In this context, a regular decomposition $\calP$ for a $3$-uniform hypergraph $H=(V,E)$ consists of two components: a vertex partition $\calP_1=\{V_1,\ldots, V_t\}$, and a set of the form 
$$
\calP_2=\{P_{ij}^{\alpha}:1\leq i,j\leq t , 1\leq \alpha\leq \ell\},
$$
 where for each $1\leq i,j\leq t$, $P_{ij}^1\cup \ldots \cup P_{ij}^{\ell}$ is a partition of $V_i\times V_j$.  Such a regular decomposition comes equipped with a pair of complexity parameters $(t,\ell)$, and a pair of error parameters, $\e_1$, and a function $\e_2:\mathbb{N}\rightarrow (0,1]$.  Roughly speaking, the elements of $\calP_2$ are required to be $\e_2(\ell)$-regular, and the edges of $H$ are required to be $\e_1$-uniformly distributed on sets of triangles formed by elements of $\calP_2$.  We state the relevant regularity lemma here (due to Gowers), and refer the reader to Section \ref{ss:regularity} for more details.
 
\begin{theorem}[Gowers \cite{Gowers.2007}]\label{thm:reg2} For all $\e_1>0$, every function $\e_2:\mathbb{N}\rightarrow (0,1]$, there exist positive integers $T=T(\e_1,\e_2)$ and $L=L(\e_1,\e_2)$, such that for every sufficiently large $3$-graph $H=(V,E)$, there exists a $\dev_{2,3}(\e_1,\e_2(\ell))$-regular, $(t,\ell,\e_1,\e_2(\ell))$-decomposition $\calP$ for $H$ with $1\leq t\leq T$ and $1\leq \ell \leq L$.
\end{theorem}
 
The proof of Theorem \ref{thm:reg2} produces wowzer type bounds for $T$ and $L$ for $\e_2$ tending to $0$ at an at most tower rate, and even worse bounds for faster $\e_2$  (for more details, see the introduction of part 3 \cite{Terry.2024c}).  

In part 3 of this series, growth functions were defined corresponding to the two size parameters appearing in Theorem \ref{thm:reg2}, one for $T$ and one for $L$.  To state these definitions, we require the following notation. Given a hereditary property $\calH$ of $3$-uniform hypergraphs, $\e_1>0$,  $\e_2:\mathbb{N}\rightarrow (0,1]$ and integers $T,L\geq 1$, let $\psi(\e_1,\e_2,T,L,\calH)$ be the statement:
\begin{align*}
&\text{For every sufficiently large $H\in \calH$, there exist $1\leq t\leq T$ and $1\leq \ell\leq L$}\\ &\text{and a $\dev_{2,3}(\e_1,\e_2(\ell))$-regular $(t,\ell,\e_1,\e_2(\ell))$-decomposition $\calP$ for $H$}.
\end{align*}
We now repeat the definitions of the two growth functions.

\begin{definition}\label{def:vertM}
Suppose $\calH$ is a hereditary property of $3$-uniform hypergraphs.  Given $\e_1>0$ and $\e_2:\mathbb{N}\rightarrow (0,1)$, let $T_{\calH}(\e_1,\e_2)$ be minimal so that $\psi(\e_1,\e_2,T_{\calH}(\e_1,\e_2),L,\calH)$ is true for some $L\geq 1$.
\end{definition}

\begin{definition}\label{def:pairsM}
Suppose $\calH$ is a hereditary property of $3$-uniform hypergraphs.  Given $\e_1>0$ and $\e_2:\mathbb{N}\rightarrow (0,1)$, let $L_{\calH}(\e_1,\e_2)$ be minimal so that $\psi(\e_1,\e_2,T,L_{\calH}(\e_1,\e_2),\calH)$ is true for some $T\geq 1$.
\end{definition}

Apart from the existence of the upper bounds from Theorem \ref{thm:reg2}, there are only two prior results about these bounds, to the author's knowledge. The first is an important  result about the $T_{\calH}$  due to Moshkovitz and Shapira \cite{MoshkovitzSimple, Moshkovitz.2019}. In particular, they showed there exist properties $\calH$, and slow-growing functions $\e_2$, for which $T_{\calH}(\e_1,\e_2)$ is bounded below by a wowzer-type function.   The second is a result due to the author, which says that when $\calH$ has finite $\VC_2$-dimension, $L_{\calH}$ is bounded above by a polynomial in $\e_1^{-1}$ \cite{Terry.2022}.  

Part 3 \cite{Terry.2024c} introduced the problem of understanding  the asymptotic behavior of functions of the form $T_{\calH}$ and $L_{\calH}$ for $\calH$ a hereditary property of $3$-uniform hypergraphs.  

\begin{problem}\label{prob:main}
Investigate the asymptotic behavior of $T_{\calH}(\e_1,\e_2)$ and $L_{\calH}(\e_1,\e_2)$ for $\e_1>0$ sufficiently small, and  $ \e_2:\mathbb{N}\rightarrow (0,1]$ tending to $0$ sufficiently fast.  
\end{problem}

As explained in part 3 \cite{Terry.2024c}, it is also of interest to restrict to the case where $\e_2$ grows at most polynomially, a problem suggested to the author by Shapira.

\begin{problem}[Shapira]\label{prob:poly}
  $\text{ }$
\begin{enumerate}
\item Investigate the asymptotic behavior of $T_{\calH}(\e_1,\e_2)$  for $\e_1>0$ sufficiently small, and {\bf polynomial} $ \e_2:\mathbb{N}\rightarrow (0,1]$ tending to $0$ sufficiently fast.  
\item Investigate the asymptotic behavior of $L_{\calH}(\e_1,\e_2)$  for $\e_1>0$ sufficiently small, and {\bf polynomial} $ \e_2:\mathbb{N}\rightarrow (0,1]$ tending to $0$ sufficiently fast.  
\end{enumerate}
\end{problem}

Part 3 \cite{Terry.2024c} addresses the $T_{\calH}$ part of Problem \ref{prob:main}, and in doing so proves several results relevant to Problem \ref{prob:poly}(1).  

\begin{theorem}[Main result of part 3 \cite{Terry.2024c}]\label{thm:strong1}
Suppose $\calH$ is a hereditary property of $3$-uniform hypergraphs.  Then one of the following holds.
\begin{enumerate}
\item (At Least Wowzer) For all sufficiently small $\e_1>0$ there exists $\e_2:\mathbb{N}\rightarrow (0,1]$ so that $ W(\Omega(\e_1^{-1/7}))\leq T_{\calH}(\e_1,\e_2)$.
\item (Exponential) There exist $C,C'>0$ and a polynomial $p(x,y)$ so that for all sufficiently small $\e_1>0$, and all $\e_2:\mathbb{N}\rightarrow(0,1]$ satisfying $\e_2(x)\leq p(\e_1,x^{-1})$, 
$$
2^{\e_1^{-C}}\leq T_{\calH}(\e_1,\e_2)\leq 2^{\e_1^{-C'}}.
$$
\item (Polynomial) There exist $C,C'>0$ and a polynomial $p(x,y)$ so that for all sufficiently small $\e_1>0$, and all $\e_2:\mathbb{N}\rightarrow(0,1]$ satisfying $\e_2(x)\leq p(\e_1,x^{-1})$, 
 $$
 \e_1^{-C}\leq T_{\calH}(\e_1,\e_2)\leq \e_1^{-C'}.
 $$
 
\item (Constant) There exists $C>0$ and a polynomial $p(x,y)$ so that for all sufficiently small $\e_1>0$, and all $\e_2:\mathbb{N}\rightarrow(0,1]$ satisfying $\e_2(x)\leq p(\e_1,x^{-1})$, 
$$
T_{\calH}(\e_1,\e_2)=C.
$$
\end{enumerate}
\end{theorem}

This paper's focus is on the $L_{\calH}$ part of Problem \ref{prob:main}.  We have in this updated version taken care to make explicit the fact that all our proofs require only polynomial $\e_2$, and are thus relevant to Problem \ref{prob:poly}(2).  Our main result, Theorem \ref{thm:mainL2} below, shows there are at least three distinct growth classes for $L_{\calH}$. 

\begin{theorem}\label{thm:mainL2}
Suppose $\calH$ is a hereditary property of $3$-uniform hypergraphs.  Then one of the following hold.
\begin{enumerate}
\item (At least exponential) There exist $C>0$ and a polynomial $p(x,y)$ so that for all $\e_2:\mathbb{N}\rightarrow (0,1]$ satisfying $\e_2(x)\leq p(\e_1,x^{-1})$, 
$$
2^{\e_1^{-C}}\leq L_{\calH}(\e_1,\e_2),
$$
\item (Polynomial) There exist $C,C'>0$ and a polynomial $p(x,y)$ so that for all $\e_2:\mathbb{N}\rightarrow (0,1]$ satisfying $\e_2(x)\leq p(\e_1,x^{-1})$, 
$$
\e_1^{-C}\leq L_{\calH}(\e_1,\e_2)\leq \e_1^{-C'},
$$
\item (Constant) There exist $C\geq 1$ and a polynomial $p(x,y)$ so that for all $\e_2:\mathbb{N}\rightarrow (0,1]$ satisfying $\e_2(x)\leq p(\e_1,x^{-1})$, 
$$
L_{\calH}(\e_1,\e_2)=C.
$$
\end{enumerate}
\end{theorem}
The proof of this theorem draws on tools developed by the author over the course of several papers, namely \cite{Terry.2021b,Terry.2022, Terry.2024b, Terry.2024a}. We note the polynomial upper bound in (2) was first proved in \cite{Terry.2022}. We will reprove it in this paper  using similar techniques, in part to provide a template for the very similar proof of the upper bound in (3).  Our proofs  will also make explicit the fact that the upper bounds in (2) and (3) can be obtained while simultaneously minimizing the size of the vertex partition relative to parameters $\e_1'$ and $\e_2'$ which are polynomially related to the original $\e_1$ and $\e_2$.    

Theorem \ref{thm:mainL2} suggests a close connection between the growth of $L_{\calH}$ and the growth of $M_{\calH}$ in the graphs setting.  This can be seen in the similarity of the growth rates in Theorems \ref{thm:alljump} and  \ref{thm:mainL2}.  Further, many of the combinatorial arguments in the proof of Theorem \ref{thm:mainL2} are related to the proof of Theorem \ref{thm:alljump} appearing in \cite{Terry.2024b}.

Our proof of Theorem \ref{thm:mainL2} will give combinatorial characterizations of the properties falling into each growth class.  For instance, we will show the properties in ranges (2) and (3) of Theorem \ref{thm:mainL2} are exactly those of finite $\VC_2$-dimension (see Section \ref{sec:VC2} for details).  As a corollary of the combinatorial characterizations arising from Theorems \ref{thm:mainL2} and Theorem \ref{thm:strong1}, we prove that $L_{\calH}(\e_1,\e_2)=1$  if and only if $T_{\calH}$ falls into ranges (2),(3), or (4) in Theorem \ref{thm:strong1}.

\begin{corollary}\label{cor:L1}
Suppose $\calH$ is a hereditary property of $3$-uniform hypergraphs.  Then the following are equivalent.
\begin{enumerate}
\item For some $k\geq 1$, $k\otimes U(k)\notin \trip(\calH)$,
\item There exists $\e_1^*>0$ and a polynomial $p(x,y)$ so that for all $0<\e_1<\e_1^*$ and all $\e_2:\mathbb{N}\rightarrow (0,1]$ satisfying $\e_2(x)\leq p(\e_1,x^{-1})$,  
$$
L_{\calH}(\e_1,\e_2)=1.
$$
\item There exists $\e_1^*,C>0$ and $\e^*_2:\mathbb{N}\rightarrow (0,1]$ so that for all $0<\e_1<\e_1^*$ and all $\e_2:\mathbb{N}\rightarrow (0,1]$ satisfying $\e_2(x)\leq \e_2^*(x)$,  
$$
T_{\calH}(\e_1,\e_2)\leq 2^{\e_1^{-C}}.
$$
\end{enumerate}
\end{corollary}

 We end this subsection with a discussion of open problems.  First, we conjecture the lower bound in Theorem \ref{thm:mainL2} (1) can be improved to at least a tower function.  The author believes this could possibly be proved by combining techniques from this paper with a good understanding of bipartite lower bound constructions for graph regularity (see e.g. \cite{Fox.2014, Moshkovitz.2019}). 

\begin{conjecture}\label{conj2}
Suppose $\calH$ is a hereditary property of $3$-uniform hypergraphs with infinite $\VC_2$-dimension.  Then for all sufficiently small $\e_1>0$, there is some $\e_2:\mathbb{N}\rightarrow (0,1]$ and $C>0$, $L_{\calH}(\e_1,\e_2)\geq Tw(\e_1^{-C})$. 
\end{conjecture}
 
Whether Conjecture \ref{conj2} holds restricted to polynomial  $\e_2$ is also an interesting open problem. It seems likely  that  within the class of properties falling into range (1) of Theorem \ref{thm:mainL2} (i.e. those of infinite $\VC_2$-dimension),  the behavior of $L_{\calH}$ is dependent on the choice of $\e_2$, as suggested by  the general upper bound for $L(\e_1,\e_2)$  (see the introduction to \cite{Terry.2024c} for details).   Nonetheless,   we conjecture that properties with infinite $\VC_2$-dimension will always exhibit ``fastest possible" $L_{\calH}$-growth. Since the meaning of ``fastest possible" may very  well depend on the growth rate of $\e_2$, we will content ourselves with the following qualitative conjecture.

\begin{conjecture}\label{conj1}
 There are no more ``jumps" in range (1) of Theorem \ref{thm:mainL2}.
\end{conjecture}

\subsection{Acknowledgements} The author would like to thank Asaf Shapira for several helpful suggestions, and for pointing out subtleties the author previously overlooked regarding the possible dependence of the growth rate in range (1) of Theorem \ref{thm:mainL2} on the growth rate of $\e_2$. Finally, the author thanks an anonymous referee led us to prepare the more detailed presentation here regarding the growth rate of $\e_2$.

\subsection{Outline} We give here an outline of the rest of the paper.  The final subsection of the introduction, Subsection  \ref{ss:notation}, will cover basic notational conventions for the paper.  In Section \ref{sec:reg}, we cover preliminaries related to hypergraph regularity, largely paraphrasing the analogous section from Part 3.  Section \ref{sec:VC2} is also based on the analogous section in Part 3, and covers the definitions of $\VC_2$-dimension and homogeneous decompositions. In Section \ref{sec:Ltools}, we introduce tools specific to the $L_{\calH}$ problem,  including irreducible bigraphs, edge-colored bigraphs, corner graphs and encodings, and a special type of blow up.  Section \ref{ss:UBL} contains proofs for the polynomial and constant upper bounds in Theorem \ref{thm:mainL2}.  Section \ref{sec:LBL} contains lower bound constructions for Theorem \ref{thm:mainL2}.  Section \ref{sec:L2} combines these results to prove Theorem \ref{thm:mainL2}.

\subsection{Notation}\label{ss:notation}
We now introduce notation we will use in the paper.  Due to the similarity of this paper and Part 3, this section is roughly the same as Section 2 there.  Given a natural number $n\geq 1$, we let $[n]=\{1,\ldots, n\}$.  For real numbers $r_1,r_2$ and $\e>0$, we write $r_1=r_2\pm \e$  to mean $|r_1-r_2|\leq \e$.  An \emph{equipartition} of a set $V$ is a partition $V=V_1\cup \ldots \cup V_t$ with the property that for each $1\leq i,j\leq t$, $||V_i|-|V_j||\leq 1$. 

Given a set $V$ and $k\geq 1$, let ${V\choose k}=\{X\subseteq V: |X|=k\}$.  A \emph{$k$-uniform hypergraph} is a pair $(V,E)$ where $E\subseteq {V\choose k}$.  For a $k$-uniform hypergraph $G$, we let $V(G)$ denote the vertex set of $V$ and $E(G)$ denote the edge set of $G$. We will refer to $k$-uniform hypergraphs as \emph{$k$-graphs} and $2$-uniform hypergraphs as \emph{graphs}.  

A $k$-graph $G=(V,E)$ is $\ell$-partite if there is a partition $V=V_1\cup \ldots \cup V_{\ell}$ so that for all $e\in E$ and $1\leq i\leq \ell$, $|e\cap V_i|\leq 1$.  In this case, we write $G=(V_1\cup \ldots \cup V_{\ell},E)$ to denote that $G$ is $\ell$-partite with vertex partition given by $V_1\cup \ldots \cup V_{\ell}$.

Given distinct elements $x,y$, we will write $xy$ for the set $\{x,y\}$.   Given sets $X,Y$, we set 
\begin{align*}
K_2[X,Y]&=\{xy: x\in X, y\in Y, x\neq y\}.
\end{align*}
Suppose $G=(V,E)$ is a graph.  We will need to frequently refer to the set of ordered pairs coming from edges of $G$.  For this reason we define the following set.
$$
\overline{E}=\{(x,y)\in V^2: xy\in E\}.
$$
Given $X,Y\subseteq V$, write 
$$
d_G(X,Y)=|\overline{E}\cap (X\times Y)|/|X||Y|.
$$ 
Note that if $X$ and $Y$ are disjoint, then $d_G(X,Y)=|E\cap K_2[X,Y]|/|X||Y|$.  Given $x\in V$, the \emph{neighborhood of $x$ in $G$} is $N_G(x)=\{y\in V: xy\in E\}$.  

We will use similar notation in the following more general contexts.  Suppose $X$ and $Y$ are sets.  Given a set $F\subseteq {X\choose 2}$ and $x\in X$, we write $N_F(x)=\{y\in X: xy\in F\}$. Note that with the  notation we have defined, we could write the neighborhood of a vertex in a graph $G$ as either $N_G(x)$ or $N_E(x)$ where $E$ is the edge set of $G$.  Similarly, for $F\subseteq X\times Y$ and $x\in X$, write $N_F(x)=\{y\in V: (x,y)\in F\}$.

We now set similar notation for $3$-graphs.  To begin with, for three distinct elements $x,y,z$, we will write $xyz$ for the set $\{x,y,z\}$. Given sets $X,Y,Z$, we set 
\begin{align*}
K_3[X,Y,Z]&=\{xyz: x\in X, y\in Y, z\in Z, x\neq y, y\neq z, x\neq z\}.
\end{align*}
Suppose $G=(V,E)$ is a $3$-graph.  We define
$$
\overline{E}=\{(x,y,z)\in V^3: xy\in E\}.
$$
Given $X,Y,Z\subseteq V$, write 
$$
d_G(X,Y,Z)=|\overline{E}\cap (X\times Y\times Z)|/|X||Y||Z|.
$$ 
For disjoint subsets $X,Y,Z\subseteq V$, we let $G[X,Y,Z]$ be the tripartite $3$-graph 
$$
(X\cup Y\cup Z, E\cap K_3[X,Y,Z]).
$$
Given $x,y\in V$, let $N_G(x)=\{uw\in V: xuw\in E\}$ and $N_G(xy)=\{w\in V: xyw\in E\}$.  

We will use similar notation in the following more general contexts.  Suppose $X,Y,Z$ are sets.  Given a set $F\subseteq {X\choose 3}$ and $x,y\in X$, we write $N_F(x)=\{yz\in X: xyz\in F\}$ and $N_F(xy)=\{z\in V: xyz\in F\}$.  With the notation we have defined, we could write neighborhoods in a $3$-graph $G$ as either $N_G(x)$, $N_G(xy)$ or $N_E(x)$, $N_E(xy)$, where $E$ is the edge set of $G$.  Similarly, for $F\subseteq X\times Y\times Z$ and $x\in X$, $y\in Y$, write 
$$
N_F(x)=\{(y,z)\in V: (x,y,z)\in F\}\text{ and }N_F(xy)=\{z\in Z: (x,y,z)\in F\}.
$$

\section{Regularity}\label{ss:regularity}\label{sec:reg}

This section contains background on regularity for $3$-graphs.  The necessary definitions and lemmas are almost identical to those in Section 3 of part 3 \cite{Terry.2024c}.  For this reason, we omit most of the exposition and refer the reader to Part 3 for more details.  

\subsection{Bigraphs,  Trigraphs, and Triads}

We begin by defining bigraphs.

\begin{definition}\label{def:bigraph}
A \emph{bigraph} is a tuple $G=(V_1,V_2;E)$ where $V_1,V_2$ are vertex sets and $E\subseteq V_1\times V_2$. 
\end{definition}

For a bigraph $G=(V_1,V_2; E)$, the \emph{vertex sets} of $G$ are the sets $V_1,V_2$, and the \emph{edge set} of $G$ is $E$, which we also denote by $E(G)$. We now define density of a bigraph.

\begin{definition}\label{def:dens2}
Given a bigraph $G=(V_1,V_2; E)$ and $X\subseteq V_1$ and $Y\subseteq V_2$, define 
$$
d_G(X,Y)=\frac{|E\cap (X\times Y)|}{|X||Y|}.
$$
The \emph{density of $G$} is $d_G(V_1,V_2)$.
\end{definition}

We now define a trigraph, which is a ternary analogue of a bigraph. 

\begin{definition}
An \emph{trigraph} is a tuple $(X,Y,Z; E)$ where $E\subseteq X\times Y\times Z$.
\end{definition}

For a triagraph $G=(V_1,V_2,V_3; E)$, the \emph{vertex sets} of $G$ are the sets $V_1,V_2,V_3$, and the \emph{edge set} of $G$ is the set $E$, which we also denote by $E(G)$. We now give notation for the trigraph which arises naturally from a $3$-graph.  

\begin{definition}\label{def:hbar}
Suppose $G=(V,E)$ is a $3$-graph.  Define $\overline{G}$ to be the trigraph $(V,V,V; \overline{E})$.
\end{definition}

Our next definition is a trigraph analogue of Definition \ref{def:dens2}.

\begin{definition}\label{def:dens3}
Given a trigraph $H=(V_1,V_2,V_3; E)$ and $X\subseteq V_1$, $Y\subseteq V_2$, and $Z\subseteq V_3$, define 
$$
d_H(X,Y,Z)=\frac{|E\cap (X\times Y\times Z)|}{|X||Y||Z|}.
$$
\end{definition}

We next define a \emph{triad}, which is the analgoue of a tripartite graph for bigraphs.

\begin{definition}\label{def:triad}
A \emph{triad} is a tuple $G=(X,Y,Z; E_{XY},E_{YZ}, E_{XZ})$ where $(X,Y;E_{XY})$, $(X,Z; E_{XZ})$, and $(Y,Z;E_{YZ})$ are bigraphs.
\end{definition}

For a triad $G=(X,Y,Z; E_{XY},E_{YZ}, E_{XZ})$, \emph{the set of ordered triangles in $G$} is 
$$
K_3(G):=\{(x,y,z)\in X\times Y\times Z: xy\in E_{XY}, yz\in E_{YZ},xz\in E_{XZ}\}.
$$
The \emph{component bigraphs of $G$} are
$$
G[X,Y]:=(X,Y;E_{XY}),\text{ }G[X,Z]:=(X,Z;E_{XZ}),\text{ and }G[Y,Z]:=(Y,Z;E_{YZ}).
$$ 

We next give the definition of a triad underlying a trigraph. 

\begin{definition}
Suppose $H=(X,Y,Z;R)$ is a trigraph and $G=(X,Y,Z; E_{XY},E_{YZ}, E_{XZ})$ is a triad. We say $G$ \emph{underlies $H$} if $R\subseteq K_3(G)$. 
\end{definition}

We now give notation for the restriction of a trigraph to the triangles of a triad.

\begin{definition}\label{def:1}
Suppose $H=(V_1,V_2,V_3;R)$ is a trigraph, $X_1\subseteq V_1$, $X_2\subseteq V_2$, $X_3\subseteq V_3$,  and $G$ is a triad with vertex sets $X_1,X_2,X_3$.  Define $H|G$ to be the following trigraph.
$$
H|G:=(X_1,X_2,X_3;R\cap K_3(G)).
$$
\end{definition}

We now define the density for a trigraph relative to a triad.
 
 \begin{definition}\label{def:dens4}
 Suppose $H=(V_1,V_2,V_3;R)$ is a trigraph, $X_1\subseteq V_1$, $X_2\subseteq V_2$, $X_3\subseteq V_3$,  and $G$ is a triad with vertex sets $X_1,X_2,X_3$.  Define
 $$
 d_H(G)=\frac{|R\cap K_3(G)|}{|K_3(G)|}.
 $$
 \end{definition}
 
We can extend this definition to $3$-graphs as follows.
 
  \begin{definition}\label{def:dens5}
 Suppose $H=(V,E)$ is a  $3$-graph, $X_1,X_2,X_3\subseteq V$, $X_2\subseteq V_2$, $X_3\subseteq V_3$,  and $G$ is a triad with vertex sets $X_1,X_2,X_3$.  Define $d_H(G)$ to be $d_{\overline{H}}(G)$ from Defintion \ref{def:dens3}.
 \end{definition}
 
Given a triad $G=(A,B,C; E_{AB},E_{BC},E_{AC})$ and $A'\subseteq A$, $B'\subseteq B$, and $C'\subseteq C$,  define 
$$
G[A',B',C']=(A',B',C'; E'_{AB},E'_{BC},E'_{AC}),
$$
 where $E_{AB}'=E_{AB}\cap (A\times B)$, $E_{AC}'=E_{AC}\cap (A\times C)$, and $E_{BC}'=E_{BC}\cap (B\times C)$.  Similarly, given a trigraph $H=(A,B,C;F)$, define
 $$
 H[A',B',C']=(A',B',C'; F\cap (A'\times B'\times C')).
 $$

\subsection{$\dev_{2,3}$-quasirandomness}

This section defines necessary notions of quasirandomness. We begin with the definition of a quasirandom bigraph. 

\begin{definition}\label{def:dev2}
Suppose $B=(U,W; E)$ is a bigraph and $|E|=d_B|U||W|$.  We say $B$ \emph{has $\dev_2(\e,d)$} if $d_B=d\pm \e$ and 
$$
\sum_{u_0,u_1\in U}\sum_{w_0,w_1\in W}\prod_{i\in \{0,1\}}\prod_{j\in \{0,1\}}g(u_i,v_j)\leq \e |U|^2|V|^2,
$$
where $g(u,v)=1-d_B$ if $uv\in E$ and $g(u,v)=-d_B$ if $uv\notin E$. 

We say $B$ simply \emph{has $\dev_2(\e)$} if it has $\dev_2(\e, d_B)$.
\end{definition}

We will use the following couting lemma, which was proved by Gowers  \cite{Gowers.20063gk}.

\begin{proposition}[Counting Lemma]\label{prop:counting}
Let $\e,d_{AB},d_{AC},d_{BC}>0$. Suppose we have a triad $G=(A,B,C; E_{AB},E_{AC}, E_{BC})$ such that $G[A,B]$, $G[B,C]$ and $G[A,C]$ have $\dev_2(\e, d_{AB})$, $\dev_2(\e, d_{BC})$, and $\dev_2(\e, d_{AC})$, respectively. Then 
$$
\Big| |K_3(G)|- d_{AB}d_{BC}d_{AC}|A||B||C||\Big|\leq 4\e^{1/4}|A||B||C|.
$$
\end{proposition}

Our next definition (also due to Gowers \cite{Gowers.20063gk}) is a notion of quasirandomness for a trigraph relative to an underlying triad. 

 \begin{definition}\label{def:regtriad}
Let $\e_1,\e_2>0$.  Assume $H=(X,Y,Z;E)$ is a trigraph, $G$ is a triad underlying $H$.  

We say that \emph{$(H,G)$ has $\dev_{2,3}(\e_1,\e_2)$} if there are $d_{XY},d_{YZ},d_{XZ}>0$ such that $G[X,Y]$, $G[X,Z]$, and $G[Y,Z]$ have $\dev_2(\e_2,d_{XY})$,  $\dev_2(\e_2,d_{XZ})$, and  $\dev_2(\e_2,d_{YZ})$ respectively, and
$$
\sum_{u_0,u_1\in X}\sum_{w_0,w_1\in Y}\sum_{z_0,z_1\in Z}\prod_{(i,j,k)\in \{0,1\}^3}h_{H,G}(u_i,w_j,z_k)\leq \e_1 d_{XY}^4d_{YZ}^4d_{XZ}^4|X|^2|Y|^2|Z|^2,
$$
where $h_{H,G}(x,y,z)=1-d_H(G)$ if $(x,y,z)\in E\cap K_3(G)$, $h_{H,G}(x,y,z)=-d_H(G)$ if $(x,y,z)\in K_3(G)\setminus E$, and $h_{H,G}(x,y,z)=0$ if $(x,y,z)\notin K_3(G)$.
\end{definition}

The following is a corollary of the counting Lemma due to Gowers (Theorem 6.8 of \cite{Gowers.20063gk}).  
  
\begin{corollary}\label{cor:countingcor}
For all $t\geq 1$, there are $D\geq 1$, and a polynomial  $p(x,y)$ so that for all $0<\e_1,d_1,d_2,\e_2<1$ satisfying  $\e_2<p(\e_1,d_2)$ and $\e_1<d_1^D$, there is $n_0$ so that the following holds. 

Let $F=([t],R_F)$ and $H=(V,R)$ be $3$-graphs. Suppose $V_1,\ldots, V_t$ are subsets of $V$, each of size at least $n_0$ and for each $1\leq i,j\leq t$, assume $G_{ij}=(V_i,V_j;E_{ij})$ is a bigraph with density $d_{ij}\geq d_2$.  For each $1\leq i,j,k\leq t$, let $G_{ijk}=(V_i,V_j,V_k; E_{ij},E_{jk},E_{ik})$, let $H^{ijk}=\overline{H}|G_{ijk}$, and let $d_{ijk}=d_{H^{ijk}}(G^{ijk})$.  Suppose the following hold.
\begin{enumerate}
\item For each $1\leq i,j\leq t$, $G_{ij}$ has $\dev_2(\e_2)$,
\item For each $ijk\in R_F$, $d_{ijk}\geq d_1$, and for each  $ijk\in {[t]\choose 3}\setminus R_F$, $d_{ijk}\leq 1-d_1$.
\item For each $1\leq i,j,k\leq t$, $(H^{ijk},G^{ijk})$ satisfies $\dev_{2,3}(\e_1,\e_2)$,
\end{enumerate}
Then there exists a tuples $(v_1,\ldots, v_t)\in \prod_{i=1}^tV_i$ such that $(v_i,v_j,v_k)\in R$ if and only if $ijk\in R_F$.
\end{corollary}

\subsection{Regular Decompositions}

This subsection contains the definitions of regular decompositions. We begin with the definition of a $(t,\ell)$-decomposition.

\begin{definition}\label{def:decomp}
Let $V$ be a vertex set and $t,\ell \in \mathbb{N}^{>0}$. A \emph{$(t,\ell)$-decomposition} $\calP$ for $V$ consists of a pair $(\calP_1,\calP_2)$ where  
$$
 \calP_1=\{V_1\cup \ldots \cup V_t\}\text{ and }\calP_2=\{P_{ij}^{\alpha}: 1\leq i,j\leq t, 1\leq \alpha\leq \ell \}
 $$
 so that $V=\bigcup_{i=1}^tV_i$ is a partition and so that for each $1\leq i, j\leq t$, $P^1_{ij}, \ldots, P^\ell_{ij}$ are disjoint sets satisfying $V_i\times V_j=P^1_{ij}\cup \ldots \cup P^\ell_{ij}$.  
\end{definition}

We note to the reader that some of the sets $P_{ij}^{\alpha}$ may be empty.  We allow this for notational convenience throughout the paper.  

A \emph{triad of $\calP$}, is a triad of the following form, for some  $1\leq i,j,k\leq t$, and $\alpha,\beta,\gamma\in [\ell]$.
$$
G^{ijk}_{\alpha\beta\gamma}:=(V_i,V_j,V_k; P_{ij}^\alpha,P_{ik}^\beta, P_{jk}^\gamma),
$$
We say $G^{ijk}_{\alpha\beta\gamma}$ is a \emph{non-empty} triad if $K_3(G^{ijk}_{\alpha\beta\gamma})\neq \emptyset$.  We denote by $\triads(\calP)$ the set of all non-empty triads of $\calP$.  Our next definition is that of a regular triad, relative to a $3$-graph. 

\begin{definition}\label{def:regtriads}
Given a $3$-graph $H=(V,R)$, a $(t,\ell)$-decomposition $\calP$ of $V$, and $G\in \triads(\calP)$, we say $G$ \emph{has $\dev_{2,3}(\e_1,\e_2)$ with respect to $H$} if $(\overline{H}|G,G)$ has $\dev_{2,3}(\e_1,\e_2)$. 
\end{definition}

 We now give another definition about decompositions.  

\begin{definition}
Suppose $\calP=(\calP_1,\calP_2)$ is a $(t,\ell)$-decomposition of $V$, where 
$$
 \calP_1=\{V_1\cup \ldots \cup V_t\}\text{ and }\calP_2=\{P_{ij}^{\alpha}: 1\leq i,j\leq t, 1\leq \alpha\leq \ell \}.
$$
We say $\calP$ is a \emph{$(t,\ell, \e_1,\e_2)$-decomposition of $V$} if the following holds, where $\Omega$ is the set of $P_{ij}^{\alpha}\in \calP_2$ such that $(V_i,V_j; P_{ij}^{\alpha})$ has $\dev_{2}(\e_2)$:
$$
\Big|\bigcup_{P\in \Omega}P\Big|\geq (1-\e_1)|V|^2.
$$
\end{definition}

We now define regular decompositions.

\begin{definition}\label{def:regdec}
Suppose $H=(V,E)$ is a $3$-graph and $\calP$ is an $(t,\ell, \e_1,\e_2)$-decomposition of $V$.  We say that $\calP$ is \emph{$\dev_{2,3}(\e_1,\e_2)$-regular} with respect to $H$ if the following holds where $\Sigma$ is the set of $G\in \triads(\calP)$ satisfying $\dev_{2,3}(\e_1,\e_2)$ with respect to $H$:
$$
\Big|\bigcup_{G\in \Sigma}K_3(G)\Big|\geq (1-\e_1)|V|^3.
$$
\end{definition}

 We can now state a version of  the regularity lemma for $\dev_{2,3}$-quasirandomness (see \cite{Gowers.20063gk, Nagle.2013, Frankl.2002}).

\begin{theorem}[Gowers \cite{Gowers.20063gk}]\label{thm:reg2} For all $\e_1>0$, and every function $\e_2:\mathbb{N}\rightarrow (0,1]$, there exist positive integers $T=T(\e_1,\e_2,t_0,\ell_0)$ and $L=L(\e_1,\e_2,t_0,\ell_0)$, such that for every sufficiently large $3$-graph $H=(V,E)$, there exists a  $\dev_{2,3}(\e_1,\e_2(\ell))$-regular, $(t,\ell,\e_1,\e_2(\ell))$-decomposition $\calP$ for $H$ with $1\leq t\leq T$ and $1\leq \ell \leq L$.
\end{theorem}

We will need the following definition of a non-trivial triads. 

\begin{definition}\label{def:nontrivialtriad}
Suppose $\calP$ is a $(t,\ell, \e_1,\e_2)$-decomposition of $V$.  We say a triad $G=(V_i,V_j,V_k; P_{ij}^{\alpha},P_{ik}^{\beta},P_{jk}^{\gamma})$ of $\calP$ is \emph{$\mu$-non-trivial} if the following hold.
\begin{enumerate}
\item $\min\{|V_i|,|V_j|,|V_k|\}\geq \mu |V|/t$,
\item $|P_{ij}^{\alpha}|\geq \mu|V_i||V_j|/\ell$, $|P_{ik}^{\beta}|\geq \mu|V_i||V_k|/\ell$, and $|P_{jk}^{\gamma}|\geq \mu|V_j||V_k|/\ell$.
\end{enumerate}
\end{definition}

Most triples come from non-trivial triads, as the next lemma tells us. 

\begin{lemma}\label{lem:nontrivialtriad}
Suppose $\calP$ is a $(t,\ell, \e_1,\e_2)$-decomposition of $V$, and $\Omega$ is the set of $\mu$-non-trivial triads of $\calP$.  Then $|\bigcup_{G\in \Omega}K_3(G)|\geq (1-2\mu)|V|^3$.
\end{lemma}

\section{Homogeneity and $\VC_2$-dimension}\label{sec:VC2}

This section contains background on homogeneous $(t,\ell)$-decompositions and $\VC_2$-dimension.  Much of this section is repetition of Section 4 in part 3 \cite{Terry.2024c}, so we will again omit much of the exposition.  We begin with the definition of a homogeneous decomposition, which is a decomposition where most triads have density close to $0$ or $1$.

\begin{definition}\label{def:homdec}
Suppose $H=(V,E)$ is a $3$-graph with $|V|=n$ and $\mu>0$. Suppose $t,\ell\geq 1$ and $\calP$ is a $(t,\ell)$-decomposition of $V$.  
\begin{enumerate}
\item Given a triad $G\in \triads(\calP)$, we say $G$ is \emph{$\mu$-homogeneous for $H$} if 
$$
d_H(G)\in [0,\e)\cup (1-\e,1].
$$ 
\item We say that $\calP$ is \emph{$\mu$-homogeneous with respect to $H$} if the following holds, where $\Sigma_{hom}$ is the set of $\mu$-homogeneous triads of $\calP$:
$$
\Big|\bigcup_{G\in \Sigma_{hom}}K_3(G)\Big|\geq (1-\mu)|V|^3.
$$ 
\end{enumerate}
\end{definition}

The following proposition shows that homogeneous decompositions are also regular (see Proposition 2.24 in \cite{Terry.2022}).

\begin{proposition}\label{prop:homimpliesrandome}
For all $0<\e<1/2$, $d_2>0$, and $0<\delta\leq (d_2/2)^{48}$, there is $N$ such that the following holds.  Suppose $H=(V_1,V_2,V_3;R)$ is a trigraph underlied  by a bigraph $G=(V_1,V_2,V_3;E_{12},E_{13},E_{23})$ so that for each $1\leq i<j\leq 3$, $G[V_i,V_j]$ has $\dev_2(\delta)$ and density at least $d_2$.  If $d_H(G)\in [0,\e)\cup (1-\e,1]$, then $(H,G)$ has $\dev_{2,3}(\delta,6\e)$.
\end{proposition}

We now turn to defining $\VC_2$-dimension, which can be used to characterize when a hereditary property admits homogeneous decompositions.

\begin{definition}\label{def:vc2}
Suppose $H=(V,E)$ is a $3$-graph.  The \emph{$\VC_2$-dimension of $H$}, $\VC_2(H)$, is the largest integer $k$ so that there exist vertices $a_1,\ldots, a_k,b_1,\ldots, b_k\in V$ and $c_S\in V$ for each $S\subseteq [k]^2$, such that $a_ib_jc_S\in E$ if and only if $(i,j)\in S$.
\end{definition}

We will use the following result, which tells us that regular triads in $3$-graphs of bounded $\VC_2$-dimension are homogeneous.  

\begin{proposition}\label{prop:suffvc2}
For all $k\geq 1$, there are $D\geq 1$ and a polynomial $p(x,y)$  so that for all integers $t,\ell\geq 1$, all $0<\e_1<2^{-D}$, all $\mu>0$, and all  $\e_2:\mathbb{N}\rightarrow (0,1]$ satisfying $\e_2(x)\leq p(\e_1,\mu x^{-1})$, the following holds.   Let $H=(V,E)$ be a sufficiently large $3$-graph with $\VC_2$-dimension less than $k$.  Suppose $\calP$ is a $\dev_{2,3}(\e_2(\ell),\e_1)$-regular $(t,\ell, \e_1,\e_2(\ell))$-decomposition of $V$. Every $\mu$-non-trivial $G\in \triads(\calP)$ satisfying $\dev_{2,3}(\e_2(\ell),\e_1)$ with respect to $H$  is $\e_1^{1/D}$-homogeneous with respect to $H$.
\end{proposition}

For more background and a proof, see part 3 \cite{Terry.2024c}. We will also use the following immediate corollary.

\begin{corollary}\label{cor:suffvc2}
For all $k\geq 1$, there is a polynomial $q(x,y)$ and $D\geq 3$  so that for all $0<\e_1<2^{-D}$ and $\e_2:\mathbb{N}\rightarrow (0,1]$ satisfying $\e_2(x)\leq q(\e_1, x^{-1})$  the following holds.  Let $t,\ell\geq 1$, and suppose $H=(V,E)$ is a sufficiently large $3$-graph with $\VC_2$-dimension less than $k$.  Suppose $\calP$ is a $\dev_{2,3}(\e_1,\e_2(\ell))$-regular $(t,\ell, \e_1,\e_2(\ell))$-decomposition of $V$.  Then $\calP$ is $\e_1^{1/D}$-homogeneous with respect to $H$ in the sense of Definition \ref{def:homdec}. 
\end{corollary}
\begin{proof}
Let $D_1$ and $p(x,y)$ be as in Proposition \ref{prop:suffvc2}.  Let $D=\max\{3,D_1\}$ and $q(x,y)=p(x,xy)$.  Fix $0<\e_1<2^{-D}$,  $\e_2:\mathbb{N}\rightarrow (0,1]$ satisfying $\e_2(x)\leq q(\e_1, x^{-1})$, $t,\ell \geq 1$, and $H=(V,E)$ a sufficiently large $3$-graph with $\VC_2(H)<k$.  Suppose $\calP$ is a $\dev_{2,3}(\e_1,\e_2(\ell))$-regular $(t,\ell, \e_1,\e_2(\ell))$-decomposition of $V$. By Proposition \ref{prop:suffvc2}, every $\e_1$-nontrivial triad of $\calP$ is $\e_1^{1/D}$ homogeneous.  By Lemma \ref{lem:nontrivialtriad}, at least $(1-\e_1)|V|^3$ triples of $V^3$ come from $\e_1$-nontrivial triads. Combining this with the fact that $\calP$ is $\dev_{2,3}(\e_1,\e_2(\ell))$-regular, we have that at least $(1-2\e_1)|V|^3$ triples from $V^3$ come from $\e_1^{1/D}$-homogeneous triads.  Since $\e_1<2^{-D}<1/4$,  $(1-2\e_1)|V|^3\geq (1-\e_1^{1/D})|V|^3$, so  $\calP$ is $\e_1^{1/D}$-homogeneous with respect to $H$ in the sense of Definition \ref{def:homdec}. 
\end{proof}

We now state the definition of the $\VC_2$-dimension of a hereditary property. 

\begin{definition}
Given a hereditary $3$-graph property, define $\VC_2(\calH)\in \mathbb{N}\cup \{\infty\}$ as follows.
$$
\VC_2(\calH)=\sup\{\VC_2(H): H\in \calH\}.
$$
\end{definition}

The following fact will be used later in the paper (a proof appears in \cite{Terry.2021b}).

\begin{observation}\label{ob:universal2}
Suppose $\calH$ is a hereditary $3$-graph property and $\VC_2(\calH)=\infty$.  The for all tripartite $3$-graphs $H=(A\cup B\cup C, E)$, there is a $3$-graph $H'\in \calH$ with vertex set $V(H')=A\cup B\cup C$ and edge set $E(H')$ satisfying $E(H')\cap K_3[A,B,C]=E$.  
\end{observation}

We now state a characterization of when a hereditary property admits homogeneous decompositions. This result was proved by Wolf and the author \cite{Terry.2021b}, and independently under a different formalism in \cite{Chernikov.2020}.

\begin{theorem}\label{thm:vc2hom}
Suppose $\calH$ is a hereditary $3$-graph property.  The following are equivalent.
\begin{enumerate}
\item $\calH$ has finite $\VC_2$-dimension,
\item $\calH$ admits homogeneous decompositions in the following sense: for all $\e_1>0$ and $\e_2:\mathbb{N}\rightarrow (0,1)$, there are $T, L\geq 1$ so that all sufficiently large $H\in \calH$, there exists $1\leq t\leq T$ and $1\leq \ell\leq L$, and an $\e_1$-homogeneous $(t,\ell,\e_1,\e_2(\ell))$-decomposition of $H$.  
\end{enumerate}
\end{theorem}

\section{Tools}\label{sec:Ltools}

This section contains the tools needed to prove Theorem \ref{thm:mainL2}.  In Subsection \ref{ss:std}, we state several standard lemmas about regularity, which also appear in Subsection 5.1 of Part 3.  Subsections \ref{ss:irr}-\ref{ss:refine} contain tools which are specific to this paper.  In particular Subsection \ref{ss:irr} defines irreducible bigraphs, and Subsection \ref{ss:ecb} defines edge-colored bigraphs and states a structure theorem in that setting.   Subsection \ref{ss:corner} defines corner graphs and encodings (machinery first developed in \cite{Terry.2021b} and \cite{Terry.2022}). Subsection \ref{ss:blowups} defines a type of blow-up for $3$-graphs, and Subsection \ref{ss:blowupcorner} shows encodings can be used to find blowups.  Finally, Subsection \ref{ss:refine} proves the existence of somewhat equitable refinements of decompositions as a corollary of a result of Frankl and R\"{o}dl.

\subsection{Lemmas}\label{ss:std}

This subsection contains lemmas we require about regularity of various kinds.  Almost all these statements also appear in Subsection 5.1 of Part 3.  For this reason, we merely state the necessary results here and refer the reader to Part 3 for more discussion.    

\begin{lemma}\label{lem:averaging}
Let $a,b,\e\in (0,1)$ satisfy $ab=\e$. Suppose $A\subseteq X$ and $|A|\geq (1-\e)|X|$.  For any partition $\calP$ of $X$, if we let $\Sigma=\{Y\in \calP: |A\cap Y|\geq (1-a)|Y|\}$, then $|\bigcup_{Y\in \Sigma}Y|\geq (1-b)|X|$.
\end{lemma}

\begin{proposition}[Sub-pairs lemma]\label{lem:sldev} Suppose $G=(A, B; E)$ is a bigraph with density $d$. Suppose $A'\subseteq A$ and $B'\subseteq B$ satisfy $ |A'| \geq \gamma |A|$ and $|B'| \geq \gamma|B|$ for some $\gamma \geq \e$, and $G$ satisfies $\dev_2(\epsilon, d)$.  Then $G':=(A', B'; E\cap (A'\times B'))$ satisfies $\dev_2(\epsilon',d)$ where $\epsilon'=2\gamma^{-1}\e^{1/12}$.
\end{proposition}

\begin{fact}[Unioning Edges]\label{fact:adding}
Suppose $E_1$ and $E_2$ are disjoint subsets of $U\times V$.  Assume $(U,V;E_1)$ has $\dev_2(\e_1,d_1)$, and $(U, V;E_2)$ has $\dev_2(\e_2,d_2)$, then $(U,V;E_1\cup E_2)$ has $\dev_2(\e_1^{1/12}+\e_2^{1/12}, d_1+d_2)$.  
\end{fact}

 \begin{lemma}[Bigraph compliments]\label{lem:complimentbi}
Suppose $B=(X,Y; E)$ is a bigraph satisfying $\dev_{2}(\e)$.   Let $B'=(X,Y;E')$ where $E'=(X\times Y)\setminus E$.  Then $B'$ also satisfies $\dev_{2}(\e)$.  
\end{lemma}

\begin{lemma}\label{lem:standreg}
Let $d>0$ and $0<\e<\min\{d^{800}, 2^{-800}\}$. Suppose $G$ is a triad with vertex sets $U,V,W$, whose component bigraphs have $\dev_2(\e)$ and densities $d_{UV},d_{UW},d_{VW}\geq d$.  Let 
$$
Y=\{uv\in E_{UV}: |N_{E_{UW}}(u)\cap N_{E_{VW}}(v)|=(1\pm \e^{1/100})d_{UW}d_{VW}|W|\}.
$$
Then $|Y|\geq (1- \e^{1/100})|E_{UV}|$.
\end{lemma}

We next state a corollary of the slicing lemma for regular triads (for more discussion, see the end of Section 5.2 in Part 3).  In particular, Corollary \ref{cor:sldev23} below tells us we can refine the vertex partition of a regular decomposition without losing much regularity. 

\begin{corollary}\label{cor:sldev23}
For all $C\geq 1$, there exists $K$ and $q(x,y)$ so that the following hold. Assume $0<\e_1$ is sufficiently small, $\ell,t\geq 1$, and $\e_2:\mathbb{N}\rightarrow (0,1)$ satisfies $\e_2(\ell)<q(\e_1,1/\ell)$. 

Suppose $H=(V,E)$ is a $3$-graph and $\calP$ is an $\dev_{2,3}(\e_1,\e_2(\ell))$-regular $(t,\ell,\e_1,\e_2(\ell))$-decomposition for $H$. Let $\calP'$ be a $(t',\ell)$-decomposition with the following properties.
\begin{enumerate}
\item $\calP_1'\preceq \calP_1$, and each set in $\calP_1$ is refined into at most $C$ parts in $\calP_1'$, 
\item For every $P'\in \calP_2'$, $P'=P\cap (X\times Y)$ for some $P\in \calP_2$ and $X,Y\in \calP_1'$. 
\end{enumerate}
Then $\calP'$ is a $\dev_{2,3}(\e_1',\e_2'(\ell))$-regular $(t',\ell, \e_1',\e_2'(\ell))$-decomposition for $H$, where 
$$
\e_1'=4\e_1^{1/2K^2}C^{K}\text{ and }\e_2'(\ell)=2C\e_1^{-1/2K^2}\e_2(\ell)^{1/12}.
$$ 
\end{corollary}

For reasons of convenience, we use the following notion of $\e$-regularity for bigraphs. 

\begin{definition}
Suppose $G=(A, B; E)$ is a bigraph.  We say it $G$ is \emph{$\e$-regular} if for all $A'\subseteq A$ and $B'\subseteq B$ with $|A'|\geq \e|A|$ and $|B'|\geq \e|B|$, $|d_G(A,B)-d_G(A',B')|\leq \e$. 
\end{definition}

The following is immediate from Lemma 3.8 of \cite{Frankl.2002}.

 \begin{lemma}\label{lem:3.8easy}
For all $\ell\geq 1$ and all sufficiently small $\e>0$, there is $n_0$ so that the following holds.  Suppose $A,B$ are sets with $|A|=|B|\geq n_0$.  Then there exists a partition $A\times B=\bigcup_{i=1}^{\ell}P_i$ so that for each $i\in [\ell]$, $(A, B;P_i)$ is $\e$-regular with density $\ell^{-1} \pm \e$. 
\end{lemma}

We will also use the following translation between regularity and $\dev_2$ (see \cite{Gowers.20063gk}).

\begin{theorem}\label{thm:equiv}
Suppose $B=(U,W; E)$ is a bigraph and $|E|=d|U||W|$.  
\begin{enumerate}
\item If $B$ is $\e$-regular then $(U,W;\overline{E})$ has $\dev_2(\e,d)$.
\item If $(U,W;\overline{E})$ has $\dev_2(\e,d)$, then $(U, W;E)$ is $\e^{1/12}$-regular. 
\end{enumerate}
\end{theorem}

\subsection{Irreducible Bigraphs}\label{ss:irr}

This section contains definitions and results about irreducible bigraphs. The main goal is to show that large irreducible bipartite graphs must contain one of a short list of special induced sub-bigraphs.  The tools here are adapted from those used to prove Theorem \ref{thm:alljump} in Part 2 \cite{Terry.2024b}.  We use bigraphs in this section due to an inherent asymmetry in later applications, which is easier to keep track of using bigraphs (as opposed to bipartite graphs, as were used for similar purposes in \cite{Terry.2022}). We begin by defining two relations on a bigraph. This is a bigraph analogue of Definition 3.22 from Part 2 \cite{Terry.2024b}.  

\begin{definition}
Suppose $G=(U,V;E)$ is a bigraph.  Given $x,y\in U$, define $x\sim_{G,U} y$ if for for all $z\in V$, $(x,z)\in E$ if and only if $(y,z)\in E$. Similarly, given $x,y\in V$, define $x\sim_{G,V} y$ if for for all $z\in U$, $(z,x)\in E$ if and only if $(z,y)\in E$.
\end{definition}

It is clear that for any bigraph $G$, $\sim_{G,U}$ is an equivalence relation on $U$ and $\sim_{G,V}$ is an equivalence relation on $V$.  We now define the notion of an irreducible bigraph.  This is analogous to Definition 3.23 from Part 2 \cite{Terry.2024b}.

\begin{definition}\label{def:irr}
Suppose $G=(U,V;E)$ is a bigraph. We say $G$ is \emph{irreducible} if every $\sim_{G,U}$-class and every $\sim_{G,V}$-class has size $1$.  
\end{definition}

We now give notation for certain special bigraphs. Similar objects were used in Part 2, and we have chosen to use the same notation here, although the definitions are formally distinct (see Definition 3.25 in \cite{Terry.2024b}).  

\begin{definition}\label{def:hk}
Given $k\geq 1$, define
\begin{align*}
H(k)&=(\{a_1,\ldots, a_k\},\{b_1,\ldots, b_k\}; \{(a_i,b_j): 1\leq i\leq j\leq k\}),\\
M(k)&=(\{a_1,\ldots, a_k\},\{b_1,\ldots, b_k\}; \{(a_i,b_i): i\in [k]\}), \text{ and }\\
\overline{M}(k)&=(\{a_1,\ldots, a_k\},\{b_1,\ldots, b_k\}; \{(a_i,b_j): 1\leq i\neq j\leq k\}).
\end{align*}
\end{definition}

Observe that for any $k\geq 1$, each of $H(k),M(k),\overline{M}(k)$ are irreducible bigraphs.  Moreover, these bigraphs are canonical in the sense that any sufficiently large irreducible bigraph must contain copies of one of them.  To make this precise, we require the following definition.

\begin{definition}\label{def:irr}
Suppose $G=(U, V;E)$ and $H=(A,B;F)$ are bigraphs. We say $G$ contains an \emph{induced copy of $H$} if there are $\{u_a: a\in A\}\subseteq U$ and $\{v_b: b\in B\}\subseteq B$ so that $(u_a,v_b)\in E$ if and only if $(a,b)\in F$.  
\end{definition}

We now come to the main statement of this subsection, Lemma \ref{lem:prime} below.  While we note that there are some notational differences, it is not difficult to see Lemma \ref{lem:prime} is an immediate corollary of Lemma 3.27 in Part 2 \cite{Terry.2024b}.  Lemma 3.27 of \cite{Terry.2024b} in turn is based on related results in the literature, namely \cite{prime}.

\begin{lemma}\label{lem:prime}
For all $k\geq 1$ there is $N$ so that the following holds.  Suppose $G=(U, V;E)$ is an irreducible bigraph with $\min\{|V|,|U|\}\geq N$. Then $G$ contains an induced copy of $H(k)$, $M(k)$, or $\overline{M}(k)$. 
\end{lemma}

\subsection{Edge-Colored Bigraphs and Trigraphs}\label{ss:ecb}

In this section we discuss edge-colored bigraphs.  These objects will arise from certain ``reduced" structures associated to regular decompositions of $3$-graphs.  We begin by defining what we mean by an edge-colored bigraph.

\begin{definition}\label{def:ecg}
A \emph{edge-colored bigraph} is a tuple $(A, B; E_0,\ldots, E_r)$, where $r\geq 1$ and $A\times B=\bigsqcup_{i=0}^rE_i$. 
\end{definition}

Similar objects play an important role in the proof of the main theorem in \cite{Terry.2022}.  More specifically, in \cite{Terry.2022} edge-colored bipartite graphs are used.  In this paper, we choose to work with edge-colored bigraphs over edge-colored bipartite graphs because bigraphs allow us to more easily keep track of asymmetry in applications.  

We now give a short sketch of the proof of the main result in \cite{Terry.2022}, which showed that $3$-graphs with small $\VC_2$-dimension admit $\dev_{2,3}$-regular $(t,\ell)$-decompositions with a relatively small $\ell$.  That proof begins with a $3$-graph $H$ of small $\VC_2$-dimension, then builds auxiliary edge-colored bigraphs from a regular decomposition $\calP$ of $H$.  In this context, the edge-colored bigraphs come with three colors, a ``dense" color $E_1$, a ``sparse" color $E_0$, and an ``irregular" color $E_2$.  It was shown in \cite{Terry.2022} that because $H$ has small $\VC_2$-dimension, these auxiliary bigraphs omit certain configurations relative to the edge colors $E_0$ and $E_1$, which in turn implies a structure theorem.  This was then used to build a more efficient regular decomposition of the original $3$-graph $H$.  We will use the same strategy in our upper bound proofs for $L_{\calH}$.    To begin making this precise, we define what it means for an edge-colored bigraph to contain a copy of a fixed bigraph in the ``dense" and ``sparse" edge colors.

\begin{definition}
Suppose $G=(U,V; E_0,E_1,E_2)$ is an edge-colored bigraph and $H=(A, B; E)$ is a bigraph.  We say $G$ \emph{contains an $E_0/E_1$-copy of $H$} if there exist vertices $\{u_a: a\in A\}\subseteq U$ and $\{v_b: b\in B\}\subseteq V$ so that $(u_a,v_b)\in E_1$ if $(a,b)\in E$ and $(u_a,v_b)\in E_0$ if $(a,b)\notin E$.
\end{definition}

In our applications, we want to keep track of $E_0/E_1$-copies of certain special bigraphs. One of these is the following bigraph version of the powerset graph.

\begin{definition}\label{def:ukbg}
Given $k\geq 1$, let $U_{bg}(k)$ be the bigraph $(B_k,A_k; E_k)$ where 
$$
B_k=\{b_S:S\subseteq [k]\},\text{ }A_k=\{a_i: i\in [k]\},\text{ and }E_k=\{(b_S, a_i): i\in S\}.
$$
\end{definition}

Observe that $U_{bg}(k)$ is an irreducible bigraph for all $k\geq 1$.  The following lemma was proved in part 1 \cite{Terry.2022}. It can also be deduced, with slightly different bounds, directly from theorems by Alon, Fischer, and Newman \cite{Alon.2007}.

\begin{lemma}[Lemma 2.14 in \cite{Terry.2022}]\label{lem:hausslercor}
For all $k\geq 1$ there is a constant $c=c(k)$ so that the following holds.  Suppose $\delta,\e>0$ satisfy $\e\leq c^{-2}(\delta/8)^{2k+2}$.    Assume $G=(U,V; E_0,E_1,E_2)$ is an edge-colored bigraph such that there is no $E_0/E_1$-copy of $U_{bg}(k)$ in $G$, and such that $|E_2|\leq \e |U||V|$.  

Then there is an integer $m\leq 2c(\delta/8)^{-k}$,  a subset $U_0\subseteq U$ with $|U_0|\leq \sqrt{\e} |U|$, and  vertices $x_1,\ldots, x_m\in U\setminus U_0$ so that for all $u\in U\setminus U_0$, $|N_{E_2}(u)|\leq \sqrt{\e}|V|$ and for some $1\leq i\leq m$, $\max\{|N_{E_1}(u)\Delta N_{E_1}(x_i)|,|N_{E_0}(u)\Delta N_{E_0}(x_i)|\}\leq \delta |V|$. 
\end{lemma}

We will later use Lemma \ref{lem:hausslercor} in our upper bound proof for the polynomial range of Theorem \ref{thm:mainL2} (see Section \ref{ss:UBL}). We will use it here in the proof of the first main result of this subsection, which is a stronger version of Lemma \ref{lem:hausslercor} in the case where $G$ omits all irreducible bigraphs up to a certain size.  To ease notation, we give a name to the collection of irreducible bigraphs with first vertex set of a certain size.

\begin{definition}\label{def:irrn}
Given an integer $C\geq 1$, let $\textrm{Irr}(C)$ denote the class of all irreducible bigraphs $G=(U,V;E)$ with $|U|= C$.
\end{definition}

We now prove the desired structure theorem of this section, which we deduce from Lemma \ref{lem:hausslercor} above.  

\begin{corollary}\label{cor:haussler}
Let $C\geq 1$ be an integer, and let $0<\e,\delta<1$ satisfy $\e\leq c^{-2}(\delta/8)^{2C+4}$, where $c=c(C)$ is from Lemma \ref{lem:hausslercor}. Assume $G=(U,V;E_0,E_1,E_2)$ is an edge-colored bigraph such that $|E_2|\leq \e |U||V|$, and such that for for all $R\in Irr(C)$, $G$ contains no $E_0/E_1$-copy of $R$.

Then there is an integer $m<C$, vertices $x_1,\ldots, x_m\in U$, and a set $U_0\subseteq U$ with $|U_0|\leq \sqrt{\e} |U|$,  so that for all $u\in U\setminus U_0$, $|N_{E_2}(u)|\leq \sqrt{\e}|V|$ and there is some $1\leq i\leq m$ so that $\max\{|N_{E_1}(u)\Delta N_{E_1}(x_i)|,|N_{E_0}(u)\Delta N_{E_0}(x_i)|\}\leq \delta |V|$. 
\end{corollary}
\begin{proof}
Let $c=c(C)$ be as in Lemma \ref{lem:hausslercor}, and assume $0<\e,\delta<1$ satisfy $\e\leq c^{-2}(\delta/8)^{2C+4}$.  It is a standard exercise to see that because $G$ contains no $E_0/E_1$-copy of any $R\in \textrm{Irr}(C)$, it also contains no $E_0/E_1$-copy of $U_{bg}(C)$.  By Lemma \ref{lem:hausslercor} there exist $m\leq 2c(\delta/8)^{-C}$, a subset $U_0\subseteq U$ with $|U_0|\leq \sqrt{\e} |U|$, and $x_1,\ldots, x_m\in U$  so that for all $u\in U\setminus U_0$, $|N_{E_2}(u)|\leq \sqrt{\e}|V|$ and there is some $1\leq i\leq m$ so that $\max\{|N_{E_1}(u)\Delta N_{E_1}(x_i)|,|N_{E_0}(u)\Delta N_{E_0}(x_i)|\}\leq \delta |V|$. 
 Assume $m$ is minimal so that such a collection $x_1,\ldots, x_m$ exists.  Let $V_0=\bigcup_{i=1}^mN_{E_2}(x_i)$.  Note $|V_0|\leq m\sqrt{\e}|V|<\delta^2|V|$, where the inequality is by the bound on $\e$. By the minimality of $m$, we know that for each $1\leq i\neq j\leq m$, 
$$
|(N_{E_1}(x_i)\Delta N_{E_1}(x_j))\cap (V\setminus V_0)|> \delta |V|-|V_0|>(\delta-\delta^2)|V|>0,
$$
where the second to last inequality is because $|V_0|<\delta^2|V|$.  Thus, for each $1\leq i\neq j\leq m$, there is some $v_{ij}\in (N_{E_1}(x_i)\Delta N_{E_1}(x_j))\cap (V\setminus V_0)$.  Observe now that the vertices $\{x_1,\ldots, x_m\}\subseteq U$ and $\{v_{ij}: 1\leq i\neq j\leq m\}\subseteq V$ form an $E_0/E_1$-copy in $G$ of some element of $\textrm{Irr}(m)$.  By assumption, we must therefore have $m<C$.
\end{proof}

We now turn to a $3$-ary analogue of Definition \ref{def:ecg}.

\begin{definition}\label{def:ecg2}
A \emph{edge-colored trigraph} is a tuple $(A, B, C; E_0,\ldots, E_r)$, where $r\geq 1$ and $A\times B\times C=\bigsqcup_{i=0}^rE_i$. 
\end{definition}

Our next goal is to prove an ``anti-symmetry" lemma for edge-colored trigraphs (Lemma \ref{lem:symmetry} below).  We begin with a symmetry lemma for bipartite graphs from \cite{Terry.2021a} (see Lemma 5.9 there).  We state it here in slightly different form, but note it follows from an identical proof.  In particular, we have replaced the assumption that the sets in question are sufficiently large with an assumption the $\e$ used is sufficiently small. 

\begin{lemma}[Lemma 5.9 \cite{Terry.2021a}]\label{lem:sym1}
For all sufficiently small $\e>0$, the following holds. Suppose $G=(U\cup W, E)$ is a bipartite graph $U$, $W$ are non-empty, and for at least $(1-\e)|U|$ many $u\in U$, $\max\{|N(u)|, |W\setminus N(u)|\}\geq (1-\e)|W|$, and for at least $(1-\e)|W|$ many $w\in W$, $\max\{|N(w)|, |U\setminus N(w)|\}\geq (1-\e)|U|$. Then $\frac{|E|}{|U||W|}\in [0,2\e^{1/2})\cup (1-2\e^{1/2},1]$.
\end{lemma}

Plugging this into the proof of Lemma 2.5 in \cite{Terry.2024a} yields the following slight alteration of Lemma 2.5 of \cite{Terry.2024a}, where again a lower bound assumption on the sizes of the sets involved is replaced by the assumption that $\e$ is sufficiently small.  We give this statement here, translating the terminology used there of ``almost good'' sets.

\begin{lemma}[Lemma 2.5 \cite{Terry.2024a}]\label{lem:sym2}
Suppose $\e>0$ is sufficiently small and $V_1,V_2,V_3$ are nonempty sets.  Let $H=(V_1\cup V_2\cup V_3, E)$ be a tripartite $3$-graph, and assume that for each $(i,j,k)$ satisfying $\{i,j,k\}=\{1,2,3\}$, the following set has size at least $(1-\e)|V_i|$:
$$
\Gamma_i=\{(x,y)\in V_j\times V_k: \min\{|N(xy)|, |V_i\setminus N(xy)|\}\geq \e^{1/4}|V_i|\}.
$$
Then $\frac{|E|}{|V_1||V_2||V_3|}\in [0,4\e^{1/16})\cup (1-4\e^{1/16},1]$.
\end{lemma}

We now deduce our desired anti-symmetry lemma, which follows from the contrapositive of Lemma \ref{lem:sym2}.  

\begin{lemma}\label{lem:symmetry}
Suppose $0<\e<1$ is sufficiently small, and $V_1,V_2,V_3$ are nonempty sets.  Assume $(V_1,V_2,V_3; E_0,E_1,E_2)$ is an edge-colored trigraph satisfying $|E_2|\leq \e |V_1||V_2||V_3|$ and 
$$
 \min\{E_0,E_1\}\geq 4\e^{1/64} |V_1||V_2||V_3|.
 $$
Then one of the following holds.
\begin{itemize}
\item there exists $x\in V_1$, $y\in V_2$, and $z,z'\in V_3$ so that $(x,y,z)\in E_0$ and $(x,y,z')\in E_1$, 
\item there exists $x\in V_1$, $y,y'\in V_2$, and $z\in V_3$ so that $(x,y,z)\in E_0$ and $(x,y',z)\in E_1$, 
\item there exists $x,x'\in V_1$, $y\in V_2$, and $z\in V_3$ so that $(x,y,z)\in E_0$ and $(x',y,z)\in E_1$.
\end{itemize}
\end{lemma}
\begin{proof}
Fix $\e>0$ sufficiently small and nonempty sets $V_1,V_2,V_3$. Clearly the general result follows from the one where we assume $V_1,V_2,V_3$ are pairwise disjoint.  Thus, without loss of generality, assume $V_1,V_2,V_3$ are pairwise disjoint.  Let   $(V_1,V_2,V_3; E_0,E_1,E_2)$ be an edge-colored trigraph such that $|E_2|\leq \e |V_1||V_2||V_3|$  and $ \min\{E_0,E_1\}\geq 4\e^{1/64} |V_1||V_2||V_3|$.  

Consider the $3$-partite $3$-graph $H$ with edge set $V_1\cup V_2\cup V_3$ and edge set 
$$
E(H)=\{xyz: (x,y,z)\in E_1\}.
$$
 Our assumed lower bound on $|E_1|$ and $|E_0|$ imply $d_H(V_1,V_2,V_3)\in (4\e^{1/64}, 1-4\e^{1/64})$.  By Lemma \ref{lem:sym2}, there is some $(i,j,k)$ satisfying $\{i,j,k\}=\{1,2,3\}$ such that the following set has size at least $\e^{1/4}|V_j||V_k|$.
$$
\Gamma=\{(x,y)\in V_j\times V_k: \min \{|N_H(xy)|, |V_i\setminus N_H(xy)|\}\geq \e^{1/16} |V_i|\}.
$$
Let $\Sigma=\{(x,y)\in V_j\times V_k: |N_{E_2}(x,y)|\leq \sqrt{\e}|V_j||V_k|\}$.  Since $|E_2|\leq \e |V_1||V_2||V_3|$, we have $|\Sigma|\geq (1-\sqrt{\e})|V_j||V_k|$.  Consequently, 
$$
|\Gamma\cap \Sigma|\geq (1-\e^{1/4}-\sqrt{\e})|V_j||V_k|>0,
$$
where the inequality is because $\e$ is sufficiently small.  Fix $(x,y)\in \Gamma\cap \Sigma$.  Then by definition of $\Gamma$, $\Sigma$ and $H$, we have
$$
|N_{E_1}(x,y)|\geq |N_H(x,y)|\geq \e^{1/4}|V_i|>0,
$$
and 
$$
|N_{E_0}(x,y)|\geq |V_i\setminus N_H(x,y)|-|N_{E_2}(x,y)|\geq \e^{1/4}|V_i|-\e^{1/2}|V_i|>0.
$$
Taking $z\in N_{E_1}(x,y)$ and $z'\in N_{E_0}(x,y)$ finishes the proof.  
\end{proof}

\subsection{Corner Graphs and Encodings}\label{ss:corner}
In this section, we define auxiliary edge-colored bigraphs built from regular decompositions of $3$-graphs.  These will play a crucial role in our upper bound proofs for $L_{\calH}$. The  definitions are adapted from \cite{Terry.2021b, Terry.2022}.  Our first definition will give us the vertex sets of these edge-colored bigraphs.

\begin{definition}\label{def:corner}
Suppose $\e >0$, $\ell, t\geq 1$, $V$ is a set, and $\calP$ is a $(t,\ell)$-decomposition for $V$ consisting of $\calP_1=\{V_i: i\in [t]\}$ and $\calP_2=\{P_{ij}^{\alpha}: 1\leq i,j\leq t, 1\leq \alpha\leq \ell\}$.    Define
\begin{align*}
\calP_{cnr}(\e)&=\{P_{ij}^\alpha P_{ik}^{\beta}: 1\leq i,j,k\leq t, 1\leq  \alpha,\beta\leq \ell, \text{ and }P_{ij}^\alpha, P_{ij}^{\beta}\text{ satisfy }\dev_2(\e)\}\\
\calP_{edge}(\e)&=\{P_{ij}^\alpha \in \calP_2: P_{ij}^\alpha\text{ satisfies }\dev_2(\e)\}.
\end{align*}
\end{definition}

The notation cnr stands for ``corner.''  Our next definition will give us the edge sets.

\begin{definition}\label{def:reducedP}
Suppose $\e,\mu>0$, $\ell,t\geq 1$, $H=(V,E)$ is a $3$-graph, and $\calP$ is a $(t,\ell)$-decomposition for $V$. Define 
\begin{align*}
\mathbf{E}_0(\e,\mu)&=\{(P_{ij}^{\alpha},P_{jk}^{\beta}P_{ik}^{\gamma})\in (\calP_{edge}(\e)\times \calP_{cnr}(\e)): |E\cap K_3(G_{ijk}^{\alpha\beta\gamma})|\leq \mu|K_3(G_{ijk}^{\alpha\beta\gamma})|\},\\
\mathbf{E}_1(\e,\mu)&=\{(P_{ij}^{\alpha},P_{jk}^{\beta}P_{ik}^{\gamma})\in  (\calP_{edge}(\e)\times \calP_{cnr}(\e)): |E\cap K_3(G_{ijk}^{\alpha\beta\gamma})|\geq (1-\mu)|K_3(G_{ijk}^{\alpha\beta\gamma})|\},\text{ and }\\
\mathbf{E}_2(\e,\mu)&=(\calP_{edge}(\e)\times \calP_{cnr}(\e))\setminus(\mathbf{E}_0(\e,\mu)\cup \mathbf{E}_1(\e,\mu)).
\end{align*}
\end{definition}

Definitions \ref{def:corner} and \ref{def:reducedP} yield an edge-colored bigraph with vertex sets $\calP_{edge}(\e), \calP_{cnr}(\e)$ and edge sets given by $\mathbf{E}_0(\e,\mu), \mathbf{E}_1(\e,\mu),\mathbf{E}_2(\e,\mu)$.   We now give the definition of an encoding, adapted from \cite{Terry.2021b,Terry.2022}.

\begin{definition}\label{def:encodingR}
Let $\e, \mu>0$ and $t,\ell\geq 1$.  Suppose $R=(U,V;E_R)$ is a bigraph,  $H$ is a $3$-graph, and $\calP$ is a $(t,\ell)$-decomposition of $V$.  An \emph{$(\e,\mu)$-encoding of $R$ in $(H,\calP)$} consists of a pair of functions $(g,f)$, where $f:U\rightarrow \calP_{edge}$ and $g:V\rightarrow \calP_{cnr}$ are such that the following hold for some $j_0k_0\in {[t]\choose 2}$.
 \begin{enumerate}
 \item $\mathrm{Im}(f)\subseteq \{P_{j_0k_0}^{\alpha}: 1\leq \alpha\leq \ell\}$, and $\textrm{Im}(g)\subseteq \{P^{\beta}_{ij_0}P^\gamma_{ik_0}: i\in [t], 1\leq \beta,\gamma\leq \ell\}$, and
 \item For all $u\in U$ and $v\in V$, if $(u,v)\in E_R$, then $(f(u),g(v))\in \mathbf{E}_1(\e,\mu)$, and if $(u,v)\notin E_R$, then $(f(u),g(v))\in \mathbf{E}_0(\e,\mu)$.
 \end{enumerate}
 Given $\delta>0$, we say that moreover the encoding is \emph{$\delta$-non-trivial} if for each $P^{\alpha}_{j_0k_0}\in Im(f)$ we have $|V_{j_0}|\geq \delta |V(H)|/t$, $|V_{k_0}|\geq \delta |V(H)|/t$, and $|P_{j_0k_0}^{\alpha}|\geq \delta |V_{j_0}||V_{k_0}|/\ell$, and for each $P^{\beta}_{ij_0}P^\gamma_{ik_0}\in Im(g)$, we have $|V_i|\geq \delta |V(H)|/t$ and $|P_{ij_0}^{\beta}|\geq \delta |V_{i}||V_{j_0}|/\ell$, $|P_{ik_0}^{\gamma}|\geq \delta |V_{i}||V_{k_0}|/\ell$.
 \end{definition}
 
Note that in Definition \ref{def:encodingR},  the two vertex sets $U$ and $V$ of have distinct roles, with one being mapped to ``edges," and the other being mapped to ``corners."  An encoding is a tool for finding a certain kind of ``blow up" in a $3$-graph.  We will define such blowups and discuss their connection to encodings in the next subsection.

 \subsection{Blow-Ups and $G$-dimension}\label{ss:blowups}

This section contains preliminaries related to a notion of ``blow up" which will be important in our proofs, a  corresponding notion of  $G$-dimension for a bigraph $G$.  The ideas are related to those used in Part 2, which deals with much simpler notions of blow ups.  In this subsection, we  work with edge-colored bigraphs, where the edge colors are indexed by an arbitrary set.  

\begin{definition}\label{def:ecggeneral}
Suppose $U$ is a set.  A \emph{ $U$-colored bigraph} is a tuple 
$$
(A,B; (P_u)_{u\in U}),
$$
 where $A\times B=\bigsqcup_{u\in U}P_u$. 
\end{definition}

We note that Definition \ref{def:ecggeneral} and Definition \ref{def:ecg} are describing similar kinds of objects (in particular, an edge-colored bigraph is also an $[r]$-colored bigraph for some integer $r\geq 1$).  We choose to have two different names for these, as they will be used differently in our proofs.  

Our next goal is to define  the desired notion of blow up.  Roughly speaking, a ``blow up'' of a bigraph $G=(A,B; E)$ will consist of a $3$-graph obtained by ``blowing up" the vertices of $A$ into  graphs relations.  This type of construction also appears in \cite{Terry.2022, Terry.2021b}.

\begin{definition}\label{def:blowup2}
Let $n\geq 1$ be an integer.  Suppose $G=(U,V; E)$ is a bigraph and $\Gamma=(A, B; (P_u)_{u\in U})$ is a $U$-colored bigraph.  An \emph{$(n,\Gamma)$-blowup of $G$} is a $3$-graph $H$ with vertex set $A\cup B\cup C$, where $C$ is a new set of vertices of the form $\bigcup_{v\in V}C_v$, where for each $v\in V$, $|C_v|=n$, such that $E(H)$ satisfies the following.
\begin{align*}
&\bigcup_{(u,v)\in E}\{xyz: (x,y)\in P_u, z\in C_v\}\subseteq E(H) \text{ and }\\
&\Big(\bigcup_{(u,v)\in (U\times V)\setminus E}\{xyz: (x,y)\in P_u, z\in C_v\}\Big)\cap E(H)=\emptyset.
\end{align*}
We then let ${\bf G}(n,\Gamma)$ denote the set of all $(n,\Gamma)$-blowups of $G$.  
\end{definition}

Which blowups appear in a property $\calH$ will be important for understanding $L_{\calH}$.  To ease our discussion, we will use the following definition, which aims to capture the ``maximal size" of a specific kind of blowup appearing in  $3$-graph $H$.

\begin{definition}\label{def:Gdim}
Suppose $H$ is a $3$-graph and $G=(U,V;E)$ is a bigraph.  The \emph{$G$-dimension of $H$} is the maximal integer $n$ so that for every $U$-colored bigraph $\Gamma=(A,B; (P_u)_{u\in U})$ with $\max\{|A|,|B|\}\leq n$,  $H$ contains an element of ${\bf G}(n,\Gamma)$ as an induced sub-$3$-graph.
\end{definition}

In other words, $H$ has $G$-dimension at least $n$ if $H$ contains an $(n,\Gamma)$-blowup of $G$ for all choices of $\Gamma$ with vertex sets of size at most $n$.  The concept of $G$-dimension for specific values of $G$ has implicitly appeared in previous papers by the author \cite{Terry.2022} and by the author and Wolf \cite{Terry.2021a,Terry.2021b}.  These works leverage the fact that $\VC_2$-dimension can be recharacterized in the form of $U_{bg}(k)$-dimension (we also reprove this in the next section, see Lemma \ref{lem:2.12a}). An analogous relationship exists between the bigraph $H(k)$ (see Definition \ref{def:hk}) and a notion called the \emph{functional order property}, which was studied  in \cite{Terry.2021a,Terry.2021b, Aldaim.2023}.  In this paper, we are also interested in $G$-dimension where $G$ ranges over certain collections of irreducible bigraphs, which necessitates the more general definition above.  This is also the motivation for the following definition, which extends the notion of ``$G$-dimension" to hereditary classes.

\begin{definition}
Given a hereditary $3$-graph property   $\calH$ and a bigraph $G$, define the \emph{$G$-dimension of $\calH$} as follows. We say $\calH$ has infinite $G$-dimension if for all $m$, there is some element of $\calH$ with $G$-dimension $m$.  Otherwise, the $G$-dimension of $\calH$ is the maximum integer $m$ so that some element in $\calH$ has $G$-dimension $m$.
\end{definition}

We now define an auxiliary class associated to a hereditary property $\calH$, based on the notion of $G$-dimension.  This can be seen as an analogue of the $\calB_{\calH}$ classes used in Part 2 of this series, although the actual definitions are distinct (see Definition 3.24 in \cite{Terry.2024b}).

\begin{definition}
Define $\calB_{\calH}$ to be the class of irreducible bigraphs $G=(U,V;E)$ so that $\calH$ has infinite $G$-dimension. 
\end{definition}

In other words, $\calB_{\calH}$ consists of the irreducible bigraphs $G$ so that arbitrary blowups of $G$ appear in $\calH$ (in the sense of Definition \ref{def:blowup2}).  We will show that when $\calB_{\calH}$ contains finitely many graphs up to isomorphism, then $L_{\calH}$ is constant, and otherwise, $L_{\calH}$ is at least polynomial.  This will characterize the constant to polynomial jump.

\subsection{Connecting blowups and encodings}\label{ss:blowupcorner}

In this section, we connect encodings and blowups.  In particular, Proposition \ref{prop:encodingtoblowup} tells us that encodings in the sense of Definition \ref{def:encodingR} produces blow ups in the sense of Definition \ref{def:blowup2}. 

\begin{proposition}\label{prop:encodingtoblowup}
For all integers  $m\geq 1$ there exist $\mu^*>0$ and a polynomial $q(x,y,z)$ such that for all $0<\delta<1$, all $0< \mu<\mu^*$, all integers $t,\ell\geq 1$, and all $0<\e_2<  q(\mu,\delta,\ell^{-1})$, the following holds.  Assume $G=(U,V;E_G)$ is a bigraph with $\max\{|U|,|V|\}\leq m$.

Suppose $H$ is a sufficiently large $3$-uniform $3$-graph and $\calP$ is a $(t,\ell)$-decomposition of $V$.  If there exists a $\delta$-non-trivial $(\e_2,\mu)$-encoding of $G$ in $(H,\calP)$, then $H$ has $G$-dimension at least $m$. 
\end{proposition}
\begin{proof}
Let $p(x,y)$ and $D$ be from Corollary \ref{cor:countingcor} applied to the integer $2m$. The define $\mu^*=(1/2)^D/6$ and set $q(x,y,z)=p(x,zy)$.  Fix $0<\delta<1$, $0<\mu<\mu^*$,  integers $t,\ell\geq 1$, and $0<\e_2< q(\mu,\delta,\ell^{-1})=p(\mu,\delta \ell^{-1})$.  Let $G=(U,V;E_G)$ be a bigraph satisfying $\max\{|U|,|V|\}\leq m$. 

Assume $H$ is a sufficiently large $3$-uniform $3$-graph and assume $\calP$ is a $(t,\ell)$-decomposition of $H$.  Say $\calP_1=\{V_1,\ldots, V_t\}$ and $\calP_2=\{P_{ij}^{\alpha}: 1\leq i,j\leq t, 1\leq \alpha\leq \ell\}$.  Assume there exists an $\delta$-non-trivial $(\e_2,\mu)$-encoding of $G$ in $(H,\calP)$. This means there are functions $f:U\rightarrow \calP_{edge}(\e_2,\mu)$ and $g:V\rightarrow \calP_{cnr}(\e_2,\mu)$ such that the following hold for some $j_0k_0\in {[t]\choose 2}$.  
\begin{enumerate}
 \item $\mathrm{Im}(f)\subseteq \{P_{j_0k_0}^{\alpha}: 1\leq \alpha\leq \ell\}$, and $\textrm{Im}(g)\subseteq \{P^{\beta}_{ij_0}P^\gamma_{ik_0}: i\in [t], 1\leq\beta,\gamma\leq \ell\}$,  
 \item For all $u\in U$ and $v\in V$, if $(u,v)\in E_G$, then $(f(u),g(v))\in \mathbf{E}_1(\e_2,\mu)$, and if $(u,v)\notin E_G$, then $(f(u),g(v))\in \mathbf{E}_0(\e_2,\mu)$.
 \item We have $|V_{j_0}|\geq \delta |V(H)|/4t$, $|V_{k_0}|\geq \delta |V(H)|/4t$, and for each $P^{\alpha}_{j_0k_0}\in Im(f)$ we have $|P_{j_0k_0}^{\alpha}|\geq \delta |V_{j_0}||V_{k_0}|/\ell$, and
 \item For each $P^{\beta}_{ij_0}P^\gamma_{ik_0}\in Im(g)$, we have $|V_i|\geq \e_1 |V(H)|/t$,  $|P_{ij_0}^{\beta}|\geq \delta |V_{i}||V_{j_0}|/\ell$, and $|P_{ik_0}^{\gamma}|\geq \delta |V_{i}||V_{k_0}|/\ell$.
 \end{enumerate}
For each $u\in U$, let $\alpha_u\in [\ell]$ be such that $f(u)=P_{j_0k_0}^{\alpha_u}$. For each $v\in V$, let $\beta_v,\gamma_v\in [\ell]$ and $i_v\in [t]$ be such that $g(v)=(P_{i_vj_0}^{\beta_v},P_{i_vk_0}^{\gamma_v})$.  Note that if  $(f(u),g(v))\in \mathbf{E}_0(\e_2,\mu)\cup \mathbf{E}_1(\e_2,\mu)$, then the triad $(V_{i_v},V_{j_0},V_{k_0}; P_{i_vj_0}^{\beta_v},P_{i_vk_0}^{\gamma_v},P_{j_0k_0}^{\alpha_u})$ satisfies $\dev_{2,3}(\e_2, 6\mu)$ with respect to $H$ by Proposition \ref{prop:homimpliesrandome}.

It suffices to show that for any $U$-colored bigraph $\Gamma=(A, B; (P_u)_{u\in U})$ with $|A|=|B|=m$, $H$ contains an element of ${\bf G}(m,\Gamma)$ as an induced sub-$3$-graph.  

Fix  a $U$-colored bigraph $\Gamma=(A, B; (P_u)_{u\in U})$ with $|A|=|B|=m$.  We will next define an auxiliary bigraph.  We begin by defining new sets of vertices,
\begin{align*}
A=A_1\cup \ldots \cup A_{m},\text{ }B= B_1\cup \ldots \cup B_{m},\text{ and }C=\bigcup_{v\in V}C_1^v\cup \ldots \cup C_n^v,
\end{align*}
where for each $1\leq x\leq m$, $A_x$ is a copy of $V_{j_0}$, for each $1\leq y\leq m$, $B_y$ is a copy of $V_{k_0}$, and for each $v\in V$ and $1\leq z\leq m$, $C_z^v$ is a copy of $V_{i_v}$.  

Given $1\leq x,y\leq m$, $1\leq z\leq m$, $u\in U$ and $v\in V$, let $G_{uv}^{xyz}$ denote the triad with vertex sets $A_x, B_y, C_z^v$, and where $G_{uv}^{xyz}[A_x,B_y]$ is a copy of $(V_{j_0},V_{k_0};P_{j_0k_0}^{\alpha_u})$, $G_{uv}^{xyz}[A_x,C_z^v]$ is a copy of $(V_{i_v},V_{j_0};P_{i_vj_0}^{\beta_v})$ and $G_{uv}^{xyz}[B_y,C_z^v]$ is a copy of $(V_{i_v}, V_{k_0}; P_{i_vk_0}^{\gamma_v})$.

We now define $\Omega$ be the $3$-partite $3$-graph with vertex set $A\cup B\cup C$ and edge set $E_{\Omega}$ defined as follows. For each $1\leq x,y\leq m$, $v\in V$, and $1\leq z\leq m$, define $E_{\Omega}\cap K_3[A_x,B_y,C_z^v]$ so that 
\[
\overline{E}_{\Omega}\cap (A_x\times B_y\times C_z^v)=\begin{cases}\text{ a copy of $\overline{E}_H\cap K_3(G_{uv}^{xyz})$}&\text{ if $(x,y)\in P_u$ and $(u,v)\in E_G$}\\
\text{ a copy of $K_3(G_{uv}^{xyz})\setminus \overline{E}_H$}&\text{ if $(x,y)\in P_u$ and $(u,v)\notin E_G$.  }\end{cases}\]

By Corollary \ref{cor:countingcor}, and since $V$ is sufficiently large, there exists some 
$$
(a_x)_{x\in [m]}(b_y)_{y\in [m]}(c_z^u)_{z\in [m], u\in U}\in \prod_{x\in [m]}A_x\times \prod_{y\in [m]}B_y\times \prod_{z\in [m],v\in V}C_z^v
$$
 so that the following holds.  If $(x,y)\in P_u$ and $(u,v)\in E_G$, then $a_xb_yc_z^u\in E_{\Omega}$, and if $(x,y)\in P_u$ and $(u,v)\notin E_G$, then $a_xb_yc_z^u\notin E_{\Omega}$.  It is an exercise to check that, by definition of $\Omega$, this yields an induced sub-$3$-graph of $H$ which is an element of  ${\bf G}(m,\Gamma)$.
\end{proof}

\subsection{An equitability lemma}\label{ss:refine}

In this subsection we show that given a $(t,\ell)$-decomposition $\calQ$, we can obtain a new decomposition $\calP$ with the same vertex partition $\calP_1=\calQ_1$ and with pairs partition $\calQ_2$ a mostly equitable refinement of $\calP_2$ (in the sense that most elements in $\calQ_2$ have approximately the same density).   Related refinement lemmas have appeared in the literature (see e.g.  \cite{Frankl.2002} or Section 3.2 of \cite{Terry.2022}). However, these are more complicated than what is needed here.  This can be done with only polynomial loss in the parameters, due to the following adaptation a lemma of Frankl and R\"{o}dl  \cite{Frankl.2002}.

\begin{lemma}\label{lem:3.8}
Suppose $0<\e<(1/2)^{12}$,  $\rho\geq 2\e^{1/12}$, $0<p<\rho/2$, and $u=[1/p]$.  Let $G=(U,V;R)$ be a bigraph satisfying $\dev_2(\e)$, where $m=\min\{|U|, |V|\}\geq m_0(\e, u)$, $d_G(U,V)=\rho$, and $\e\geq 10(p/m)^{1/5}$.  Then there exists a partition 
$$
R=E_0\cup \ldots \cup E_u,
$$
so that $|E_0|\leq \rho p(1+o(1))|U||V|$, and for each $1\leq \alpha\leq u$, $(U,V; E_{\alpha})$ satisfies $\dev_2(\e^{1/12})$ with density $\rho p(1+o(1))$, where $o(1)\rightarrow 0$ as $m\rightarrow \infty$.  Moreover, if $1/p$ is an integer, then $E_0=\emptyset$.
\end{lemma}
 
We refer the reader to the appendix for details on how to deduce this from \cite{Frankl.2002}. We now state and prove our refinement lemma.  Below we use the notation $\calQ_2\preceq \calP_2$ to denote that every element in $\calQ_2$ is a subset of an element in $\calP_2$.

\begin{lemma}\label{lem:equitable}
There is $\e_1^*>0$ and a polynomial $q(x,y)$ so that the following holds. Fix $0<\e_1<\e_1^*$, $\e_2:\mathbb{N}\rightarrow (0,1]$ satisfying, $\e_2(x)< q(\e_1,x^{-1})$, integers $t,\ell \geq 1$, a polynomial $p(x,y)\geq 2x^2y^2$, and $V$ a sufficiently large set.

Suppose $\calQ$ is a $(t,\ell,\e_1,\e_2(\ell))$-decomposition of $V$, and 
$$
\ell_1=\lceil p(\e_1^{-1},\ell)\rceil\text{ and }\ell_2=3\ell \ell_1.
$$
Then there exists a $(t,\ell_2, 4\e_1,\e_2(\ell)^{1/12})$-decomposition $\calP$ of $V$ satisfying the following.
\begin{enumerate}
\item $\calP_1=\calQ_1$ and  $\calP_2\preceq \calQ_2$, and
\item At least $(1-4\e_1)|V|^2$ many pairs from $V^2$ are  in some $P_{ij}^{\alpha}\in \calP_2$ where  $(V_i,V_j; P_{ij}^{\alpha})$ satisfies $\dev_2(\e_2(\ell)^{1/12}, 1/\ell_1)$ and $\min\{|V_i|,|V_j|\}\geq \e_1|V|/t$.
\end{enumerate}
\end{lemma}
\begin{proof}
Let $\e_1^*>0$ be sufficiently small and set $q(x,y)=(xy/2)^{12}$.  Fix $0<\e_1<\e_1^*$, $\e_2:\mathbb{N}\rightarrow (0,1]$ satisfying $\e_2(x)< q(\e_1,x^{-1})$, integers $t,\ell \geq 1$, and a polynomial $p(x,y)\geq 2x^2y^2$.   

Let  $V$ be sufficiently large and assume $\calQ=(\calQ_1,\calQ_2)$ is a $(t,\ell,\e_1,\e_2(\ell))$-decomposition of $V$, where 
 $$
 \calQ_1=\{V_1,\ldots, V_t\}\text{ and }\calQ_2=\{Q_{ij}^{\alpha}: 1\leq i,j\leq t,1\leq \alpha\leq \ell\}.
 $$ 
We call $V_i\in \calQ_1$ \emph{non-trivial} if $|V_i|> \e_1|V|/t$.  For each $Q_{ij}^{\alpha}\in \calQ_2$, set $d_{ij}^{\alpha}=|Q_{ij}^{\alpha}|/|V_i||V_j|$, and call $Q_{ij}^{\alpha}$ \emph{non-trivial} if $V_i,V_j$ are both non-trivial and $d^{\alpha}_{ij}>\e_1|V_i||V_j|/\ell$.  Note that 
\begin{align}\label{al:triv}
\nonumber\sum_{Q_{ij}^{\alpha}\in \calQ_2 \text{ trivial }}|Q_{ij}^{\alpha}|&\leq \sum_{V_i\in \calQ_1\text{ trivial }}|V_i||V|+\sum_{V_i\in \calQ_1\text{ non-trivial }}\Big(\sum_{V_j\in \calQ_1\text{ non-trivial }}\Big(\sum_{\{\alpha\in [\ell]: Q_{ij}^{\alpha} \text{ trivial }\}}|Q_{ij}^{\alpha}|\Big)\Big)\\
\nonumber&\leq t\Big(\frac{\e_1 |V|^2}{t}\Big)+\sum_{(V_i,V_j)\in \calQ_1^2}\ell\Big(\frac{\e_1|V_i||V_j|}{\ell}\Big)\\
&\leq 2\e_1|V|^2.
\end{align}
We moreover say $Q_{ij}^{\alpha}\in \calQ_2$ is \emph{relevant} if it is non-trivial and satisfies $\dev_2(\e_2(\ell))$.  Note that since $\calQ$ is a $(t,\ell,\e_1,\e_2(\ell))$-decomposition of $V$, we have
 \begin{align*}
 \sum_{Q_{ij}^{\alpha}\in \calQ_2\text{ failing }\dev_2(\e_2(\ell))}|Q_{ij}^{\alpha}||V|\leq \e_1 |V|^3,
 \end{align*}
 and consequently, $ \sum_{Q_{ij}^{\alpha}\in \calQ_2\text{ failing }\dev_2(\e_2(\ell))}|Q_{ij}^{\alpha}|\leq \e_1 |V|^2$.  Combining with (\ref{al:triv}), we have 
 \begin{align}\label{al:rel}
 \sum_{Q_{ij}^{\alpha}\in \calQ_2\text{ not relevant }}|Q_{ij}^{\alpha}|\leq \sum_{Q_{ij}^{\alpha} \in \calQ_2\text{ trivial }}|Q_{ij}^{\alpha}|+\sum_{Q_{ij}^{\alpha}\in \calQ_2\text{ failing }\dev_2(\e_2(\ell))}|Q_{ij}^{\alpha}|\leq 3\e_1|V|^2.
 \end{align}
We now set  
$$
\ell_1=\lceil p(\e_1^{-1},\ell)\rceil\text{ and }\ell_2=3\ell\ell_1.
$$
We will obtain our new decomposition by refining all relevant $Q_{ij}^{\alpha}$ into parts which are mostly of density $1/\ell_1$, and leaving the non-relevant $Q_{ij}^{\alpha}$ untouched.  

Suppose $Q_{ij}^{\alpha}\in \calQ_2$ is relevant.  We claim the hypotheses of Lemma \ref{lem:3.8} are satisfies by $(V_i, V_j; Q_{ij}^{\alpha})$ with parameters $\e=\e_2(\ell)^{1/12}$, $\rho=d_{ij}^{\alpha}$, and $p=(\ell_1 d_{ij}^{\alpha})^{-1}$. Indeed, by our choice of $q(x,y)$ and our assumption on $\e_2$, 
$$
\e=\e_2(\ell)^{1/12}<q(\e_1,\ell^{-1})^{1/12}\leq \frac{\e_1}{2\ell}\leq \rho.
$$
Further, by definition of $\rho$ and since $Q_{ij}^{\alpha}$ is non-trivial, $\rho/2=d_{ij}^{\alpha}/2> \e_1/2\ell$.  Combining with our assumption that $p(x,y)\geq 2x^2y^2$ and definition of $\ell_1$, we have
$$
0<p=(\ell_1)^{-1}(d_{ij}^{\alpha})^{-1}< \frac{\e_1^2}{2\ell^2} \frac{\ell}{\e_1}=\frac{\e_1}{2\ell}< \rho/2.
$$
Thus, we can apply Lemma \ref{lem:3.8} to $(V_i, V_j; Q_{ij}^{\alpha})$ to obtain a partition  
$$
Q_{ij}^{\alpha}=Q_{ij}^{\alpha}(0)\cup Q_{ij}^{\alpha}(1)\cup \ldots \cup Q_{ij}^{\alpha}(\ell_{ij}^{\alpha}),
$$
where $|Q_{ij}^{\alpha}(0)|\leq (1+\e_2(\ell))\ell_1^{-1}|V_i||V_j|$, and where for each $1\leq \beta\leq \ell_{ij}^{\alpha}$,  $(V_i,V_j; Q_{ij}^{\alpha}(\beta))$ satisfies $\dev_2(\e_2(\ell)^{1/12}, \ell_1^{-1})$. Note we have inserted $\e_2(\ell)$ for the $o(1)$ term in the conclusion of Lemma \ref{lem:3.8}. We can do because $V_i$ and $V_j$ are non-trivial and $|V|$ is sufficiently large.  Note $\ell_{ij}^{\alpha}\leq \ell_1$ by definition.

Now for each $1\leq i,j\leq t$, let $P_{ij}^1,\ldots, P_{ij}^{\ell_{ij}}$ be an enumeration satisfying
$$
\{P_{ij}^1,\ldots, P_{ij}^{\ell_{ij}}\}=\{Q_{ij}^{\alpha}(u) : \alpha\in [\ell], Q_{ij}^{\alpha}\text{ relevant}, 0\leq u\leq \ell_{ij}^{\alpha}\}\cup \{Q_{ij}^{\beta}: \beta\in [\ell], Q_{ij}^{\beta}\text{ not relevant}\}.
$$
Note that $\ell_{ij}\leq \ell(\ell_1+1)+\ell\leq \ell_2$.  We then set $P_{ij}^{\alpha}=\emptyset$ for any $\ell_{ij}<\alpha\leq \ell_2$.

We now define  $\calP=(\calP_1,\calP_2)$ by setting
$$
 \calP_1=\{V_1,\ldots, V_t\}\text{ and }\calP_2=\{P_{ij}^{\alpha}: 1\leq i,j\leq t,1\leq \alpha\leq \ell_2\}.
 $$ 
We show $\calP$ satisfies the desired conclusions. Set 
$$
\Omega=\{P_{ij}^{\alpha}\in \calP_2: V_i,V_j\text{ are non-trivial and }(V_i,V_j;P_{ij}^{\alpha})\text{ satisfies }\dev_2(\e_2(\ell)^{1/12}, \ell_1^{-1})\}.
$$
It suffices to show the set $\Theta:=\bigcup_{P_{ij}^{\alpha}}P_{ij}^{\alpha}$ has size at least $(1-4\e_1)|V|^2$. Observe, using (\ref{al:rel}) and the construction above,
\begin{align*}
|(V\times V)\setminus \Theta|&\leq \sum_{Q_{ij}^{\alpha}\in \calQ_2\text{ not relevant }}|Q_{ij}^{\alpha}|+\sum_{Q_{ij}^{\alpha}\in \calQ_2\text{ relevant }}|Q_{ij}^{\alpha}(0)|\\
&\leq 3\e_1|V|^2+\sum_{Q_{ij}^{\alpha}\in \calQ_2\text{ relevant }}(1+\e_2(\ell))\ell_1^{-1}|V_i||V_j|\\
&\leq 3\e_1|V|^2+(1+\e_2(\ell))\ell_1^{-1}|V|^2\\
&\leq 4\e_1 |V|^2,
\end{align*}  
where the last inequality uses that $(1+\e_2(\ell))\ell_1^{-1}\leq \e_1$, which holds by our assumption on $\e_2$, the definition of $\ell_1$, the definition of $q(x,y)$, and since $\e_1<\e_1^*$ is sufficiently small.
\end{proof}

\section{Upper Bounds}\label{ss:UBL}

The main results of this section are Propositions \ref{prop:main2} and \ref{prop:mainL}, which will give us sufficient conditions for the upper bounds appearing in ranges (2) and (3) of Theorem \ref{thm:mainL2}, respectively.  

Our first goal is to prove Proposition \ref{prop:main2}, which shows that given a sufficiently regular decomposition $\calP$ for a $3$-graph $H$ with small $\VC_2$-dimension, we can build a new regular decomposition $\calP'$ with the same vertex partition as $\calP$, and where the complexity of the pairs component in $\calP'$ is only polynomial. A version of this result was first proved in \cite{Terry.2022}, and an examination of the proofs there show they require only polynomial $\e_2$, as required for Theorem \ref{thm:mainL2}.  We go through the trouble to repeat a proof here for several reasons. First, as the growth rate of $\e_2$ was not of primary interest in  \cite{Terry.2022}, it was not explicitly stated there that only polynomial $\e_2$ was required (although this is  immediate from the polynomial dependence in the counting lemma used there).  Second, while \cite{Terry.2022} dealt exclusively with equitable decompositions, mild adjustments to the proof allow the removal of this assumption.  Third we provide a different end to the proof which we think is conceptually simpler than that appearing in \cite{Terry.2022}.  Finally, it will serve as a warm up to the proof of the constant upper bound in range (3) of Theorem \ref{thm:mainL2}, which is a new result with a very similar proof.

The first main ingredient for Proposition \ref{prop:main2} is a connection between $\VC_2$-dimension and the $G$-dimension when $G$ has the form $U_{bg}(k)$.  In particular,  it follows from arguments already appearing in \cite{Terry.2021a,Terry.2021b} that a $3$-graph with sufficiently large $U_{bg}(k)$-dimension has $\VC_2$-dimension at least $k$ (see Definition  \ref{def:Gdim}).  We repeat a proof of this here as the results in \cite{Terry.2021a, Terry.2021b} are not phrased in this language.  For the convenience of the reader, we recall  that $U_{bg}(k)=(B_k, A_k;  E_k)$, where $B_k=\{b_S: S\subseteq [k]\}$,  $A=\{a_1,\ldots, a_k\}$, and $E_k=\{(b_S, a_i): i\in S\}$ (see  Definition \ref{def:ukbg}).  

\begin{lemma}\label{lem:2.12a}
Suppose $H$ is a $3$-graph with $U_{bg}(k)$-dimension at least $2^{k^2}$. Then $H$ has $\VC_2$-dimension at least $k$.
\end{lemma}
\begin{proof}
Suppose $H$ is a $3$-graph with $U_{bg}(k)$-dimension at least $2^{k^2}$. Define a $B_k$-colored bigraph $\Gamma=(X,Y;(P_{b_S})_{b_S\in B_k})$ as follows. Set 
$$
X=\{x_1,\ldots, x_k\}\text{ and }Y=\{y_T: T\subseteq [k]^2\},
$$
 and define, for each $S\subseteq [k]$,
\begin{align}\label{ps}
P_{b_S}:=\{(x_{\alpha},y_T)\in X\times Y: S=\{j\in [k]: (\alpha,j)\in T\}\}.
\end{align}
By assumption, $H$ contains a $(2^{k^2},\Gamma)$-blow up of $U_{bg}(k)$ as an induced sub-$3$-graph.  After relabeling if necessary, we may assume $X,Y\subseteq V(H)$, and there exist non-empty sets of vertices $Z_{a_1},\ldots, Z_{a_k}\in V(H)$ so that 
\begin{align}\label{edges}
&\bigcup_{(b_S,a_i)\in E_k}\{xyz: (x,y)\in P_{b_S}, z\in Z_{a_i}\}\subseteq E(H)\\
\nonumber&\bigcup_{(b_S,a_i)\notin E_k}\{xyz: (x,y)\in P_{b_S}, z\in Z_{a_i}\}\cap E(H)=\emptyset.
\end{align}
Fix any $z_1\in Z_{a_1},\ldots, z_k\in Z_{a_k}$.  Given $T\subseteq [k]^2$ and $\alpha,\beta\in [k]$, we have by (\ref{edges}) that $x_\alpha y_Tz_\beta\in E(H)$ if and only if $(x_\alpha, y_T)\in P_{b_S}$ for some $S\subseteq[k]$ satisfying $(b_S, a_{\beta})\in E_k$.    Combining with the definition of $E_k$, we have that $x_\alpha y_Tz_\beta\in E(H)$ if and only if $(x_\alpha, y_T)\in P_{b_S}$ for some $S\subseteq [k]$ satisfying $\beta\in S$.  Combining this with the definition of the edge colors (\ref{ps}), this tells us   $x_\alpha y_Tz_\beta\in E(H)$ if and only if  $\beta\in \{j\in [k]: (\alpha,j)\in T\}$, i.e. if and only if $(\alpha,\beta)\in T$.  We have now shown $H$ has $\VC_2$-dimension at least $k$, as desired. 
\end{proof}

Combining Lemma \ref{lem:2.12a} with Proposition \ref{prop:encodingtoblowup} yields essentially another proof of Proposition 2.12 from  \cite{Terry.2022}, this time with the explicit statement that $\e_2$ can be taken to be polynomial. 

\begin{corollary}[Proposition 2.12 of \cite{Terry.2022}]\label{cor:2.12}
For all integers  $k\geq 1$ there exist $\mu^*>0$ and a polynomial $q(x,y,z)$ such that for all $0<\delta<1$, all $0< \mu<\mu^*$,  all integers $t,\ell\geq 1$, and all $0<\e_2< q(\mu,\delta,\ell^{-1})$, and the following holds. 

Suppose $H$ is a sufficiently large $3$-uniform $3$-graph and $\calP$ is a $(t,\ell)$-decomposition of $V(H)$.  If there exists a $\delta$-non-trivial $(\e_2,\mu)$-encoding of $U_{bg}(k)$ in $(H,\calP)$, then $H$ has $\VC_2$-dimension at least $k$.  
\end{corollary}
\begin{proof}
Let $\mu^*>0$ and $q(x,y,z)$ be as in Proposition \ref{prop:encodingtoblowup} applied to $m=2^{k^2}$.  Fix $0<\delta<1$, $0<\mu<\mu^*$, $t,\ell\geq 1$, and $0<\e_2<q(\mu,\delta,\ell^{-1})$. Assume  $H$ is a sufficiently large $3$-graph and $\calP$ is  a $(t,\ell )$-decomposition of $V(H)$.  Suppose there exists a $\delta$-non-trivial $(\e_2,\mu)$-encoding of $U_{bg}(k)$ in $(H,\calP)$. By Proposition \ref{prop:encodingtoblowup}, $H$ has $U_{bk}(k)$-dimension at least $m$, implying by Lemma \ref{lem:2.12a} that $\VC_2(H)\geq k$.
\end{proof}

  We now state and prove Proposition \ref{prop:main2}.

\begin{proposition}\label{prop:main2}
For all $k\geq 1$, there are polynomials $p(x,y)$, $r(x), q(x,y,z), s(x,y)$ and a constant $\e_1^*>0$ so that for all  $0<\e_1<\e_1^*$, all non-increasing $\e_2:\mathbb{N}\rightarrow (0,1]$ with  $\e_2(x)\leq p(\e_1,x^{-1})$, and all integers $t,\ell\geq 1$, the following holds.  Suppose $H=(V,E)$ is a sufficiently large $3$-graph satisfying the following.
\begin{enumerate}
\item $\VC_2(H)<k$, 
\item There exists a $\dev_{2,3}(\e_1',\e_2'(\ell))$-regular $(t,\ell, \e_1',\e_2'(\ell))$-decomposition $\calQ$ for  $H$ where 
$$
\e_1'=r(\e_1)\text{ and }\e_2'(\ell)= q(\e_1,\ell^{-1},\e_2(\lceil s(\e_1^{-1},\ell)\rceil))).
$$
\end{enumerate}
Then there exists $\ell_*\leq \e_1^{-O_k(k)}$ and a $\dev_{2,3}(\e_1,\e_2(\ell_*))$-regular $(t,\ell_*,\e_1,\e_2(\ell_*))$-decomposition for $H$ with the same vertex partition as $\calQ$.   
\end{proposition}
\begin{proof}
Fix $k\geq 1$ and let $c=c(k)$ be as in Lemma \ref{lem:hausslercor}.  Let $\mu_1>0$ and $p_1(x,y,z)$ be as in Corollary \ref{cor:2.12} for $k$. Let $p_2(x,y)$ and $D\geq 3$ be from Corollary \ref{cor:suffvc2} for $k$, and let $\mu_2=2^{-D}$.  Let $\mu_3>0$ and $p_3(x,y)$ be from Lemma \ref{lem:equitable}, and set  $p_4(x,y)=2x^3y^3$.  Let $\mu_5$ be sufficiently small so that $(1-2\mu)^7>1-\mu^{1/2}$, and $\frac{1+\mu}{1-\mu}<3/2$, and define $\e_1^*=\min\{2^{-16},\mu_1,\mu_2,\mu_3,\mu_4,\mu^5_5\}$.  
%Let $r(x)=(x/60)^{256}$, $r'(x)=(c^{-2}(r(x)/8)^{2k+2})^{64}$, and $r_1(x)=r'(x)^D$, and define
%$$
%p(x,y)=\Big(\frac{r_1(x)}{4x^3}\Big)^{100}p_1(r_1(x),y)p_2(r_1(x),y)p_3(r_1(x),y).
%$$

%Fix $0<\e_1<\e_1^*$ and $\e_2:\mathbb{N}\rightarrow (0,1]$ satisfying $\e_2(x)\leq p(\e_1,x^{-1})$.  To ease notation, we set 
%$$
%\delta=r(\e_1),\text{ }, \mu=r'(\delta),\text{ and }\e_1'=\mu^D=r_1(\e_1).
%$$
%Note $\e_1'<\mu<\delta<\e_1$.  Observe that our assumption on $\e_2$ implies that for all $x\in \mathbb{N}$, 
%\begin{align}\label{e2}
%\e_2(x)\leq \Big(\frac{\e_1'}{4x^3}\Big)^{100}\min\{p_1(\e_1',x^{-1}), p_2(\e_1',x^{-1}), p_3(\e_1', x^{-1})\}.
%\end{align}
Fix $0<\e_1<\e_1^*$. To ease notation, set 
$$
\delta=\Big(\frac{\e_1}{60}\Big)^{256},\text{ } \mu=\Big(c^{-2}\Big(\frac{\delta}{8}\Big)^{2k+2}\Big)^{64}\text{ and }\e_1'=\mu^D.
$$
Clearly there is a polynomial $r(x)$, depending only on $k$, so that $\e_1'=r(\e_1)$.  It will be useful for the reader to keep in mind the following relationship among these paramters.
$$
\e_1'<\mu<\delta<\e_1.
$$
Clearly there exists a polynomial $p(x,y)$, depending only on $k$,  so that 
$$
p(\e_1,y)=\Big(\frac{\e_1'}{4y^3}\Big)^{100}p_1(\mu, \e_1',y)p_2(\e_1',y)p_3(\e_1',y),
$$
and fix a non-increasing $\e_2:\mathbb{N}\rightarrow (0,1]$ satisfying $\e_2(y)\leq p(\e_1,y^{-1})$ for all $y\in \mathbb{N}$.  Define $\e_2':\mathbb{N}\rightarrow (0,1]$ by setting
\begin{align*}
\e_2'(y)=\Big(\frac{\e_1'\mu\delta\e_2(\lceil 6yp_4(\e_1',y^{-1})^{-1}\rceil )}{6yp_4(\e_1',y^{-1})^{-1}}\Big)^{144}.
\end{align*}
Clearly there exist polynomials $q(x,y,z)$ and $s(x,y)$, depending only on $k$, so that $\e_2'(y)$ is equal to $q(\e_1, y^{-1}, \e_2(\lceil s(\e_1^{-1},y))\rceil )$.

Fix integers $t,\ell\geq 1$, suppose $H$ is a sufficiently large $3$-graph of $\VC_2$-dimension less than $k$, and assume there exists a $\dev_{2,3}(\e_1',\e_2'(\ell))$-regular $(t,\ell,\e_1',\e_2'(\ell))$-decomposition $\calQ$  for $H$. To ease notation, throughout the proof we let $V=V(H)$ and $E=E(H)$.  Since $\VC_2(H)<k$, Corollary \ref{cor:suffvc2} implies $\calQ$ is $\mu$-homogeneous in the sense of Definition \ref{def:homdec} (recall $\mu=(\e_1')^{1/D}$).  Set  
$$
\ell_1= \lceil 2(\e'_1)^{-3}\ell^3\rceil\text{ and }\ell_2=3\ell \lceil 2(\e'_1)^{-3}\ell^3\rceil.
$$
In other words, 
$$
\ell_1=\lceil p_4(\e_1',\ell^{-1})^{-1}\rceil\text{ and }\ell_2=3\ell\lceil p_4(\e_1',\ell^{-1})^{-1}\rceil.
$$
We now list several inequalities, for easy reference later in the proof, which hold by the definitions of $\e_2',\ell_1,\ell_2$, and our assumption on $\e_2$.
\begin{align}
&(\ell_1+1)(\ell_1^{-1}-\e_2'(\ell)^{1/12})>1\label{relboundin}\\
&4(\e_2'(\ell)^{1/12})^{1/4}<\mu \label{align:1in}\\
&\e_2'(\ell)^{1/12}<\mu\ell_1^{-1} \label{align:2in}\\
&\e_2'(\ell)^{1/12}<q(\mu^{1/2},\e_1',\ell_2^{-1})\label{encode}\\
&\ell_1\e_2'(\ell)^{1/144}<\e_2(\ell_2)\label{devin}\\
&4\e_2(\ell_2)^{1/4}<\delta/\ell_1^3 \label{al:gcountin}.
\end{align}
%(\ell_1+1)(\ell_1^{-1}-\e_2'(\ell)^{1/12})>1 (relbound), 4(\e_2'(\ell)^{1/12})^{1/4}<\mu (align:1), \e_2'(\ell)^{1/12}<\mu\ell_1^{-1}(align:2) , \ell_1\e_2'(\ell)^{1/44}<\e_2(\ell_2), 4\e_2(\ell_2)^{1/4}<\delta/\ell_1^3 (al:gcount)
%\begin{align}\label{es}
%\e_2'(\ell)<\min\Big\{\Big(\frac{1}{\ell_1}-\frac{1}{\ell_1+1}\Big)^{12},\Big(\frac{\e_2(\ell_2)}{\ell_2}\Big)^{144}, \Big(\frac{\mu}{4\ell_1^3}\Big)^{12}\Big\}, \text{ and }\e_2(\ell_2)<\Big(\frac{\delta}{4\ell_1^3}\Big)^4. 
%\end{align}
By Lemma \ref{lem:equitable} (applied with polynomial $p_4(x,y)$), there exist
$$
 \calP_1= \{V_1,\ldots, V_t\}\text{ and }\calP_2=\{P_{ij}^{\alpha}: 1\leq i,j\leq t,1\leq \alpha\leq \ell_2\},
$$
so that $\calP=(\calP_1,\calP_2)$ is a $(t,\ell_2,4\e_1',\e_2'(\ell)^{1/12})$-decomposition and so that the following hold.
\begin{enumerate}[(a)]
\item $\calP_1=\calQ_1$ and $\calP_2\preceq \calQ_2$, 
\item At least $(1-4\e'_1)|V|^2$ pairs from $V^2$ lie in some $P_{ij}^{\alpha}\in \calP_2$ such that $(V_i,V_j; P_{ij}^{\alpha})$ satisfies $\dev_2(\e_2'(\ell)^{1/12},\ell_1^{-1})$ and where $\min\{|V_i|,|V_j|\}\geq \e'_1|V|/t$.  
\end{enumerate}
We say  $P_{ij}^{\alpha}\in \calP_2$ is \emph{relevant} if it satisfies the conditions in (b), i.e. if $(V_i,V_j; P_{ij}^{\alpha})$ satisfies $\dev_2(\e_2'(\ell)^{1/12},\ell_1^{-1})$ and $\min\{|V_i|,|V_j|\}\geq \e'_1|V|/t$.   Then (b) tells us that at least $(1-4\e_1')|V|^2$ pairs from $V^2$ lie in a relevant $P\in \calP_2$.  It is straightforward to see that for any $1\leq i,j\leq t$, 
 \begin{align}\label{relbound}
 |\{\alpha\in [\ell_1]: P_{ij}^{\alpha}\text{ is relevant }\}|\leq \ell_1.
 \end{align}
Indeed, (\ref{relbound}) follows from the fact that each relevant $P_{ij}^{\alpha}$ has density $\ell_1^{-1}\pm \e_2'(\ell)^{1/12}$ in $V_i\times V_j$, and $\e_2'(\ell)^{1/12}$ is sufficiently small compared to $\ell_1^{-1}$ (see (\ref{relboundin})).
 
Note that by (a),  the partition  $V^3=\bigcup_{G\in \triads(\calP)}K_3(G)$ refines the partition $V^3=\bigcup_{G\in \triads(\calQ)}K_3(G)$. Consequently, since $\calQ$ is $\mu$-homogeneous, Lemma \ref{lem:averaging} implies $\calP$ is $\mu^{1/2}$-homogeneous.  

We now set up some notation for triads of $\calP$.  For each $1\leq i,j,s\leq t$ and $1\leq \alpha,\beta,\gamma\leq \ell_2$,  set $G_{ijs}^{\alpha\beta\gamma}=(V_i, V_j,V_s; P_{ij}^\alpha, P_{js}^\beta, P_{is}^\gamma)$ and  define
$$
d_{ijs}^{\alpha\beta\gamma}=\frac{|\overline{E}\cap K_3(G_{ijs}^{\alpha\beta\gamma})|}{|K_3(G_{ijs}^{\alpha\beta\gamma})|}.
$$
We will call a triad from $\calP$ \emph{relevant} if it has the form $G_{ijs}^{\alpha\beta\gamma}$   where each of  $P_{ij}^\alpha, P_{js}^\beta, P_{is}^\gamma$ are relevant.   By Proposition \ref{prop:counting} and (\ref{align:1in}),  for any relevant $G_{ijs}^{\alpha\beta\gamma}$,
\begin{align}\label{align:1}
|K_3(G_{ijs}^{\alpha\beta\gamma})|=\frac{(1\pm \mu)|V_i||V_j||V_s|}{\ell_1^3}.
\end{align}
Similarly, by (\ref{align:2in}), for any relevant $P_{ij}^{\alpha}$,
\begin{align}\label{align:2}
|P_{ij}^{\alpha}|=\frac{|V_i||V_j|(1\pm \e_2'(\ell)^{1/12})}{\ell_1}=\frac{|V_i||V_j|(1\pm \mu)}{\ell_1}.
\end{align}
We now define a partition of $\triads(\calP)$.   
\begin{align*}
\mathbf{F}_{err}&=\{G_{ijs}^{\alpha\beta\gamma} \in \triads(\calP):G_{ijs}^{\alpha,\beta,\gamma}\text{ is not relevant or  }d_{ijs}^{\alpha\beta\gamma}\in (\mu^{1/2},1-\mu^{1/2})\},\\
\mathbf{F}_1&=\{G_{ijs}^{\alpha\beta\gamma} \in \triads(\calP)\setminus \mathbf{F}_{err}: d_{ijs}^{\alpha\beta\gamma}\geq 1- \mu^{1/2}\},\text{ and }\\
\mathbf{F}_0&=\{G_{ijs}^{\alpha\beta\gamma} \in \triads(\calP)\setminus \mathbf{F}_{err}:  d_{ijs}^{\alpha\beta\gamma}\leq  \mu^{1/2}\}.
\end{align*}

Our strategy, like \cite{Terry.2022}, will be to define auxiliary edge-colored bigraphs built from the components of $\calP$, to which we will then apply Lemma \ref{lem:hausslercor}.  These auxiliary bigraphs must avoid a series of error sets, each consisting of objects too closely associated with $\mathbf{F}_{err}$.  We will next mimic several steps from \cite{Terry.2022} to define and bound these error sets, with adjustments made to deal with our lack  of a vertex equipartition.  In particular, at several junctures where \cite{Terry.2022} works with elements of $\calP$, we will instead work with corresponding subsets of $V, V^2,V^3$.   

To start with, we will work not only with $\mathbf{F}_{err}$, but also the corresponding subset of $V^3$ defined below. 
\begin{align*}
F_{err}=\bigcup_{G\in \mathbf{F}_{err}}K_3(G).
\end{align*}
By (b) and since $\calP$ is $\mu^{1/2}$-homogeneous, we have 
\begin{align}\label{Ferr}
\nonumber|F_{err}|&\leq \sum_{\{G  \in \triads(\calP)\text{ not $\mu^{1/2}$-homogeneous for $H$}\}}|K_3(G)|+3\Big(\sum_{\{P_{ij}^{\alpha}\in \calP_2 \text{ not relevant}\}}|P_{ij}^{\alpha}||V|\Big)\\
\nonumber&\leq \mu^{1/2}|V|^3+12\e_1'|V|^3\\
&\leq 2\mu^{1/2}|V|^3,
\end{align}
where the last inequality is because $\e_1'$ is sufficiently small compared to $\mu$.  Next, we define the set of vertex pairs which intersect too many triples from $F_{err}$.
$$
X_{err}=\{(x,y)\in V^2: |N_{F_{err}}(x,y)|\geq \mu^{1/4} |V|\}.
$$
Note that if $P_{ij}^{\alpha}\in \calP_2$ and $P_{ij}^{\alpha}\cap X_{err}\neq \emptyset$, then $P_{ij}^{\alpha}\subseteq X_{err}$.  Moreover, if $P_{ij}^{\alpha}$ is not relevant, then $P_{ij}^{\alpha}\subseteq X_{err}$.   We now show $X_{err}$ is not too big.  By definition of $X_{err}$, we have $|F_{err}|\geq \mu^{1/4} |X_{err}||V|$. Combining this with our upper bound on $|F_{err}|$ above, we have
\begin{align}\label{xerr}
|X_{err}|\leq 2\mu^{1/4}|V|^2.
\end{align}
We now define the set of pairs $(V_i,V_j)$ where $V_i\times V_j$ contains too may elements of $X_{err}$.
\begin{align*}
\Psi&=\{(V_i,V_j)\in \calP_1^2: |X_{err}\cap (V_i\times V_j)|\geq \mu^{1/8}|V_i||V_j|\}.
\end{align*}
 Observe $|X_{err}|\geq \sum_{(V_i,V_j)\in \Psi} \mu^{1/8}|V_i||V_j|= \mu^{1/8}\sum_{(V_i,V_j)\in \Psi} |V_i||V_j|$. Combining this with (\ref{xerr}) we have
\begin{align}\label{al:psi}
\sum_{(V_i,V_j)\in \Psi} |V_i||V_j|\leq 2\mu^{1/8}|V|^2.
\end{align}
Roughly at this stage of the argument in \cite{Terry.2022}, auxiliary edge-colored bigraphs are defined.  Before doing this here, we must consider one additional type of error set. Given $1\leq i,j\leq t$, let 
 \begin{align*}
\Sigma_{ij}=\Big\{(P_{is}^{\beta},P_{js}^{\gamma}): &s\in [t], \beta,\gamma\in [\ell_2], \text{$P_{is}^{\beta},P_{js}^{\gamma}$ both relevant}\Big\},
 \end{align*}
and let
 \begin{align*}
 \Sigma^{bad}_{ij}=\Big\{(P_{is}^{\beta},P_{js}^{\gamma})\in \Sigma_{ij}: \sum_{\{\alpha\in [\ell_2]: P_{ij}^{\alpha} \text{ relevant and }G_{ijk}^{\alpha\beta\gamma}\in \mathbf{F}_{err}\}}|P_{ij}^{\alpha}|\geq \mu^{1/8} |V_i||V_j|\Big\}.
 \end{align*}

Our next few argument are aimed at giving a lower bound related to $\Sigma_{ij}$ and an upper bound related to $\Sigma^{bad}_{ij}$, in the case where $(V_i,V_j)\notin \Psi$.  First, observe that for  $(V_i,V_j)\notin \Psi$, the following holds by definition of $\Psi$ and $X_{err}$.
\begin{align}\label{xerrpsi}
\nonumber|F_{err}\cap (V_i\times V_j\times V)|&\leq |X_{err}\cap (V_i\times V_j)||V|+\mu^{1/4}|(V_i\times V_j)\setminus X_{err}||V|\\
\nonumber&\leq  \mu^{1/8}|V_i||V_j||V|+\mu^{1/4}|V_i||V_j||V|\\
&\leq 2\mu^{1/8}|V_i||V_j||V|.
\end{align}
On the other hand, by definition of $\Sigma_{ij}$ and (\ref{align:1}),
\begin{align}\label{ferr1}
|(V_i\times V_j\times V)\setminus F_{err}|\leq \sum_{\{\alpha\in [\ell_2]: P_{ij}^{\alpha}\text{ relevant}\}}\sum_{(P_{is}^{\beta},P_{js}^{\gamma})\in \Sigma_{ij}}|K_3(G_{ijs}^{\alpha\beta\gamma})|\leq \sum_{(P_{is}^{\beta},P_{js}^{\gamma})\in \Sigma_{ij}}\frac{(1+\mu)|V_i||V_j||V_s|}{\ell_1^2},
\end{align}
where the second inequality also uses  (\ref{relbound}).  Combining with (\ref{xerrpsi}), we have
\begin{align}\label{uij3}
\sum_{(P_{is}^{\beta},P_{js}^{\gamma})\in \Sigma_{ij}}|V_s|&\geq \frac{(1-2\mu^{1/8})|V|\ell_1^2}{(1+\mu)}\geq (1-\delta)\ell_1^2|V|,
\end{align}
where the last inequality is because $\mu$ is sufficiently small compared to $\delta$.  This is our desired inequality pertaining to $\Sigma_{ij}$.  We now turn to $\Sigma_{ij}^{bad}$.  Observe that  the following holds by definition of $F_{err}$, (\ref{align:1}), (\ref{align:2}), and definition of $\Sigma^{bad}_{ij}$.  
\begin{align*}
|F_{err}\cap (V_i\times V_j\times V)|&\geq \sum_{(P_{is}^{\beta},P_{js}^{\gamma})\in \Sigma^{bad}_{ij}}\sum_{\{\alpha\in [\ell_2]: P_{ij}^{\alpha} \text{ relevant and }G_{ijk}^{\alpha,\beta,\gamma}\in \mathbf{F}_{err}\}}|K_3(G_{ijs}^{\alpha\beta\gamma})|\\
 &\geq (1-\mu) \sum_{(P_{is}^{\beta},P_{js}^{\gamma})\in \Sigma^{bad}_{ij}}\sum_{\{\alpha\in [\ell_2]: P_{ij}^{\alpha} \text{ relevant and }G_{ijk}^{\alpha,\beta,\gamma}\in \mathbf{F}_{err}\}}\frac{|V_i||V_j||V_s|}{\ell_1^3}\\
 &\geq \frac{1-\mu}{1+\mu} \sum_{(P_{is}^{\beta},P_{js}^{\gamma})\in \Sigma^{bad}_{ij}}\sum_{\{\alpha\in [\ell_2]: P_{ij}^{\alpha} \text{ relevant and }G_{ijk}^{\alpha,\beta,\gamma}\in \mathbf{F}_{err}\}}\frac{|V_s|}{\ell_1^2}|P_{ij}^{\alpha}|\\
&= \frac{1-\mu}{1+\mu} \sum_{(P_{is}^{\beta},P_{js}^{\gamma})\in \Sigma^{bad}_{ij}}\frac{|V_s|}{\ell_1^2}\sum_{\{\alpha\in [\ell_2]: P_{ij}^{\alpha} \text{ relevant and }G_{ijk}^{\alpha,\beta,\gamma}\in \mathbf{F}_{err}\}}|P_{ij}^{\alpha}|\\
&\geq  \frac{1-\mu}{1+\mu}  \sum_{(P_{is}^{\beta},P_{js}^{\gamma})\in \Sigma^{bad}_{ij}}\frac{|V_s|}{\ell_1^2}\mu^{1/8}|V_i||V_j|\\
&= \frac{(1-\mu)\mu^{1/8}|V_i||V_j|}{(1+\mu)\ell_1^2}\sum_{(P_{is}^{\beta},P_{js}^{\gamma})\in \Sigma^{bad}_{ij}}|V_s|.
\end{align*}
Combining this with (\ref{xerrpsi}) yields our desired inequality relating to $\Sigma^{bad}_{ij}$:
\begin{align}\label{al:sigmaij}
\sum_{(P_{is}^{\beta},P_{js}^{\gamma})\in \Sigma^{bad}_{ij}}|V_s|\leq \frac{2\mu^{1/8}(1+\mu)\ell_1^2}{(1-\mu)\mu^{1/8}}|V|\leq 3\mu^{1/16}\ell_1^2|V|,
\end{align}
where the last inequality is because $\mu$ is sufficiently small. This completes our definitions of the various error sets. 
 
Given $1\leq i,j\leq t$ with $(V_i,V_j)\notin \Psi$, we are ready to start building our auxiliary edge-colored bigraph associated to the pair $(V_i,V_j)$. We begin by defining some auxiliary sets similar to those in \cite{Terry.2022}. 
\begin{align*}
W_{ij}&=\{P_{ij}^\alpha: \alpha\in [\ell_2], P_{ij}^{\alpha}\text{ relevant, and  }P_{ij}^{\alpha}\cap X_{err}=\emptyset \},\\
U_{ij}&=\{(P_{is}^\beta, P_{js}^\gamma): s\in [t],P_{is}^\beta, P_{js}^\gamma\text{ both relevant}, (P_{is}^\beta, P_{js}^\gamma)\notin \Sigma_{ij}^{bad}\},\text{ and }\\
E^1_{ij}&=\{(P_{ij}^\alpha,P_{is}^\beta, P_{js}^\gamma)\in W_{ij}\times U_{ij}: G_{ijk}^{\alpha\beta\gamma}\in \mathbf{F}_1\}\\
E^0_{ij}&=\{(P_{ij}^\alpha,  P_{is}^\beta, P_{js}^\gamma)\in W_{ij}\times U_{ij}: G_{ijk}^{\alpha\beta\gamma}\in \mathbf{F}_0\}\\
E^2_{ij}&=\{(P_{ij}^\alpha,P_{is}^\beta, P_{js}^\gamma)\in W_{ij}\times U_{ij}: G_{ijk}^{\alpha\beta\gamma}\in \mathbf{F}_{err}\}.
\end{align*}
In \cite{Terry.2022}, the proof essentially worked with the edge-colored bigraph $\calG_{ij}=(W_{ij},U_{ij}; E_{ij}^0,E_{ij}^1,E_{ij}^2)$, applying Lemma \ref{lem:hausslercor} to it, and using the result to define a new partition of $V_i\times V_j$.  While we also work with $\calG_{ij}$ here, we will instead apply Lemma \ref{lem:hausslercor} to a slightly different edge-colored bigraph to account for the fact that our vertex partition may be unbalanced. In particular, we  define the following substitute for $U_{ij}$, obtained by ``fattening" elements of the form $(P_{is}^\beta, P_{js}^\gamma)$ with vertices from $V_s$.
\begin{align*}
U'_{ij}&=\{(P_{is}^\beta, P_{js}^\gamma, x) : (P_{is}^\beta, P_{js}^\gamma)\in U_{ij}, x\in V_s \}.
\end{align*}
We then define straightforward analogues of $E_{ij}^0,E_{ij}^1,E_{ij}^2$ as follows. 
\begin{align*}
R^1_{ij}&=\{(P_{ij}^\alpha,P_{is}^\beta, P_{js}^\gamma,x)\in W_{ij}\times U'_{ij}: G_{ijs}^{\alpha\beta\gamma}\in \mathbf{F}_1\}\\
R^0_{ij}&=\{(P_{ij}^\alpha,  P_{is}^\beta, P_{js}^\gamma,x)\in W_{ij}\times U'_{ij}: G_{ijs}^{\alpha\beta\gamma}\in \mathbf{F}_0\}\\
R^2_{ij}&=\{(P_{ij}^\alpha,P_{is}^\beta, P_{js}^\gamma,x)\in W_{ij}\times U'_{ij}: G_{ijs}^{\alpha\beta\gamma}\in \mathbf{F}_{err}\}.
\end{align*}
We will end up applying Lemma \ref{lem:hausslercor} to the edge-colored bigraph $\calG_{ij}'=(W_{ij}, U_{ij}';  R_{ij}^0,R_{ij}^1,R_{ij}^2)$ rather than $\calG_{ij}$.  This adjustment ensures the results of Lemma \ref{lem:hausslercor}   encode  meaningful information about $H$, despite the fact our vertex partition may not be equitable.  We first make a few observations about the sizes of $W_{ij}$ and $U_{ij}'$, when $(V_i,V_j)\notin \Psi$.  To this end, fix $(V_i,V_j)\notin \Psi$. Note  
\begin{align}\label{wij}
\sum_{P_{ij}^{\alpha}\in W_{ij}}|P_{ij}^{\alpha}|=|(V_i\times V_j)\setminus X_{err}|\geq (1-\mu^{1/8})|V_i||V_j|,
\end{align}
where the last inequality is because $(V_i,V_j)\notin \Psi$.  Consequently, using (\ref{wij}), the definition of $W_{ij}$, and (\ref{align:2}), we can deduce that 
\begin{align}\label{al:wij}
|W_{ij}|\geq \frac{(1-\mu^{1/8})|V_i||V_j|}{(1-\mu)|V_i||V_j|\ell_1^{-1}}\geq (1-\delta)\ell_1,
\end{align}
where the last inequality uses that $\mu$ is sufficiently small compared to $\delta$.   On the other hand, by (\ref{relbound}) and definition of $W_{ij}$,  $|W_{ij}|\leq \ell_1$.  We now turn to computing a lower bound for $U_{ij}'$.  By definition, 
\begin{align}\label{al:uij2}
|U_{ij}'|\geq \sum_{(P_{is}^{\beta},P_{js}^{\gamma})\in \Sigma_{ij}}|V_s|-\sum_{(P_{is}^{\beta},P_{js}^{\gamma})\in \Sigma^{bad}_{ij}}|V_s|.
\end{align}
Combining this with (\ref{uij3}) and  (\ref{al:sigmaij}), we have
\begin{align}\label{al:uij}
|U_{ij}'|\geq  (1-\delta)\ell_1^2|V| -3\mu^{1/16}\ell_1^2|V|\geq (1-2\delta)\ell_1^2|V|,
\end{align}
where the last inequality is because $\mu$ is sufficiently small compared to $\delta$.  

We will next obtain a structure theorem for $\calG'_{ij}$.  Given $v,v'\in W_{ij}$, write $v\sim v'\in W_{ij}$ if for each $\tau\in \{0,1\}$, $|R_{ij}^\tau(v)\Delta R_{ij}^\tau(v')|\leq \delta |U'_{ij}|$.  We check the hypotheses of Lemma \ref{lem:hausslercor} apply to $\calG'_{ij}$ with parameters $\delta$ and $\mu^{1/16}$.  First, we observe that by definition, any $E_{ij}^0/E_{ij}^1$-copy of $U_{bg}(k)$ in $\calG_{ij}$ would yield an $\e_1'$-nontrivial $(\mu^{1/2},\e_2'(\ell)^{1/12})$-encoding of $U_{bg}(k)$ in $(H,\calP)$.  Since (\ref{encode}) holds, Corollary \ref{cor:2.12}, implies the existence of any such encoding would contradict that $\VC_2(H)<k$.  Thus there exist no $E_{ij}^0/E_{ij}^1$-copies of $U_{bg}(k)$ in $\calG_{ij}$.  It is not difficult to see that by definition, this implies there is no $R_{ij}^0/R_{ij}^1$-copy of $U_{bg}(k)$ in $\calG_{ij}'$.  Second, the definition of $\delta$ and $\mu$ imply $\mu^{1/16}<c^{-2}(\delta/8)^{2k+2}$.    Third, we observe that for all $(P_{is}^{\beta},P_{js}^{\gamma},x)\in U'_{ij}$, since $(P_{is}^{\beta},P_{js}^{\gamma})\notin \Sigma^{bad}_{ij}$, 
\begin{align*}
\mu^{1/8}|V_i||V_j| &\geq \sum_{\{\alpha\in [\ell]: P_{ij}^{\alpha}\text{ relevant and }  G_{ijs}^{\alpha,\beta,\gamma}\in \mathbf{F}_{err}\}}|P_{ij}^{\alpha}|  \geq(1-\mu) |N_{R_{ij}^2}(P_{is}^{\beta},P_{js}^{\gamma},x)||V_i||V_j|\ell_1^{-1},
\end{align*}
where the second inequality uses (\ref{align:2}). Rearranging, this implies 
\begin{align}\label{r2}
|N_{R_{ij}^2}(P_{is}^{\beta},P_{js}^{\gamma},x)|\leq \frac{\mu^{1/8}\ell_1}{1-\mu}\leq  \frac{\mu^{1/8}|W_{ij}|}{(1-\mu)(1-\delta)}\leq \mu^{1/16}|W_{ij}|,
\end{align}
  where the second inequality uses (\ref{al:wij}), and the last uses that $\mu$ and $\delta$ are sufficiently small.   Since (\ref{r2})  holds for all $(P_{is}^{\beta},P_{js}^{\gamma},x)\in U'_{ij}$, we have $|R_{ij}^2|\leq \mu^{1/16}|W_{ij}||U'_{ij}|$.

We can now apply Lemma \ref{lem:hausslercor} to $\calG'_{ij}$. This tells us there exist $W_{ij}^0\subseteq W_{ij}$, an integer $m_{ij}\leq 2c(\delta/8)^{-k}$, and $x_{ij}^1,\ldots, x_{ij}^{m_{ij}}\in W_{ij}$ so that the following hold.
\begin{enumerate}[(i)]
\item $|W_{ij}^0|\leq \mu^{1/32}|W_{ij}|$,
\item For all $v\in W_{ij}\setminus W_{ij}^0$, there is $1\leq \alpha\leq m_{ij}$ so that $v\sim x_{ij}^{\alpha}$, 
\item For all $v\in W_{ij}\setminus W_{ij}^0$,  $|N_{R_{ij}^2}(v)|\leq\mu^{1/32}|U'_{ij}|$.  
\end{enumerate}
For each $1\leq u\leq m_{ij}$, let 
$$
W_{ij}^u=\{v\in W_{ij}\setminus W_{ij}^0: v\sim x_{ij}^u\text{ and for all $1\leq u'<u$}, v\nsim x_{ij}^{u'}\}.
$$
This defines a partition $W_{ij}=W_{ij}^0\cup W_{ij}^1\cup \ldots \cup W_{ij}^{m_{ij}}$.  Note that for each $1\leq u\leq m_{ij}$, $W_{ij}^u$ is non-empty since it contains $x_{ij}^u$.

One can from here conclude the proof as in \cite{Terry.2022} with minor adjustments. However, we instead choose  to end the proof here in a slightly different way, which we think is conceptually more clear.  We also choose to present our argument in an order which is easily adapted to the proof of Proposition \ref{prop:mainL}.  We first define a set $\Gamma\subseteq V^3$ consisting of triples in $V^3$ where the behavior of $\overline{E}$ is consistent with the structural information obtained from our application of Lemma \ref{lem:hausslercor} (recall $E=E(H)$, $\overline{E^1}=\overline{E}$, and $\overline{E^0}=V^3\setminus \overline{E}$).  
\begin{align*}
\Gamma:=\bigcup_{(V_i,V_j)\in \calP_1^2\setminus \Psi}\Big(\bigcup_{u=1}^{m_{ij}}\Big(\bigcup_{\tau\in \{0,1\}}\Big(\bigcup_{(P_{is}^{\beta},P_{js}^{\gamma})\in N_{E_{ij}^{\tau}}(x_{ij}^u)\cap N_{E_{ij}^{\tau}}(P_{ij}^{\alpha})}\overline{E^{\tau}}\cap K_3(G_{ijs}^{\alpha\beta\gamma})  \Big)\Big)\Big).
\end{align*}
We then define $\Gamma'$ to be the subset of $V^3$ consisting of those triples who lie in $\Gamma$ under any permutation:
$$
\Gamma'=\{(x_1,x_2,x_3): (x_{\sigma(1)},x_{\sigma(2)},x_{\sigma(3)})\in \Gamma\text{ each permutation }\sigma\in S_3\}.
$$
We will show $\Gamma'$ covers almost all of $V^3$.  While proving this requires some amount of computation (see Claim \ref{cl:gamma} below), working with $\Gamma'$ will streamline the end of our proof. 
 
\begin{claim}\label{cl:gamma}
 $|\Gamma'|\geq (1-6\delta^{1/2})|V|^3$.  
\end{claim}
\begin{proof}
Fix $(V_i,V_j)\in \calP_1^2\setminus \Psi$ and $1\leq u\leq m_{ij}$. Given $P_{ij}^{\alpha}\in W_{ij}^u$, $\tau\in \{0,1\}$, and $(P_{is}^{\beta},P_{js}^{\gamma})\in N_{E_{ij}^{\tau}}(x_{ij}^u)\cap N_{E_{ij}^{\tau}}(P_{ij}^{\alpha})$, we have by definition of $E_{ij}^{\tau}$ that $G_{ijs}^{\alpha\beta\gamma}\in \mathbf{F}_{\tau}$, and consequently,  using (\ref{align:1}), 
\begin{align*}
|\overline{E^{\tau}}\cap K_3(G_{ijs}^{\alpha\beta\gamma}) |\geq (1-\mu^{1/2})|K_3(G_{ijs}^{\alpha\beta\gamma}) |&\geq \frac{(1-\mu^{1/2})(1-\mu)|V_i||V_j||V_s|}{\ell_1^3}\\
&\geq \frac{(1-\delta)|V_i||V_j||V_s|}{\ell_1^3},
\end{align*}
where the last inequality is because $\mu$ is sufficiently small compared to $\delta$.  Thus for any $P_{ij}^{\alpha}\in W_{ij}^u$ and $\tau\in \{0,1\}$, we have the following.
\begin{align*}
\sum_{(P_{is}^{\beta},P_{js}^{\gamma})\in N_{E_{ij}^{\tau}}(x_{ij}^u)\cap N_{E_{ij}^{\tau}}(P_{ij}^{\alpha})}|\overline{E^{\tau}}\cap K_3(G_{ijs}^{\alpha\beta\gamma})|&\geq  (1-\delta) \sum_{(P_{is}^{\beta},P_{js}^{\gamma})\in N_{E_{ij}^{\tau}}(x_{ij}^u)\cap N_{E_{ij}^{\tau}}(P_{ij}^{\alpha})} \frac{|V_i||V_j||V_s|}{\ell_1^3}\\
&=  (1-\delta)\frac{|V_i||V_j|}{\ell_1^3} \sum_{(P_{is}^{\beta},P_{js}^{\gamma})\in N_{E_{ij}^{\tau}}(x_{ij}^u)\cap N_{E_{ij}^{\tau}}(P_{ij}^{\alpha})} |V_s|\\
&= (1-\delta)\frac{|V_i||V_j|}{\ell_1^3}|N_{R_{ij}^{\tau}}(x_{ij}^u)\cap N_{R_{ij}^{\tau}}(P_{ij}^{\alpha})|,
\end{align*}
where the last inequality is by  definition of $R_{ij}^\tau$. For any $P_{ij}^{\alpha}\in W_{ij}^u$, we know that $P_{ij}^{\alpha}\sim x_{ij}^u$, and consequently,  
$$
|N_{R_{ij}^{0}}(x_{ij}^u)\cap N_{R_{ij}^{0}}(P_{ij}^{\alpha})|+|N_{R_{ij}^{1}}(x_{ij}^u)\cap N_{R_{ij}^{1}}(P_{ij}^{\alpha})|\geq (1-2\delta) |U_{ij}'|.
$$
  Combining this with the above and (\ref{al:uij}),  we have
\begin{align*}
\sum_{\tau\in \{0,1\}} \sum_{(P_{is}^{\beta},P_{js}^{\gamma})\in N_{E_{ij}^{\tau}}(x_{ij}^u)\cap N_{E_{ij}^{\tau}}(P_{ij}^{\alpha})}|\overline{E^{\tau}}\cap K_3(G_{ijs}^{\alpha\beta\gamma})|&\geq (1-\delta)\frac{|V_i||V_j|}{\ell_1^3}(1-2\delta)|U'_{ij}|\\
&\geq (1-\delta)\frac{|V_i||V_j|}{\ell_1^3}(1-2\delta)^2\ell_1^2|V|\\
&\geq \frac{(1-2\delta)^3|V_i||V_j||V|}{\ell_1}.
\end{align*}
Combining with  (ii), we have
\begin{align*}
|\Gamma|\geq \sum_{(V_i,V_j)\in \calP_1^2\setminus \Psi}\sum_{u=1}^{m_{ij}}\frac{|W_{ij}^u|(1-2\delta)^3|V_i||V_j||V|}{\ell_1}&=\frac{(1-2\delta)^3|V|}{\ell_1}\sum_{(V_i,V_j)\in \calP_1^2\setminus \Psi}|V_i||V_j|\sum_{u=1}^{m_{ij}}|W_{ij}^u|\\
&=\frac{(1-2\delta)^3|V|}{\ell_1}\sum_{(V_i,V_j)\in \calP_1^2\setminus \Psi}|V_i||V_j|(|W_{ij}|- |W_{ij}^0|)\\
&\geq \frac{(1-2\delta)^3|V|}{\ell_1}\sum_{(V_i,V_j)\in \calP_1^2\setminus \Psi}|V_i||V_j|(1-\mu^{1/32})|W_{ij}|\\
&\geq \frac{(1-2\delta)^4|V|}{\ell_1}\sum_{(V_i,V_j)\in \calP_1^2\setminus \Psi}|V_i||V_j||W_{ij}|,
\end{align*}
where the last inequality is because $\mu$ is sufficiently small compared to $\delta$.  Finally, combining this with 
(\ref{al:wij}), and (\ref{al:psi}), we have
\begin{align*}
|\Gamma|\geq   \frac{(1-2\delta)^5|V|}{\ell_1}\sum_{(V_i,V_j)\in \calP_1^2\setminus \Psi}|V_i||V_j|(1-\delta)\ell_1&\geq  (1-2\delta)^6|V|\sum_{(V_i,V_j)\in \calP_1^2\setminus \Psi}|V_i||V_j|\\
&\geq (1-2\delta)^6(1-2\mu^{1/8}) |V|^3\\
&\geq  (1-2\delta)^7|V|^3\\
&\geq (1-\delta^{1/2})|V|^3,
\end{align*}
where the second to last inequality is because $\e_1'$ is sufficiently small compared to $\delta$ and the last inequality is because $\delta$ is sufficiently small. The stated inequality for $\Gamma'$ is immediate from the lower bound above for $\Gamma$, finishing our proof of this claim.
\end{proof}
 
 We will now define new partitions of $V_i\times V_j$ for each $(V_i,V_j)\notin \Psi$.  To this end, fix $(V_i,V_j)\in \calP_1^2\setminus \Psi$.  For  each $1\leq u\leq m_{ij}$, set
$$
\mathbf{W}_{ij}^u=\bigcup_{P_{ij}^{\alpha}\in W_{ij}^u}P_{ij}^{\alpha},
$$
and then define
$$
\mathbf{W}_{ij}^0=(V_i\times V_j)\setminus (\mathbf{W}_{ij}^1\cup \ldots \cup \mathbf{W}_{ij}^{m_{ij}}).
$$
It is straightforward to see these sets are sufficiently regular.
\begin{claim}\label{cl:wijreg}
For each $(V_i,V_j)\notin \Psi$ and $1\leq u\leq m_{ij}$, $\mathbf{W}_{ij}^u$ satsfies $\dev_2(\ell_1\e_2'(\ell)^{1/144}, |W_{ij}^u|/\ell_1)$ and thus $\dev_2(\e_2(\ell_2),|W_{ij}^u|/\ell_1)$.
\end{claim}
\begin{proof} For each $1\leq u\leq m_{ij}$, $\mathbf{W}_{ij}^u$ is by definition a union of at most $\ell_1$ many relevant $P_{ij}^{\alpha}$. By Fact \ref{fact:adding} and (\ref{align:2}), $(V_i,V_j;\mathbf{W}_{ij}^u)$ satisfies $\dev_2(\ell_1\e_2'(\ell)^{1/144}, |W_{ij}^u|/\ell_1)$.   By (\ref{devin}),  this implies $(V_i,V_j;\mathbf{W}_{ij}^u)$ satisfies $\dev_2(\e_2(\ell_2),|W_{ij}^u|/\ell_1)$.  
\end{proof}

Given $(V_i,V_j),(V_i,V_s),(V_j,V_s)\notin \Psi$ and $1\leq u\leq m_{ij}$, $1\leq v\leq m_{is}$, and $1\leq w\leq m_{js}$, let
$$
\mathbf{W}_{ijs}^{uvw}:=(V_i,V_j,V_s; \mathbf{W}_{ij}^{u}, \mathbf{W}_{is}^v, \mathbf{W}_{js}^w).
$$
In this case, Claim \ref{cl:wijreg},  Proposition \ref{prop:counting}, and (\ref{al:gcountin}) imply
\begin{align}\label{al:gcount}
|K_3(\mathbf{W}_{ijs}^{uvw})|=(1\pm \delta)\frac{|W_{ij}^u||W_{is}^v||W_{js}^w||V_i||V_j||V_s|}{\ell_1^3}.
\end{align}
 We now show that if such a $\mathbf{W}_{ijs}^{uvw}$ is mostly covered by $\Gamma'$, then it is homogeneous with respect to $H$.

\begin{claim}\label{gammahom}
Suppose  $(V_i,V_j),(V_i,V_s),(V_j,V_s)\in \calP_1^2\setminus \Psi$ and $1\leq u\leq m_{ij}$, $1\leq v\leq m_{is}$, and $1\leq m_{js}\leq w$.  Assume 
$$
|K_3(\mathbf{W}_{ijs}^{uvw})\cap \Gamma'|\geq (1-6\delta^{1/4})|K_3(\mathbf{W}_{ijs}^{uvw})|.
$$
Then $\mathbf{W}_{ijs}^{uvw}$ is $9\delta^{1/256}$-homogeneous with respect to $H$. 
\end{claim}
\begin{proof}
Fix $\mathbf{W}_{ijs}^{uvw}\in \Omega$ and set 
\begin{align*}
S_{err}&=\{(P_{ij}^{\alpha},P_{is}^{\beta},P_{js}^{\gamma})\in  W_{ij}^u\times W_{is}^v\times W_{js}^w: \Gamma'\cap K_3(G_{ijs}^{\alpha\beta\gamma})=\emptyset\},\\
S_0&=\{(P_{ij}^{\alpha},P_{is}^{\beta},P_{js}^{\gamma})\in (W_{ij}^u\times W_{is}^v\times W_{js}^w)\setminus S_{err}: G_{ijs}^{\alpha\beta\gamma}\in \mathbf{F}_0\},\text{ and }\\
S_1&=\{(P_{ij}^{\alpha},P_{is}^{\beta},P_{js}^{\gamma})\in (W_{ij}^u\times W_{is}^v\times W_{js}^w)\setminus S_{err}: G_{ijs}^{\alpha\beta\gamma}\in \mathbf{F}_1\}.
\end{align*}
Note that for any  $G\in \mathbf{F}_{err}$, the definition of $\Gamma'$ implies $K_3(G)\cap \Gamma'=\emptyset$. Consequently $S_0\cup S_1\cup S_2$ is a partition of $W_{ij}^u\times W_{is}^v\times W_{js}^w$.  We first show $S_{err}$ is not too large.   Indeed, since $|K_3(\mathbf{W}_{ijs}^{uvw})\cap \Gamma'|>(1-6\delta^{1/4})|K_3(\mathbf{W}_{ijs}^{uvw})|$,  
\begin{align*}
|\bigcup_{(P_{ij}^{\alpha},P_{is}^{\beta},P_{js}^{\gamma})\in S_{err}}K_3(G_{ijs}^{\alpha\beta\gamma})|&\leq 6\delta^{1/4}|K_3(\mathbf{W}_{ijs}^{uvw})|\leq \frac{6\delta^{1/4}(1+\delta)|W_{ij}^u||W_{is}^v||W_{js}^w||V_i||V_j||V_s|}{\ell_1^{3}},
\end{align*}
where the second inequality is by (\ref{al:gcount}).  On the other hand,  by (\ref{align:1}),
$$
|\bigcup_{(P_{ij}^{\alpha},P_{is}^{\beta},P_{js}^{\gamma})\in S_{err}}K_3(G_{ijs}^{\alpha,\beta,\gamma})|\geq(1- \mu)|S_{err}||V_i||V_j||V_s|\ell_1^{-3}.
$$
Combining these, we can conclude 
\begin{align}\label{serr}
|S_{err}|\leq \frac{6\delta^{1/4} (1+\delta)}{1-\mu}  |W_{ij}^u||W_{is}^v||W_{js}^w|\leq 12\delta^{1/4}  |W_{ij}^u||W_{is}^v||W_{js}^w|,
\end{align}
where the last inequality is because $\mu$ and $\delta$ are sufficiently small.  

We now claim $\min\{S_0,S_1\}\leq 8\delta^{1/256} |W_{ij}^u||W_{is}^v||W_{js}^w|$. Suppose towards a contradiction $\min\{S_0,S_1\}\geq 8\delta^{1/256} |W_{ij}^u||W_{is}^v||W_{js}^w|$.   Combining this assumption with (\ref{serr}), we have by Lemma \ref{lem:symmetry} that the following holds (possibly after relabeling):   There exists some $\alpha,\alpha',\beta,\gamma$ so that $(P_{ij}^{\alpha},P_{is}^{\beta},P_{js}^{\gamma})\in S_1$ and $(P_{ij}^{\alpha'},P_{is}^{\beta},P_{js}^{\gamma})\in S_0$.   Thus 
\begin{align}
\label{1}&\overline{E^1}\cap \Gamma'\cap K_3(G_{ijs}^{\alpha,\beta,\gamma}) \neq \emptyset\text{ and }\\
\label{2}& \overline{E^0}\cap \Gamma'\cap K_3(G_{ijs}^{\alpha',\beta,\gamma})\neq \emptyset.
\end{align}
 By definition of $\Gamma'$, (\ref{1}) implies $P_{is}^{\beta}P_{js}^{\gamma}\in N_{E_{ij}^1}(x_{ij}^u)\cap N_{E_{ij}^1}(P_{ij}^{\alpha})$, while (\ref{2}) implies  $P_{is}^{\beta}P_{js}^{\gamma}\in N_{E_{ij}^0}(x_{ij}^u)\cap N_{E_{ij}^0}(P_{ij}^{\alpha'})$.  But now we have a contradiction as $N_{E_{ij}^1}(x_{ij}^u)$ and $N_{E_{ij}^0}(x_{ij}^u)$ are disjoint.

Thus $\min\{S_0,S_1\}\leq 8\delta^{1/256} |W_{ij}^u||W_{is}^v||W_{js}^w|$. Combining with (\ref{serr}), this implies there is some $\tau\in \{0,1\}$ so that 
\begin{align}\label{stau}
|S_{\tau}|\geq (1-12\delta^{1/4}-8\delta^{1/256}) |W_{ij}^u||W_{is}^v||W_{js}^w|.
\end{align}
We now have
\begin{align*}
|\overline{E^{\tau}}\cap K_3(\mathbf{W}_{ijs}^{uvw})|&\geq \sum_{(P_{ij}^{\alpha},P_{is}^{\beta},P_{js}^{\gamma})\in S_{\tau}} |\overline{E^{\tau}}\cap K_3(G_{ijs}^{\alpha,\beta,\gamma})|\\
&\geq \sum_{(P_{ij}^{\alpha},P_{is}^{\beta},P_{js}^{\gamma})\in S_{\tau}}(1-\mu^{1/2})|K_3(G_{ijs}^{\alpha,\beta,\gamma})|\\
&\geq (1-\mu)(1-\mu^{1/2})\sum_{(P_{ij}^{\alpha},P_{is}^{\beta},P_{js}^{\gamma})\in S_{\tau}}|V_i||V_j||V_s|\ell_1^{-3}\\
&\geq(1-\mu)(1-\mu^{1/2})(1-2\delta^{1/4}-8\delta^{1/256}) |W_{ij}^u||W_{is}^v||W_{js}^w||V_i||V_j||V_s|\ell_1^{-3}\\
& \geq \frac{(1-\mu)(1-\mu^{1/2})(1-2\delta^{1/4}-8\delta^{1/256})}{1-\delta}|K_3(\mathbf{W}_{ijs}^{uvw})|\\
&\geq (1-9\delta^{1/256})|K_3(\mathbf{W}_{ijs}^{uvw})|,
\end{align*}
where first inequality is by definition of $S_{\tau}$, the second is by (\ref{align:1}), the third is by (\ref{stau}), the fourth is by (\ref{al:gcount}),  and the last is since $ \mu$ is sufficiently small compared to $\delta$, and $\delta$ is sufficiently small.  This finishes the proof of this claim.
 \end{proof}
 
 We now turn to defining  our final decomposition. First, set
$$
\ell^*=\max\{m_{ij}: 1\leq i,j\leq t, (V_i,V_j)\notin \Psi\},
$$ 
and note that by construction,
\begin{align}\label{ellstar}
\ell^*\leq 2c(\delta/8)^{-k}\leq O_k(\e_1^{-O_k(1)}),
\end{align}
where the second inequality is by definition of $\delta$ and $c$.  For each $(V_i,V_j)\in \Psi$, choose an arbitrary partition 
$$
V_i\times V_j=\mathbf{W}_{ij}^0\cup \ldots \cup \mathbf{W}_{ij}^{\ell^*},
$$
Recall we have already defined partitions $V_i\times V_j=\mathbf{W}_{ij}^0\cup \ldots \cup \mathbf{W}_{ij}^{\ell^*}$ for $(V_i,V_j)\notin \Psi$.  We now let $\calP^*=(\calP_1^*,\calP_2^*)$ where $\calP_1^*=\calP_1=\calQ_1$ and 
$$
\calP_2^*=\{\mathbf{W}_{ij}^u: 1\leq i,j\leq t, 0\leq u\leq \ell^*\}. 
$$
By definition, $\calP^*$ is a $(t,\ell^*+1)$-decomposition with the same vertex partition as $\calQ$, and where by (\ref{ellstar}), $\ell^*+1\leq \e_1^{-O_k(1)}$.  We just have left to show $\calP^*$ is sufficiently regular with respect to $H$.  Let
 \begin{align*}
 \Omega=&\{G\in \triads(\calP^*): |K_3(G)\cap \Gamma'|\geq (1-6\delta^{1/4})|K_3(G)|\}.  
 \end{align*}
Setting $Y=\bigcup_{G\in \Omega}K_3(G)$, we have by Claim \ref{cl:gamma} and Lemma \ref{lem:averaging}   that 
 \begin{align}\label{y1}
 |Y|\geq (1-\delta^{1/4})|V|^3\geq (1-\e_1)|V|^3,
 \end{align}
where the last inequality is because $\delta$ is sufficiently small compared to $\e_1$.  Thus, it suffices to show that any $G\in \Omega$ is $\dev_2(\e_1,\e_2(\ell^*+1))$-regular with respect to $H$.

To this end, fix $\mathbf{W}_{ijs}^{uvw}\in \Omega$. Since $K_3(\mathbf{W}_{ijs}^{uvw})\cap \Gamma'\neq \emptyset$, we have that $(V_i,V_j),(V_i,V_s),(V_j,V_s)\notin \Psi$ and $u,v,w>0$, and thus by Claim \ref{cl:wijreg},  each of $\mathbf{W}_{ij}^u,\mathbf{W}_{is}^v,\mathbf{W}_{js}^w$ have $\dev_2(\e_2(\ell_2))$ ( and consequently $\dev_2(\e_2(\ell^*+1))$  since   $\e_2$ is decreasing and $\ell^*+1\leq \ell_2$).  Further, by Claim \ref{gammahom}, $\mathbf{W}_{ijs}^{uvw}$ is $9\delta^{1/256}$-homogeneous with respect to $H$.   Thus, since $6(9\delta^{1/256}) <\e_1$, we have by Proposition \ref{prop:homimpliesrandome}  that $\mathbf{W}_{ijs}^{uvw}$ is $\dev_2(\e_1,\e_2(\ell^*+1))$-regular with respect to $H$, as desired.
\end{proof}

As a corollary, we have that whenever a hereditary $\calH$ has finite $\VC_2$-dimension, $L_{\calH}$ can be bounded above by a polynomial.  Moreover, this can be done while minimizing the size of the vertex partition with respect to polynomially related parameters.

\begin{corollary}\label{cor:polyub}
For all $k\geq 1$ there exist $\e_1^*,C>0$ and a polynomial $p(x,y)$ so that the following holds. Suppose $\calH$ is a hereditary $3$-graph property with $\VC_2(\calH)<k$.  Then for all $0<\e_1<\e^*_1$ and all $\e_2:\mathbb{N}\rightarrow (0,1]$ satisfying $\e_2(x)\leq p(\e_1,x^{-1})$, $L_{\calH}(\e_1,\e_2)\leq \e_1^{-C}$.
\end{corollary}
\begin{proof}
Let $p(x,y), r(x),q(x,y,z),s(x,y),\e_1^*$ be as in Proposition \ref{prop:main2}. Let $0<\e_1< \e_1^*$, and assume $\e_2:\mathbb{N}\rightarrow (0,1)$ satisfies $\e_2(x)\leq p(\e_1,x^{-1})$.    Let $L$ be any integer such that $\Psi(\e_1',\e_2',T_{\calH}(\e_1',\e_2'),L,\calH)$ holds, where $\e_1'=r(\e_1)$ and  $\e'_2(x)=q(\e_1,x^{-1},\e_2(\lceil s(\e_1^{-1},x)\rceil ))$.  Suppose $H$ is a sufficiently large element of $\calH$.  By definition of $L$,  there is a $\dev_{2,3}(\e_1',\e_2'(\ell))$-regular $(t,\ell, \e_1',\e_2'(\ell))$-decomposition $\calP$ for $H$ with $1\leq t\leq T_{\calH}(\e_1',\e_2')$ and $1\leq \ell\leq L$.  By Proposition \ref{prop:mainL}, there exists $\calP'$ a $\dev_{2,3}(\e_1,\e_2(\ell^*))$-regular $(t,\ell^*,\e_1,\e_2(\ell))$-decomposition for $H$ with $\ell^*\leq \e_1^{-O_k(1)}$.  This shows $L_{\calH}(\e_1,\e_2)\leq \e_1^{-O_k(1)}$.
\end{proof}

We next prove an analogue of Proposition \ref{prop:main2} for constant growth, Proposition \ref{prop:mainL} below.  The hypothesis of small $\VC_2$-dimension in Proposition \ref{prop:main2} is replaced by a bound on the $G$-dimension of all $G\in \textrm{Irr}(K)$ for some constant $K$.  The proof of Proposition \ref{prop:mainL} follows the proof of Proposition \ref{prop:main2}  very closely, to the extent that we will omit several steps in the argument that are identical.

\begin{proposition}\label{prop:mainL}
For all integers $C,m\geq 1$, there exist polynomials $p(x,y),r(x), q(x,y,z), s(x,y)$ and a constant $\e_1^*>0$ so that for all $0<\e_1<\e_1^*$, all non-increasing $\e_2:\mathbb{N}\rightarrow (0,1]$ satisfying $\e_2(x)< p(\e_1,x^{-1})$,  and all integers $t,\ell\geq 1$, the following holds.  Suppose $H$ is a sufficiently large $3$-graph satisfying the following.
\begin{enumerate}
\item For all bigraphs $G\in Irr(C+1)$, $H$ has $G$-dimension at most $m$. 
\item There exists a $\dev_{2,3}(\e_1',\e_2'(\ell))$-regular $(t,\ell,\e_1',\e_2'(\ell))$-decomposition $\calQ$ for $H$ where 
$$
\e_1'=r(\e_1), \text{ and }\e_2'(\ell)=q(\e_1,\ell^{-1},\e_2(\lceil s(\e_1^{-1},\ell)\rceil ))).
$$
Then there exists a $\dev_{2,3}(\e_1,\e_2(C))$-regular $(t,C,\e_1,\e_2(C))$-decomposition for $H$ with the same vertex partition as $\calQ$.  
\end{enumerate}
\end{proposition}
\begin{proof}
Define $k=(C+1)m$.  We begin by picking parameters in an almost identical fashion to the proof of Proposition \ref{prop:mainL}.  Let $c=c(C+1)$ be from Corollary \ref{cor:haussler}. Let $\mu_1>0$ and $p_1(x,y,z)$ be as in Proposition \ref{prop:encodingtoblowup} for the integer $2^{k}$. Let $p_2(x,y)$ and $D=D(k)\geq 3$ be as in Corollary \ref{cor:suffvc2} for $k$, and let $\mu_2=2^{-D}$. Let $\mu_3>0$ and $p_3(x,y)$ be from Lemma \ref{lem:equitable}, and set  $p_4(x,y)=2x^3y^3$.  Let $0<\mu_5<1$ be sufficiently small so that $(1-2\mu_5)^7>1-\mu_5^{1/2}$, and $\frac{1+\mu_5}{1-\mu_5}<5/4$, and define $\e_1^*=\min\{2^{-16},\mu_1,\mu_2,\mu_3,\mu_4,\mu^5_5\}$.  

Fix $0<\e_1<\e_1^*$.  Similar to the proof of Proposition \ref{prop:main2}, we let  
$$
\delta=\Big(\frac{\e_1}{60}\Big)^{256},\text{ } \mu=\Big(c^{-2}\Big(\frac{\delta}{8}\Big)^{2(C+1)+4}\Big)^{64}\text{ and }\e_1'=\mu^D.
$$
Clearly there exists a polynomial $r(x)$, depending only on $k$, so that $\e_1'=r(\e_1)$.  Observe that by definition,
$$
\e_1'<\mu<\delta<\e_1.
$$
We now let $p(x,y)$ be the polynomial, depending only on $k$, satisfying
$$
p(\e_1,y)=\Big(\frac{\e_1'}{4y^3}\Big)^{100}p_1(\mu^{1/2},\e_1',y)p_2(\e_1',y)p_3(\e_1',y).
$$
Fix a non-increasing $\e_2:\mathbb{N}\rightarrow (0,1]$ satisfying $\e_2(y)< p(\e_1,y^{-1})$ for all $y\in \mathbb{N}$.  Define $\e_2':\mathbb{N}\rightarrow (0,1]$ by setting
\begin{align*}
\e_2'(y)=\Big(\frac{\e_1'\mu\delta\e_2(C\lceil 6yp_4(\e_1',y^{-1})^{-1}\rceil )}{6yp_4(\e_1',y^{-1})^{-1}}\Big)^{288}.
\end{align*}
Observe that there exist $q(x,y,z)$ and $s(x,y)$  polynomials, depending only on $k$, so that  $\e_2'(y)=q(\e_1, y^{-1}, \e_2(\lceil s(\e_1^{-1},y))\rceil )$.

Fix integers $t,\ell\geq 1$, and suppose $H$ is a sufficiently large $3$-graph so that for all $G\in Irr(C+1)$, $H$ has $G$-dimension at most $m$.  Assume there exists a $\dev_{2,3}(\e_1',\e_2'(\ell))$-regular $(t,\ell_1,\e_1',\e_2'(\ell))$-decomposition $\calQ$  for $H$. To ease notation, throughout the proof we let $V=V(H)$ and $E=E(H)$.   Note that our assumption that the $G$-dimension of $H$ is at most $m$ for all $G\in Irr(C+1)$, and our definition of $k$, imply  $\VC_2(H)<k$. Hence, since $\mu=(\e_1')^{1/D}$, Corollary \ref{cor:suffvc2} implies that $\calQ$ is $\mu$-homogeneous with respect to $H$, in the sense of Definition \ref{def:homdec}.  We now define
$$
\ell_1= \lceil 2(\e'_1)^{-3}\ell^3\rceil=\lceil p_4((\e_1')^{-1},\ell)^{-1}\rceil\text{ and }\ell_2=3\ell \lceil 2(\e'_1)^{-3}\ell^3\rceil=3\ell \lceil p_4((\e_1')^{-1},\ell)^{-1}\rceil.
$$
For easy reference later in the proof, we now list several inequalities similar to those in Proposition \ref{prop:main2}, which hold by definitions of $\e_2',\ell_1,\ell_2$,  our assumption on $\e_2$.
 \begin{align}
&(\ell_1+1)(\ell_1^{-1}-\e_2'(\ell)^{1/12})>1\label{relboundinb}\\
&4(\e_2'(\ell)^{1/12})^{1/4}<\mu \label{align:1inb}\\
&\e_2'(\ell)^{1/12}<\mu\ell_1^{-1} \label{align:2inb}\\
&\e_2'(\ell)^{1/12}\leq p_1(\mu^{1/2},\e_1',\ell_2^{-1})\label{encode}\\
&\ell_1\e_2'(\ell)^{1/144}<\e_2(\ell_2)\label{devinb}\\
&4\e_2(\ell_2)^{1/4}<\delta/\ell_1^3 \label{al:gcountingb}.
\end{align}
We will use an additional list of inequalities in this proof which are not needed in Proposition \ref{prop:main2}.   To ease notation, for these we define one more auxiliary paramter.
$$
\e:=\ell_1\e_2'(\ell)^{1/144}.
$$
The following  additional inequalities  hold by definitions of $\e_2,\e_2',\ell_1,\ell_2,\e$, and will be used at the end of this proof, where it diverges from Proposition \ref{prop:main2}.
\begin{align}
&2\e\leq \frac{3\delta^{1/2}(1-\e)}{\ell_1}\label{w0bdb}\\
&2\e<\ell_1^{-1}-\e\label{w0bd1b}\\
&\e<\delta^{1/2}\label{w0bd2b}\\
&4\ell_1^4\e <1\label{w0bd3b}\\
&\e<\frac{1}{8\ell_1}\label{w0bd4b}\\
&\e+4\ell_1^2\e^{1/144}<\delta^{1/2}\Big(\frac{1}{\ell_1}-\e_2(\ell_2)\Big)\label{w0bd5b}\\
&4(5\ell_1^2\e^{1/144})^{1/4}<\mu\Big(\frac{1}{4\ell_1^2}-\e\Big)\label{end1b}\\
&\e<\frac{4\delta^{1/2}(\ell_1^{-1}-\mu)^3}{(1+\mu)}\label{end2.5b}\\
&4\ell_1^2\e^{1/144}<\e_2(C).\label{end2b}
\end{align}
%\begin{align}
%&(\ell_1+1)(\ell_1^{-1}-\e_2'(\ell)^{1/12})>1\label{relboundin}\\
%&4(\e_2'(\ell)^{1/12})^{1/4}<\mu \label{align:1in}\\
%&\e_2'(\ell)^{1/12}<\mu\ell_1^{-1} \label{align:2in}\\
%&\ell_1\e_2'(\ell)^{1/144}<\e_2(\ell_2)\label{devin}\\
%&4\e_2(\ell_2)^{1/4}<\delta/\ell_1^3 \label{al:gcountin}.
%\end{align}
By Lemma \ref{lem:equitable} (applied with polynomial $p_4(x,y)$), there exist
$$
 \calP_1= \{V_1,\ldots, V_t\}\text{ and }\calP_2=\{P_{ij}^{\alpha}: 1\leq i,j\leq t,1\leq \alpha\leq \ell_2\},
$$
so that $\calP=(\calP_1,\calP_2)$ is a $(t,\ell_2,4\e_1',\e_2'(\ell)^{1/12})$-decomposition such that the following hold. 
\begin{enumerate}[(a)]
\item $\calP_1=\calQ_1$ and $\calP_2\preceq \calQ_2$, 
\item At least $(1-4\e'_1)|V|^2$ pairs from $V^2$ lie in some $P_{ij}^{\alpha}\in \calP_2$ such that $(V_i,V_j; P_{ij}^{\alpha})$ satisfies $\dev_2(\e_2'(\ell)^{1/12},\ell_1^{-1})$ and where $\min\{|V_i|,|V_j|\}\geq \e'_1|V|/t$.  
\end{enumerate}
We say  $P_{ij}^{\alpha}\in \calP_2$ is \emph{relevant} if it satisfies the conditions in (b), i.e. if $(V_i,V_j; P_{ij}^{\alpha})$ satisfies $\dev_2(\e_2'(\ell)^{1/12},\ell_1^{-1})$ and $\min\{|V_i|,|V_j|\}\geq \e'_1|V|/t$.  As in the proof of Proposition \ref{prop:main2}, we can use the definition of relevant and (\ref{relboundinb}) to show that for all $1\leq i,j\leq t$, 
 \begin{align}\label{relboundb}
 |\{\alpha\in [\ell_1]: P_{ij}^{\alpha}\text{ is relevant }\}|\leq \ell_1.
 \end{align}
%Indeed, this follows from the fact that each relevant $P_{ij}^{\alpha}$ has density $\ell_1^{-1}\pm \e_2'(\ell)^{1/12}$ in $V_i\times V_j$, and $\e_2'(\ell)^{1/12}$ is sufficiently small compared to $\ell_1^{-1}$ (see (\ref{relboundin})).
 
We also note that by (a),  the partition $V^3=\bigcup_{G\in \triads(\calP)}K_3(G)$ refines the partition $V^3=\bigcup_{G\in \triads(\calQ)}K_3(G)$. Consequently, since $\calQ$ is $\mu$-homogeneous with respect to $H$, Lemma \ref{lem:averaging} implies $\calP$ is $\mu^{1/2}$-homogeneous with respect to $H$.  

We now set up some notation for triads of $\calP$.  For each $1\leq i,j,s\leq t$ and $1\leq \alpha,\beta,\gamma\leq \ell_2$,  set $G_{ijs}^{\alpha,\beta,\gamma}=(V_i, V_j,V_s; P_{ij}^\alpha, P_{js}^\beta, P_{is}^\gamma)$ and  define
$$
d_{ijs}^{\alpha,\beta,\gamma}=\frac{|\overline{E}\cap K_3(G_{ijs}^{\alpha,\beta,\gamma})|}{|K_3(G_{ijs}^{\alpha,\beta,\gamma})|}.
$$
We will call a triad from $\calP$ \emph{relevant} if it has the form $G_{ijs}^{\alpha,\beta,\gamma}$   where each of  $P_{ij}^\alpha, P_{js}^\beta, P_{is}^\gamma$ are relevant.   As in the proof of Proposition \ref{prop:main2},  we have by Proposition \ref{prop:counting} and (\ref{align:1inb}) that for any relevant $G_{ijs}^{\alpha,\beta,\gamma}$,
\begin{align}\label{align:1b}
|K_3(G_{ijs}^{\alpha,\beta,\gamma})|=\frac{(1\pm \mu)|V_i||V_j||V_s|}{\ell_1^3}.
\end{align}
and, for any relevant $P_{ij}^{\alpha}$, by (\ref{align:2inb}),
\begin{align}\label{align:2b}
|P_{ij}^{\alpha}|=\frac{|V_i||V_j|(1\pm \mu)}{\ell_1}.
\end{align}
We now define a partition of $\triads(\calP)$.   
\begin{align*}
\mathbf{F}_{err}&=\{G_{ijs}^{\alpha,\beta,\gamma} \in \triads(\calP):G_{ijs}^{\alpha,\beta,\gamma}\text{ is not relevant or  }d_{ijs}^{\alpha,\beta,\gamma}\in (\mu^{1/2},1-\mu^{1/2})\},\\
\mathbf{F}_1&=\{G_{ijs}^{\alpha,\beta,\gamma} \in \triads(\calP)\setminus \mathbf{F}_{err}: d_{ijs}^{\alpha,\beta,\gamma}\geq 1- \mu^{1/2}\},\text{ and }\\
\mathbf{F}_0&=\{G_{ijs}^{\alpha,\beta,\gamma} \in \triads(\calP)\setminus \mathbf{F}_{err}:  d_{ijs}^{\alpha,\beta,\gamma}\leq  \mu^{1/2}\}.
\end{align*}

Our strategy is based on the proof of Proposition \ref{prop:main2} (which in turn is based on \cite{Terry.2022}).  In particular, we will define auxiliary edge-colored bigraphs built from the components of $\calP$, to which we will then apply Corollary \ref{cor:haussler}.  Following closely the proof of Proposition \ref{prop:main2}, we first define and bound a series error sets.  To start with, we let  
\begin{align*}
F_{err}=\bigcup_{G\in \mathbf{F}_{err}}K_3(G).
\end{align*}
Following the steps exactly as in the proof of Proposition \ref{prop:main2} (see the proof of inequality (\ref{Ferr})), we have by (b) and the homogeneity of $\calP$ that 
\begin{align*}
|F_{err}|\leq   2\mu^{1/2}|V|^3.
\end{align*}
Next, we define the set of vertex pairs which intersect too many triples from $F_{err}$.
$$
X_{err}=\{(x,y)\in V^2: |N_{F_{err}}(x,y)|\geq \mu^{1/4} |V|\}.
$$
Exactly as in the proof of Proposition \ref{prop:main2} (see the proof of inequality (\ref{xerr})), we have by the upper bound on $F_{err}$ and definition of $X_{err}$ that
\begin{align}\label{xerrb}
|X_{err}|\leq 2\mu^{1/4}|V|^2.
\end{align}
We now define the set of pairs $(V_i,V_j)$ where $V_i\times V_j$ contains too may elements of $X_{err}$.
\begin{align*}
\Psi&=\{(V_i,V_j)\in \calP_1^2: |X_{err}\cap (V_i\times V_j)|\geq \mu^{1/8}|V_i||V_j|\}.
\end{align*}
Exactly as in the proof of Proposition \ref{prop:main2} (see the proof of (\ref{al:psi})) we have 
\begin{align}\label{al:psib}
\sum_{(V_i,V_j)\in \Psi} |V_i||V_j|\leq 2\mu^{1/8}|V|^2.
\end{align}
Given $1\leq i,j\leq t$, set 
 \begin{align*}
\Sigma_{ij}&=\Big\{(P_{is}^{\beta},P_{js}^{\gamma}): s\in [t], \beta,\gamma\in [\ell_2], \text{$P_{is}^{\beta},P_{js}^{\gamma}$ both relevant}\Big\},\text{ and }\\
 \Sigma^{bad}_{ij}&=\Big\{(P_{is}^{\beta},P_{js}^{\gamma})\in \Sigma_{ij}: \sum_{\{\alpha\in [\ell_2]: P_{ij}^{\alpha} \text{ relevant and }G_{ijk}^{\alpha,\beta,\gamma}\in \mathbf{F}_{err}\}}|P_{ij}^{\alpha}|\geq \mu^{1/8} |V_i||V_j|\Big\}.
 \end{align*}
Exactly as in the proof of Proposition \ref{prop:main2} (see the proofs of (\ref{uij3}) and (\ref{al:sigmaij})), we have 
\begin{align}
\label{uij3b}\sum_{(P_{is}^{\beta},P_{js}^{\gamma})\in \Sigma_{ij}}|V_s|&\geq (1-\delta)\ell_1^2|V|\text{ and }\\
\label{al:sigmaijb}\sum_{(P_{is}^{\beta},P_{js}^{\gamma})\in \Sigma^{bad}_{ij}}|V_s|&\leq   3\mu^{1/16}\ell_1^2|V|.
\end{align}
This completes our definitions and bounds of the various error sets. 
 
Given $1\leq i,j\leq t$ with $(V_i,V_j)\notin \Psi$, we are ready to start building our auxiliary edge-colored bigraph associated to the pair $(V_i,V_j)$. We will work with an initial edge-colored bigraph $\calG_{ij}=(W_{ij},U_{ij}; E_{ij}^0,E_{ij}^1,E_{ij}^2)$, where we define  
\begin{align*}
W_{ij}&=\{P_{ij}^\alpha: \alpha\in [\ell_2], P_{ij}^{\alpha}\text{ relevant, and  }P_{ij}^{\alpha}\cap X_{err}=\emptyset \},\\
U_{ij}&=\{(P_{is}^\beta, P_{js}^\gamma): s\in [t],P_{is}^\beta, P_{js}^\gamma\text{ both relevant}, (P_{is}^\beta, P_{js}^\gamma)\notin \Sigma_{ij}^{bad}\},\text{ and }\\
E^1_{ij}&=\{(P_{ij}^\alpha,P_{is}^\beta, P_{js}^\gamma)\in W_{ij}\times U_{ij}: G_{ijk}^{\alpha,\beta,\gamma}\in \mathbf{F}_1\}\\
E^0_{ij}&=\{(P_{ij}^\alpha,  P_{is}^\beta, P_{js}^\gamma)\in W_{ij}\times U_{ij}: G_{ijk}^{\alpha,\beta,\gamma}\in \mathbf{F}_0\}\\
E^2_{ij}&=\{(P_{ij}^\alpha,P_{is}^\beta, P_{js}^\gamma)\in W_{ij}\times U_{ij}: G_{ijk}^{\alpha,\beta,\gamma}\in \mathbf{F}_{err}\}.
\end{align*}
As in the proof of Proposition \ref{prop:main2}, we will also work with $\calG_{ij}':=(W_{ij},U'_{ij}; R_{ij}^0,R_{ij}^1,R_{ij}^2)$, where 
\begin{align*}
U'_{ij}&=\{(P_{is}^\beta, P_{js}^\gamma, x) : (P_{is}^\beta, P_{js}^\gamma)\in U_{ij}, x\in V_s \},\\
R^1_{ij}&=\{(P_{ij}^\alpha,P_{is}^\beta, P_{js}^\gamma,x)\in W_{ij}\times U'_{ij}: G_{ijs}^{\alpha,\beta,\gamma}\in \mathbf{F}_1\}\\
R^0_{ij}&=\{(P_{ij}^\alpha,  P_{is}^\beta, P_{js}^\gamma,x)\in W_{ij}\times U'_{ij}: G_{ijs}^{\alpha,\beta,\gamma}\in \mathbf{F}_0\}\\
R^2_{ij}&=\{(P_{ij}^\alpha,P_{is}^\beta, P_{js}^\gamma,x)\in W_{ij}\times U'_{ij}: G_{ijs}^{\alpha,\beta,\gamma}\in \mathbf{F}_{err}\}.
\end{align*}
Exactly as in the proof of Proposition \ref{prop:main2},  we have the following size estimates on  $W_{ij}$ and $U_{ij}'$ (see the proofs of (\ref{al:wij}) and (\ref{al:uij})).
\begin{align}
\label{al:wijb}|W_{ij}|&\geq  (1-\delta)\ell_1\text{ and }\\
\label{al:uijb}|U_{ij}'|&\geq  (1-2\delta)\ell_1^2|V|.
\end{align}

We will next obtain a structure theorem for $\calG'_{ij}$.  Given $v,v'\in W_{ij}$, write $v\sim v'\in W_{ij}$ if for each $\tau\in \{0,1\}$, $|R_{ij}^\tau(v)\Delta R_{ij}^\tau(v')|\leq \delta |U'_{ij}|$.  We check the hypotheses of Corollary \ref{cor:haussler} apply to $\calG'_{ij}$ with parameters $\delta$ and $\mu^{1/16}$.   First, observe that any $E_{ij}^0/E_{ij}^1$-copy of an element $G\in Irr(C+1)$ in $\calG_{ij}$ would yield an $\e_1'$-nontrivial $(\e_2'(\ell)^{1/12},\mu^{1/2})$-encoding of $G$ in $(H,\calP)$ (see Definition \ref{def:encodingR}).  By Proposition \ref{prop:encodingtoblowup} (which applies here by (\ref{encode})), the existence of any such encoding would imply $H$ has $G$-dimension at least $m$, a contradiction. Thus, there exist  no $E_{ij}^0/E_{ij}^1$-copies of any element of $\textrm{Irr}(C+1)$ in $\calG_{ij}$.  This immediately implies there is no $R_{ij}^0/R^1$-copies of any element of $\textrm{Irr}(C+1)$ in $\calG_{ij}'$.   Second, the definition of $\delta$ and $\mu$ imply $\mu^{1/16}<c^{-2}(\delta/8)^{2(C+1)+4}$. Third,  the exact same argument as in the proof of Proposition \ref{prop:main2} shows $|R_{ij}^2|\leq \mu^{1/16}|W_{ij}||U'_{ij}|$ (see the proof of (\ref{r2})).

We can now apply Corollary \ref{cor:haussler} to $\calG'_{ij}$. This tells us there exist $W_{ij}^0\subseteq W_{ij}$, an integer $m_{ij}\leq C$, and $x_{ij}^1,\ldots, x_{ij}^{m_{ij}}\in W_{ij}$ so that the following hold.
\begin{enumerate}[(i)]
\item $|W_{ij}^0|\leq \mu^{1/32}|W_{ij}|$,
\item For all $v\in W_{ij}\setminus W_{ij}^0$, there is $1\leq \alpha\leq m_{ij}$ so that $v\sim x_{ij}^{\alpha}$, 
\item For all $v\in W_{ij}\setminus W_{ij}^0$,  $|N_{R_{ij}^2}(v)|\leq\mu^{1/32}|U'_{ij}|$.  
\end{enumerate}
For each $1\leq u\leq m_{ij}$, let 
$$
W_{ij}^u=\{v\in W_{ij}\setminus W_{ij}^0: v\sim x_{ij}^u\text{ and for all $1\leq u'<u$}, v\nsim x_{ij}^{u'}\}.
$$
This defines a partition $W_{ij}=W_{ij}^0\cup W_{ij}^1\cup \ldots \cup W_{ij}^{m_{ij}}$.  Note that for each $1\leq u\leq m_{ij}$, $W_{ij}^u$ is non-empty since it contains $x_{ij}^u$.

Exactly as in the proof of Proposition \ref{prop:main2}, we define the following subset of $V^3$ consisting of triples where $E=E(H)$ ``agrees" with what is predicted by the outcome of our application of Corollary \ref{cor:haussler}. 
\begin{align*}
\Gamma:=\bigcup_{(V_i,V_j)\in \calP_1^2\setminus \Psi}\Big(\bigcup_{u=1}^{m_{ij}}\Big(\bigcup_{\tau\in \{0,1\}}\Big(\bigcup_{(P_{is}^{\beta},P_{js}^{\gamma})\in N_{E_{ij}^{\tau}}(x_{ij}^u)\cap N_{E_{ij}^{\tau}}(P_{ij}^{\alpha})}\overline{E^{\tau}}\cap K_3(G_{ijs}^{\alpha,\beta,\gamma})  \Big)\Big)\Big).
\end{align*}
We then let $\Gamma'$ below be the set of tuples where every permutation is in $\Gamma$.
$$
\Gamma':=\{(x_1,x_2,x_3): (x_{\sigma(1)},x_{\sigma(2)},x_{\sigma(3)})\in \Gamma\text{ each permutation }\sigma\in S_3\}.
$$
The exact same proof as in Proposition \ref{prop:main2} shows the following (see the proof of Claim \ref{cl:gamma}).
 
\begin{claim}\label{cl:gamma1}
 $|\Gamma'|\geq (1-6\delta^{1/2})|V|^3$.  
\end{claim}

We now define a new partition of $V_i\times V_j$ in the case where $(V_i,V_j)\notin \Psi$.  For each $1\leq i,j\leq t$ such that $(V_i,V_j)\notin \Psi$, and each $1\leq u\leq m_{ij}$, set
$$
\mathbf{W}_{ij}^u=\bigcup_{P_{ij}^{\alpha}\in W_{ij}^u}P_{ij}^{\alpha},
$$
and set
$$
\mathbf{W}_{ij}^0=(V_i\times V_j)\setminus (\mathbf{W}_{ij}^1\cup \ldots \cup \mathbf{W}_{ij}^{m_{ij}}).
$$
Note that while $\mathbf{W}_{ij}^0$ contains $\bigcup_{P_{ij}^{\alpha}\in W_{ij}^0}P_{ij}^{\alpha}$ as a subset, $\mathbf{W}_{ij}^0$ may be strictly larger (namely when there exist $P_{ij}^{\alpha}$ which are not in $W_{ij}$). 

We next show that these yield $\dev_2$-regular bigraphs (Claim \ref{cl:wijregb} below).  It is at this point we begin to use the additional parameter defined at the beginning of the proof, which we now recall.
$$
\e=\ell_1\e_2'(\ell)^{1/144}.
$$

\begin{claim}\label{cl:wijregb}
Suppose $(V_i,V_j)\notin \Psi$.  For all $1\leq u\leq m_{ij}$, $(V_i,V_j;\mathbf{W}_{ij}^u)$ satisfies $\dev_2(\e,|W_{ij}^u|/\ell_1)$, and consequently, $\dev_2(\e_2(\ell_2), |W_{ij}^u|/\ell_1)$.

Further, $(V_i,V_j; \mathbf{W}_{ij}^0)$ satisfies $\dev_2(\e,\frac{\ell_1-N_{ij}}{\ell_1})$,  where
$$
N_{ij}=|W_{ij}^1|+\ldots+|W_{ij}^{m_{ij}}|.
$$ 
\end{claim}
\begin{proof}
For all $1\leq u\leq m_{ij}$,   $\mathbf{W}_{ij}^u$ is a union of $|W_{ij}^u|$-many relevant $P_{ij}^{\alpha}$. By Fact \ref{fact:adding}, (\ref{align:2b}),  since $|W_{ij}^u|\leq \ell_1$, and by definition of $\e$, we can thus conclude $(V_i,V_j;\mathbf{W}_{ij}^u)$ satisfies $\dev_2(\e, |W_{ij}^u|/\ell_1)$.   By (\ref{devinb}),  this implies $(V_i,V_j;\mathbf{W}_{ij}^u)$ satisfies $\dev_2(\e_2(\ell_2),|W_{ij}^u|/\ell_1)$. 

We now observe that $(V_i\times V_j)\setminus \mathbf{W}_{ij}^0$ is a union of $N_{ij}$-many relevant $P_{ij}^{\alpha}$, and thus by the same reasoning as above, $(V_i,V_j; (V_i\times V_j)\setminus \mathbf{W}_{ij}^0)$ satisfies $\dev_2(\e,\frac{N_{ij}}{\ell_1})$.  By Lemma \ref{lem:complimentbi}, $(V_i,V_j; \mathbf{W}_{ij}^0)$   satisfies $\dev_2(\e,\frac{\ell_1-N_{ij}}{\ell_1})$.  This finishes the proof of this claim.
\end{proof}

Given  $(V_i,V_j), (V_i,V_s),(V_j,V_s)\in \calP_1^2\setminus \Psi$ and $1\leq u\leq m_{ij}$, $1\leq v\leq m_{is}$, $1\leq w\leq m_{js}$, define 
$$
\mathbf{W}_{ijs}^{uvw}:=(V_i, V_j, V_s;\mathbf{W}_{ij}^{u}, \mathbf{W}_{is}^v, \mathbf{W}_{js}^w).
$$
As in Proposition \ref{prop:main2}, we have that for each such $i,j,s$ and $u,v,w$,   Claim \ref{cl:wijregb}, Proposition \ref{prop:counting}, and (\ref{al:gcountingb}) imply 
\begin{align}\label{al:gcountb}
|K_3(\mathbf{W}_{ijs}^{uvw})|=(1\pm \delta)\frac{|W_{ij}^u||W_{is}^v||W_{js}^w||V_i||V_j||V_s|}{\ell_1^3}.
\end{align}

Using exactly the same argument as in Proposition \ref{prop:main2}, we can show that given any such $\mathbf{W}_{ijs}^{uvw}$, if  $K_3(\mathbf{W}_{ijs}^{uvw})$ is mostly covered by $\Gamma'$, then $\mathbf{W}_{ijs}^{uvw}$ is homogeneous with respect to $H$ (see the proof of Claim \ref{gammahom}).

\begin{claim}\label{gammahom1}
Suppose $(V_i,V_j), (V_i,V_s),(V_j,V_s)\in \calP_1^2\setminus \Psi$, and $1\leq u\leq m_{ij}$, $1\leq v\leq m_{is}$,  and $1\leq w\leq m_{js}$.  If 
$$
|K_3(K_3(\mathbf{W}_{ijs}^{uvw})\cap \Gamma'|\geq (1-6\delta^{1/4})K_3(G)|
$$
the $K_3(\mathbf{W}_{ijs}^{uvw})$ is  $9\delta^{1/256}$-homogeneous with respect to $H$.
\end{claim}

Our next goal is to define our end decomposition.   This requires more care than in Proposition \ref{prop:main2}, as here we need to achieve an exact upper bound, $C$, on the pairs complexity, whereas in Proposition \ref{prop:main2} we were happy with any polynomial upper bound.   For each $(V_i,V_j)\notin\Psi$ we have thus far defined a partition $V_i\times V_j=\mathbf{W}_{ij}^0\cup \mathbf{W}_{ij}^1\cup \ldots \cup \mathbf{W}_{ij}^{m_{ij}}$ where $m_{ij}\leq C$.  However, this is a partition possibly into $C+1$ parts, rather than the desired $C$, due to the presence of $\mathbf{W}_{ij}^0$.  We will fix this by redistributing the elements in $\mathbf{W}_{ij}^0$ to the other parts in the partition.   The following claim will play a crucial role in this.

\begin{claim}\label{cl:sij}
Given  $(V_i,V_j)\notin \Psi$, there exists a partition 
$$
\mathbf{W}_{ij}^0=\mathbf{W}_{ij}^0(1)\cup \ldots \cup \mathbf{W}_{ij}^0(m_{ij})
$$
 so that for each $1\leq u\leq m_{ij}$,  one of the following hold.
 \begin{itemize}
 \item $|\mathbf{W}_{ij}^0(u)|\leq \e$, or
 \item $(V_i,V_j; \mathbf{W}_{ij}^0(u))$ satisfies $\dev_2(4\ell^2_1\e^{1/144})$ and 
 $$
 (\ell_1^{-2}-4\ell_1^2\e^{1/144})|V_i||V_j|\leq |\mathbf{W}_{ij}^0(u)|\leq 4\delta^{1/2}|\mathbf{W}_{ij}^u|.
 $$
 \end{itemize}
\end{claim}
\begin{proof}
To ease notation, let $\rho_{ij}=\frac{|\mathbf{W}_{ij}^0|}{|V_i||V_j|}$. Recall from Claim \ref{cl:wijreg}, that 
\begin{align}\label{rhoij}
\rho_{ij}=\frac{\ell_1-N_{ij}}{\ell_1}\pm \e,
\end{align}
where $N_{ij}=|W_{ij}^1|+\cdots+|W_{ij}^{m_{ij}}|$.  We obtain  the desired partition of $\mathbf{W}^0_{ij}$  in two cases.

Suppose first $\rho_{ij}\leq \e$.  In this case, set $\mathbf{W}_{ij}^0(1)=\mathbf{W}_{ij}^0$ and for each $2\leq u\leq m_{ij}$, set $\mathbf{W}_{ij}^0(u)=\emptyset$.  Clearly we have that for all $1\leq u\leq m_{ij}$,    $|\mathbf{W}_{ij}^0(u)|\leq \e|V_i||V_j|$, so in this case we are done.

Suppose now $|\mathbf{W}^0_{ij}|>\e$.   Combining this with (\ref{rhoij})  tells us $N_{ij}<\ell_1$.  By definition of $\mathbf{W}_{ij}^0$, definition of $\Psi$, (ii), (\ref{align:2b}),  and the fact all elements in $W_{ij}^0$ are relevant, we have
\begin{align*}
|\mathbf{W}_{ij}^0|\leq |(V_i\times V_j)\cap X_{err}|+\sum_{P_{ij}^{\alpha}\in W_{ij}^0}|P_{ij}^{\alpha}|&\leq \mu^{1/8}|V_i||V_j|+\mu^{1/32}|W_{ij}|(1+\mu)|V_i||V_j|\ell_1^{-1}\\
&\leq 2\mu^{1/8}|V_i||V_j| +\frac{\mu^{1/32}(1+\mu)|V_i||V_j|}{1-\delta}\\
&\leq 2\delta^{1/2}|V_i||V_j|,
\end{align*}
where the second to last inequality uses that $|W_{ij}|\leq \ell_1$,  and the last uses that $\delta$ is sufficiently small and that $\mu$ is sufficiently small compared to $\delta$.  This shows $\rho_{ij}\leq 2\delta^{1/2}$.  Combining with (\ref{rhoij}), we have
\begin{align}\label{nij}
N_{ij}\geq (1-2\delta^{1/2}-\e)\ell_1\geq \ell_1(1-3\delta^{1/2}),
\end{align}
where the last inequality is by (\ref{w0bd2b}). We now show the hypotheses of Lemma \ref{lem:3.8} are satisfied by $(V_i,V_j; \mathbf{W}^0_{ij})$ with parameters $p=\lceil 4\ell_1\rho_{ij}N_{ij}\rceil^{-1}$, $\rho=\rho_{ij}$, and $\e$.  Since $N_{ij}<\ell_1$,  (\ref{rhoij}) and (\ref{w0bd1b}) imply $\rho_{ij}>2\e$.   Since $(V_i,V_j)\notin \Psi$, $V_i$ and $V_j$ are nontrivial, so since $|V|$ is sufficiently large, we have $\e\geq 10(p/\min\{|V_i|,|V_j|\})^{1/5}$. Note that (\ref{rhoij}) implies  
$$
(1-4\ell_1N_{ij}\e)4N_{ij}(\ell_1-N_{ij})\leq  4\ell_1\rho_{ij}N_{ij}\leq (1+4\ell_1N_{ij}\e)4N_{ij}(\ell_1-N_{ij}),
$$
so since $\e$ is sufficiently small (see (\ref{w0bd3b})), we must have $1/p=4N_{ij}(\ell_1-N_{ij})$.  We have left to verify $0<p<\rho/2$.  Note that by (\ref{nij}) and since $N_{ij}<\ell_1$,
\begin{align*}
p=\frac{1}{4N_{ij}(\ell_1-N_{ij})}\leq \frac{1}{4(1-3\delta^{1/2})\ell_1}<\frac{3}{8\ell_1}<\frac{1}{2}\Big(\frac{1}{\ell_1}-\e\Big)\leq \frac{1}{2}\Big(\frac{\ell_1-N_{ij}}{\ell_1}-\e\Big)\leq \frac{\rho_{ij}}{2},
\end{align*}
where the second inequality is because $\delta$ is sufficiently small, the third is by (\ref{w0bd4b}), the fourth is since $N_{ij}<\ell_1$ and the last is by (\ref{rhoij}).

We now apply Lemma \ref{lem:3.8} to obtain a  partition 
$$
\mathbf{W}^0_{ij}=A(1)\cup \ldots \cup A(4N_{ij}(\ell_1-N_{ij})),
$$
where for each $1\leq v\leq 4N_{ij}(\ell_1-N_{ij})$, $(V_i,V_j; A(v))$ satisfies $\dev_2(\e^{1/12})$ with density $\rho p(1\pm o(1))=\frac{1}{4\ell_1N_{ij}}\pm \e$ (we may assume the $o(1)$ term in the conclusion of Lemma \ref{lem:3.8} is less than $\e$ since $\min\{|V_i|,|V_j|\}$  is sufficiently large).   

Now let $I_1\cup \ldots \cup I_{m_{ij}}$ be any partition of $\{1,\ldots, 4N_{ij}(\ell_1-N_{ij})\}$ such that for each $1\leq u\leq m_{ij}$, 
 $$
 |I_u|=4|W_{ij}^u|(\ell_1-N_{ij}).
 $$
 For each $1\leq u\leq m_{ij}$, let $\mathbf{W}^0_{ij}(u)=\bigcup_{\alpha\in I_u}A(\alpha)$.  By definition, $\mathbf{W}_{ij}^0(u)$ is union of  $4|W_{ij}^u|(\ell_1-N_{ij})\leq 4\ell_1^2$ edge sets, each satisfying $\dev_2(\e^{1/12}, \frac{1}{4\ell_1N_{ij}})$. Thus, Fact \ref{fact:adding} implies $(V_i,V_j;\mathbf{W}_{ij}^0(u))$ satisfies $\dev_2(4\ell_1^2\e^{1/144})$ with density at least 
 $$
  \frac{4|W_{ij}^u|(\ell_1-N_{ij})}{\ell_1N_{ij}} -4\ell_1^2\e^{1/144}\geq \frac{1}{\ell_1^2}-4\ell_1^2\e^{1/144},
 $$
 and at most
 $$
 \frac{|W_{ij}^u|(\ell_1-N_{ij})}{\ell_1N_{ij}} +4\ell_1^2\e^{1/144}\leq \frac{3\delta^{1/2}|W_{ij}^u|}{\ell_1}+4\ell_1^2\e^{1/144}\leq \frac{3\delta^{1/2}|\mathbf{W}_{ij}^u|}{|V_i||V_j|}+\e+4\ell_1^2\e^{1/144}\leq 4\delta^{1/2}\frac{|\mathbf{W}_{ij}^u|}{|V_i||V_j|},
 $$
where the first inequality is by (\ref{nij}),  the second is by Claim \ref{cl:wijregb}, and the last is by (\ref{w0bd5b}) and Claim \ref{cl:wijregb}. This finishes the proof of the claim.
 \end{proof}
 
 We can now finish defining our decomposition.  For each $(V_i,V_j)\in \Psi$, choose an arbitrary partition
$$
V_i\times V_j=\mathbf{P}_{ij}^1\cup \ldots \cup \mathbf{P}_{ij}^C.
$$
For $(V_i,V_j)\notin \Psi$, let $\mathbf{W}_{ij}^0(1)\cup \ldots \cup \mathbf{W}_{ij}^0(m_{ij})$ be the partition of $\mathbf{W}_{ij}^0$ from Claim \ref{cl:sij}. Then for each $1\leq u\leq m_{ij}$, set 
$$
\mathbf{P}_{ij}^u=\mathbf{W}_{ij}^u\cup \mathbf{W}_{ij}^0(u),
$$
and for any $m_{ij}<\alpha\leq C$, set $\mathbf{P}_{ij}^{\alpha}=\emptyset$.  By construction, we have $V_i\times V_j=\mathbf{P}_{ij}^1\cup \ldots \cup \mathbf{P}_{ij}^C$.  We will need to know the following statement, which tells us most of our newly defined bigraphs are sufficiently dense and regular, where we use the notation
$$
d_{\mathbf{P}_{ij}^u}:=\frac{|\mathbf{P}_{ij}^u|}{|V_i||V_j|}.
$$
\begin{claim}\label{claimp1}
For each $(V_i,V_j)\notin \Psi$ and $1\leq u\leq m_{ij}$,   $(V_i,V_j;\mathbf{P}_{ij}^u)$ satisfies $\dev_2(5\ell_1^2\e^{1/144})$ with$d_{\mathbf{P}_{ij}^u}\geq \ell_1^{-1}-\mu$. 
\end{claim}
\begin{proof}
For each $(V_i,V_j)\notin \Psi$, Claim \ref{cl:wijregb}, Claim \ref{cl:sij},  Fact \ref{fact:adding}, and the definition of $\mathbf{P}_{ij}^u$ imply $(V_i,V_j;\mathbf{P}_{ij}^u)$ satisfies $\dev_2(5\ell_1^2\e^{1/144})$. The lower bound on the density is by (\ref{align:2b}) and because $\mathbf{P}_{ij}^u$ contains at least one relevant $P_{ij}^{\alpha}$ as a subset.
\end{proof}

We now define our end decomposition to be $\calP^*=(\calP_1^*,\calP_2^*)$, where 
 $$
 \calP_1^*=\calP_1=\calQ_1\text{ and }\calP_2^*=\{\mathbf{P}_{ij}^u: 1\leq i,j\leq t, 1\leq u\leq C\}.
 $$
We just have left to show $\calP^*$ is sufficiently $\dev_{2,3}$-regular with respect to $H$.   For  each $1\leq i,j,s\leq t$ and $1\leq u,v,w\leq C$, set
 $$
\mathbf{P}_{ijs}^{uvw}:=(V_i, V_j, V_s;\mathbf{P}_{ij}^{u}, \mathbf{P}_{is}^v, \mathbf{P}_{js}^w).
$$
Define
 $$
 \Omega^*=\{\mathbf{P}_{ijs}^{uvw}\in \triads(\calP^*): \mathbf{W}_{ijs}^{uvw}\in \Omega\}.
 $$
Our goal is to show that all triads in $\Omega^*$ are homogeneous with respect to $H$.  Note that by construction, for each  $\mathbf{P}_{ijs}^{uvw}\in \Omega^*$, $K_3(\mathbf{W}_{ijs}^{uvw})\subseteq K_3(\mathbf{P}_{ijs}^{uvw})$.  We now show that in fact, $K_3(\mathbf{P}_{ijs}^{uvw})$ is mostly covered by $K_3(\mathbf{W}_{ijs}^{uvw})$.

\begin{claim}\label{cl:pw}
For each $\mathbf{P}_{ijs}^{uvw}\in \Omega^*$, $|K_3(\mathbf{P}_{ijs}^{uvw})\setminus K_3(\mathbf{W}_{ijs}^{uvw})|\leq 15\delta^{1/2}| K_3(\mathbf{P}_{ijs}^{uvw})|$.
\end{claim}
\begin{proof}
Fix  $\mathbf{P}_{ijs}^{uvw}\in \Omega^*$.  Note this implies $(V_i,V_j),(V_i,V_s),(V_j,V_s)\notin \Psi$ and $1\leq u\leq m_{ij}$, $1\leq v\leq m_{is}$, $1\leq w\leq m_{js}$.   We have by Proposition \ref{prop:counting}, (\ref{end1b}), and Claim \ref{claimp1} that 
\begin{align}\label{countp1}
|K_3(V_i,V_j,V_s; \mathbf{P}_{ij}^u, \mathbf{P}_{is}^v, \mathbf{P}_{js}^w)|=(1\pm \mu)d_{\mathbf{P}_{ij}^u}d_{\mathbf{P}_{is}^v}d_{\mathbf{P}_{js}^w}|V_i||V_j||V_s|.
\end{align}
Set  
$$
d_{\mathbf{P}_{ij}^u\setminus \mathbf{W}_{ij}^{0}(u)}=\frac{|\mathbf{P}_{ij}^u\setminus \mathbf{W}_{ij}^{0}(u)|}{|V_i||V_j|}.
$$
Since by construction, $(V_i,V_j; \mathbf{P}_{ij}^u\setminus \mathbf{W}_{ij}^u)=(V_i,V_j; \mathbf{W}_{ij}^0(u))$, we have by Claim \ref{cl:wijregb} that either $d_{\mathbf{P}_{ij}^u\setminus \mathbf{W}_{ij}^{0}(u)}\leq \e$, or $(V_i,V_j; \mathbf{P}_{ij}^u\setminus \mathbf{W}_{ij}^u)$ satisfies $\dev_2(\e)$ and 
$$
\ell_1^{-2}-\e\leq d_{\mathbf{P}_{ij}^u\setminus \mathbf{W}_{ij}^{0}(u)}\leq 4\delta^{1/2}d_{\mathbf{W}_{ij}^u}\leq4\delta^{1/2}d_{\mathbf{P}_{ij}^u}.
$$
If   $d_{\mathbf{P}_{ij}^u\setminus \mathbf{W}_{ij}^{0}(u)}\leq \e$, then combining (\ref{countp1}) with (\ref{end2.5b}) and the lower bound on the densities from Claim \ref{claimp1}, we have
\begin{align*} 
|K_3(V_i,V_j,V_s; \mathbf{P}_{ij}^u\setminus \mathbf{W}_{ij}^0(u), \mathbf{P}_{is}^v, \mathbf{P}_{js}^w)|\leq \e|V_i||V_j||V_s|\leq 4\delta^{1/2} |K_3(  \mathbf{P}_{ijs}^{u,v,w})|.
\end{align*}
 On the other hand, if $d_{\mathbf{P}_{ij}^u\setminus \mathbf{W}_{ij}^{0}(u)}> \e$, then $(V_i,V_j; \mathbf{P}_{ij}^u\setminus \mathbf{W}_{ij}^u)$ satisfies $\dev_2(\e)$ and 
\begin{align}\label{dpw}
\ell_1^{-2}-\e\leq d_{\mathbf{P}_{ij}^u\setminus \mathbf{W}_{ij}^{0}(u)}\leq 4\delta^{1/2}d_{\mathbf{P}_{ij}^u}.
\end{align}
In this case,  we have by Proposition \ref{prop:counting}, the lower bound in (\ref{dpw}), the lower bounds on the densities in Claim \ref{claimp1}, and (\ref{end1b}) that
\begin{align*} 
|K_3(V_i,V_j,V_s; \mathbf{P}_{ij}^u\setminus \mathbf{W}_{ij}^0(u), \mathbf{P}_{is}^v, \mathbf{P}_{js}^w)|&\leq (1+\mu)d_{\mathbf{P}_{ij}^u\setminus \mathbf{W}_{ij}^{0}(u)}d_{\mathbf{P}_{is}^v}d_{\mathbf{P}_{js}^w}|V_i||V_j||V_s|\\
&\leq (1+\mu)4\delta^{1/2}d_{\mathbf{P}_{ij}^u}d_{\mathbf{P}_{is}^v}d_{\mathbf{P}_{js}^w}|V_i||V_j||V_s|\\
&\leq \frac{(1+\mu)}{(1-\mu)}4\delta^{1/2}|K_3(  \mathbf{P}_{ijs}^{u,v,w})|\\
&\leq 5\delta^{1/2}|K_3(  \mathbf{P}_{ijs}^{u,v,w})|,
\end{align*}
where the second inequality is by the upper bound in (\ref{dpw}), the third inequality is by  (\ref{countp1}), and  the last inequality is since $\mu$ is sufficiently small.  Thus, in both cases we have
\begin{align}\label{count2}
|K_3(V_i,V_j,V_s; \mathbf{P}_{ij}^u\setminus \mathbf{W}_{ij}^0(u), \mathbf{P}_{is}^v, \mathbf{P}_{js}^w)|\leq 5\delta^{1/2}|K_3(  \mathbf{P}_{ijs}^{u,v,w})|.
\end{align}
We note symmetric versions of (\ref{count2}) clearly hold  where the roles of $i$, $j$, and $s$ are permuted. Consequently we have 
\begin{align*}
 |K_3(\mathbf{P}_{ijs}^{uvw})\setminus K_3(\mathbf{W}_{ijs}^{uvw})|\leq &|K_3(V_i,V_j,V_s; \mathbf{P}_{ij}^u\setminus \mathbf{W}_{ij}^0(u), \mathbf{P}_{is}^v, \mathbf{P}_{js}^w)|+\\
 &|K_3(V_i,V_j,V_s; \mathbf{P}_{ij}^u, \mathbf{P}_{is}^v\setminus \mathbf{W}_{is}^0(v), \mathbf{P}_{js}^w)|+\\
 &|K_3(V_i,V_j,V_s; \mathbf{P}_{ij}^u, \mathbf{P}_{is}^v, \mathbf{P}_{js}^w\setminus \mathbf{W}_{js}^0(w))|\\
 &\leq 15\delta^{1/2}| K_3(\mathbf{P}_{ijs}^{uvw})|.
\end{align*}
\end{proof}

We can now show all triads form $\Omega^*$ are homogeneous with respect to $H$.  In particular, combining Claims \ref{cl:pw} and \ref{gammahom1}, we have that for every $\mathbf{P}_{ijs}^{uvw}\in \Omega^*$, there is $\tau\in \{0,1\}$ so that 
\begin{align*}
|\overline{E^{\tau}}\cap K_3(\mathbf{P}_{ijs}^{uvw})|\geq |\overline{E^{\tau}}\cap K_3(\mathbf{W}_{ijs}^{uvw})|&\geq (1-9\delta^{1/256})|K_3(\mathbf{W}_{ijs}^{uvw})|\\
&\geq (1-9\delta^{1/256})(1-15\delta^{1/2})|K_3(\mathbf{P}_{ijs}^{uvw})|\\
&\geq \Big(1-\frac{\e_1}{6}\Big)|K_3(\mathbf{P}_{ijs}^{uvw})|,
\end{align*} 
where the last inequality is because $\delta$ is sufficiently small compared to $\e_1$.  By Claim \ref{claimp1} and (\ref{end2b}), each of $(V_i,V_j;\mathbf{P}_{ij}^u)$, $(V_i,V_s;\mathbf{P}_{is}^v)$, and $(V_j,V_s;\mathbf{P}_{js}^w)$ have $\dev_2(\e_2(C))$.  Thus, by Proposition \ref{prop:homimpliesrandome}, every  $\mathbf{P}_{ijs}^{uvw}\in \Omega^*$ is $\dev_{2,3}(\e_1,\e_2(C))$-regular with respect to $H$.  This finishes the proof since, by construction and Claim \ref{cl:gamma1}, 
$$
|\bigcup_{G\in \Omega^*}K_3(G)|\geq |\bigcup_{G\in \Omega}K_3(G)|\geq (1-\delta^{1/2})|V|^3\geq (1-\e_1)|V|^3,
$$
where the last inequality is because $\delta$ is sufficiently small compared to $\e_1$.
  \end{proof}

An immediate corollary of Proposition \ref{prop:mainL} is the following, which says that when ${\bf B}_{\calH}$ contains only finitely many bigraphs up to isomorphism, then $L_{\calH}$ is bounded by a constant.  Again this can be done while minimizing the size of the vertex partition with respect to polynomially related parameters.

\begin{corollary}\label{cor:constantub}
Suppose $\calH$ is a hereditary $3$-graph property and for some integer $L\geq 1$,
$$
L=\max\{K\in \mathbb{N}: \calB_{\calH}\cap Irr(K)\neq \emptyset\}.
$$
Then there is $\e_1^*>0$ and a polynomial $p(x,y)$ so that for all $0<\e_1<\e^*_1$ and all $\e_2:\mathbb{N}\rightarrow (0,1]$ satisfying $\e_2(x)\leq p(\e_1,x^{-1})$, $L_{\calH}(\e_1,\e_2)\leq L$.
\end{corollary}
\begin{proof}
By definition of $L$, there exists an integer $m$ so that for all $G\in Irr(L+1)$, every element in $\calH$ has $G$-dimension at most $m$.  Let $p(x,y), r(x),q(x,y,z),s(x,y),\e_1^*$ be as in Proposition \ref{prop:mainL} for $L$ and $m$. Let $0<\e_1< \e_1^*$, and assume $\e_2:\mathbb{N}\rightarrow (0,1)$ satisfies $\e_2(x)\leq p(\e_1,x^{-1})$.    Let $L$ be any integer so that $\Psi(\e_1',\e_2',T_{\calH}(\e_1',\e_2'), L, \calH)$ holds, where $\e_1'=r(\e_1)$ and $\e'_2(x)=q(\e_1,x^{-1},\e_2(\lceil s(\e_1^{-1},x)\rceil ))$.  Suppose $H$ is a sufficiently large element of $\calH$.  By our choice of $L$,  there is a  $\dev_{2,3}(\e_1',\e_2'(\ell))$-regular $(t,\ell, \e_1',\e_2'(\ell))$-decomposition $\calP$ for $H$ with $1\leq t\leq T_{\calH}(\e_1',\e_2')$ and $1\leq \ell\leq L$.  By Proposition \ref{prop:mainL}, there exists $\calP'$ a $\dev_{2,3}(\e_1,\e_2(C))$-regular $(t,C,\e_1,\e_2(\ell))$-decomposition.  This shows $L_{\calH}(\e_1,\e_2)\leq C$.
\end{proof}

\section{Lower Bounds }\label{sec:LBL}

In this section, we provide a general type of construction, which we will then use to produce the lower bound examples for all three growth rates of Theorem \ref{thm:mainL2}.  Roughly speaking, these constructions arise as $(n,\Gamma)$-blowups of irreducible bigraphs, where $\Gamma$ is an extremely quasirandom partition of $[n]\times [n]$.  For reasons of convenience, we also use the notion of $\e$-regularity to measure this quasirandomness.

\begin{definition}
Suppose $G=(A, B; E)$ is a bigraph.  We say it $G$ is \emph{$\e$-regular} if for all $A'\subseteq A$ and $B'\subseteq B$ with $|A'|\geq \e|A|$ and $|B'|\geq \e|B|$, $|d_G(A,B)-d_G(A',B')|\leq \e$. 
\end{definition}

Another ingredient needed this section is Theorem \ref{thm:translate} below, due to Nagle, R\"{o}dl, and Schacht \cite{NRS}. This result gives a quantitative relationship between the $\dev_{2,3}$-quasirandomness first developed by Gowers and the $\disc_{2,3}$-quasirandomness first developed by Frankl and R\"{o}dl.  A proof with non-explicit bounds appears in \cite{Nagle.2013}.  We note the polynomial dependence of $\e_2$ on $d_2$ stated in Theorem \ref{thm:translate} below does not appear in the statement in \cite{NRS}, but is immediately apparent upon inspection of their proof. To state Theorem \ref{thm:translate} we require the following definition. Given a triad 
$$
G=(A,B,C;E_{AB},E_{AC},E_{BC}),
$$
 a \emph{sub-triad of $G$} is a triad of the form $G'=(A',B',C';E_{A'B'},E_{A'C'},E_{B'C'})$ where $A'\subseteq A$, $B'\subseteq B$, $C'\subseteq C$, $E_{A'B'}\subseteq E_{AB}$, $E_{A'C'}\subseteq E_{AB}$, and $E_{B'C'}\subseteq E_{BC}$.

\begin{theorem}[Proposition 1.4 in \cite{NRS}]\label{thm:translate}
There is a polynomial $p(x)$ so that for all $0<\e_1,d_1,d_2<1$ and all $0<\e_2<p(d_2)$,  the following holds. Suppose $H=(V_1,V_2,V_3;F)$ is a trigraph and $G$ is a triad whose component bigraphs are all $\e_2$-regular with density at least $d_2$. Suppose $(H,G)$ satisfies $\dev_{2,3}(\e_1,\e_2)$.  Then for every sub-triad $G'$ of $G$, 
\[||E(H)\cap K_3(G')|-d_G(H)|K_3(G')||\leq (2\e_1)^{1/8} d_{12}d_{13}d_{23}|V_1||V_2||V_3|,\]
where for each $1\leq i<j\leq 3$, $d_{ij}=d_G(V_i,V_j)$.
\end{theorem}

We now prove a technical lemma which is at the heart of all three of our lower bounds. It roughly says that given a bigraph $\calG$ and a sufficiently quasirandom partition $\Gamma$ of $[n]\times [n]$, any sufficiently regular decomposition $\calP$ of a $\Gamma$-blowup of $\calG$ will be forced to have $\calP_2$ correlating with the partition given by $\Gamma$.

\begin{lemma}\label{lem:LB2}
There are $ \e^*>0$ and a polynomial $q(x,y)$  so that the following holds.  Suppose $\calG=(U,V;E)$ is an irreducible bigraph.  For all $0<\e_1<\min\{|V|^{-64}/2,\e^*\}$,  all integers $t\geq 1$, and all $\e_2:\mathbb{N}\rightarrow (0,1)$ satisfying $\e_2(x)<p(\e_1,x^{-1})$, there is $\e\in (0,1)$ so that for all sufficiently large $n$,  the following holds.

Suppose $\Gamma=(A, B;(P_u)_{u\in U})$ is a $U$-colored bigraph with $|A|=|B|=n$ such that for each $u\in U$, $(A, B; P_u)$ is $\e$-regular with density $|U|^{-1}\pm \e$.  Suppose $H$ is an $(n/|V|,\Gamma)$-blowup of $\calG$.  Let $A\cup B\cup C$ be the vertex set of $H$, where $C=\bigcup_{v\in V}C_v$ and for each $v\in V$, $|C_v|=n/|V|$, and let $E(H)$ be the edge set of $H$, where
$$
E(H)\cap K_3[A,B,C]=\bigcup_{(u,v)\in E(\calG)}\{xyz: (x,y)\in P_u, z\in C_v\}.
$$
Suppose $\calP$ is a $\dev_{2,3}(\e_1,\e_2(\ell))$-regular $(3t,\ell,\e_1,\e_2(\ell))$-decomposition of $H$ with 
$$
\calP_1=\{A_i,B_i,W_i: i\in [t]\},
$$
 where 
$$
A=\bigcup_{i\in [t]}A_i,\text{ }B=\bigcup_{i\in [t]}B_i,\text{ and }C=\bigcup_{i\in [t]}W_i,
$$
and 
$$
\calP_2=\{P_{A_iB_j}^{\alpha},P_{A_iC_k}^{\beta},P_{B_jW_k}^{\gamma}: 1\leq i,j,k\leq t, 1\leq \alpha,\beta,\gamma\leq \ell\}.
$$
Then there are $A_i,B_j\in \calP_1$, a set $\calZ\subseteq [\ell]$ and for all $\alpha\in \calZ$, a set $\calC_{\alpha}\subseteq \{C_v: v\in V\}$ satisfying the following.
\begin{enumerate}
\item  $|A_i|\geq \e_1|A|/3t$ and $|B_j|\geq \e_1|B|/3t$,
\item  $|\bigcup_{\alpha\in \calZ}P_{A_iB_j}^{\alpha}|\geq (1-\e_1^{1/4})|A_i||B_j|$,
\item For all $\alpha\in \calZ$, $|\calC_{\alpha}|\geq (1-3\e_1^{1/8})|V|$,
\item For all $\alpha\in \calZ$ and $C_v\in \calC_{\alpha}$, there is $f(\alpha,v)\in \{0,1\}$ so that the following holds, where $N_{\calG}^1(v)=N_{\calG}(v)$ and $N_{\calG}^0(v)=U\setminus N_{\calG}(v)$.
$$
\Big|P_{A_iB_j}^{\alpha}\cap \Big(\bigcap_{C_v\in \calC_{\alpha}}\Big(\bigcup_{u\in N_{\calG}^{f(\alpha,v)}(v)}P_u\Big)\Big)\Big|\geq (1-\e^{1/64}_1)|P_{A_iB_j}^{\alpha}|.
$$
\end{enumerate}
\end{lemma}
\begin{proof}
Let $p(x)$ be as in Theorem \ref{thm:translate}.  Let $\e^*>0$ be sufficiently small so that 
$$
\frac{(\e^*)^{1/16}}{2}\geq (2\e^*)^{1/8}\text{ and }6(\e^*)^{1/4}<(\e^*)^{1/32}.
$$
  Define 
$$
q(x,y)=p(x y)\cdot \Big(\frac{xy}{4}\Big)^{2000}.
$$
Suppose $\calG=(U,V;E)$ is an irreducible bigraph.   Fix $0<\e_1<\min\{|V|^{-64}/2,\e^*\}$, $t\geq 1$, and $\e_2:\mathbb{N}\rightarrow (0,1]$ satisfying $\e_2(y)\leq p(\e_1,y^{-1})$.   Choose $\e>0$ sufficiently small compared to the previously mentioned parameters (including $t^{-1}$ and $|U|^{-1}$) and assume $n$ is sufficiently large.

Assume $\Gamma=(A,B; (P_u)_{u\in U})$ is a $U$-colored bigraph with $|A|=|B|=n$ such that for each $u\in U$, $(A, B; P_u)$ is $\e$-regular with density $|U|^{-1}\pm \e$.  Suppose  $H$ is an $(n/|V|,\Gamma)$-blowup of $G$.  We may assume $H$ has vertex set $A\cup B\cup C$ where $C=\bigcup_{v\in V}C_v$ and for each $v\in V$, $|C_v|=n/|V|$, and we may assume the edge set of $H$ satisfies 
$$
E(H)\cap K_3[A,B,C]=\bigcup_{(u,v)\in E(\calG)}\{xyz: (x,y)\in P_u, z\in C_v\}.
$$

Assume $\calP$ is a $\dev_{2,3}(\e_1,\e_2(\ell))$-regular $(t,\ell,\e_1,\e_2(\ell))$-decomposition of $H$ with $\calP_1=\{A_i,B_i,W_i: i\in [t]\}$, where the following hold.
$$
A=\bigcup_{i\in [t]}A_i,\text{ }B=\bigcup_{i\in [t]}B_i,\text{ and }C=\bigcup_{i\in [t]}W_i.
$$
Fix an enumeration of $\calP_2$ as follows.
$$
\calP_2=\{P_{A_iB_j}^{\alpha},P_{A_iC_k}^{\beta},P_{B_jC_k}^{\gamma}: 1\leq i,j,k\leq t, 1\leq \alpha,\beta,\gamma\leq \ell\}.
$$
We now set up a bit more notation.  Given $X,Y\in \calP_1$ and $1\leq \alpha\leq \ell$,  denote the density of the bigraph $(X,Y; P_{XY}^{\alpha})$ by $d_{XY}^{\alpha}=|P_{XY}^{\alpha}|/|X||Y|$.  Given $X,Y,Z\in \calP_1$ and $1\leq \alpha,\beta,\gamma\leq \ell$, set $G_{XYZ}^{\alpha\beta\gamma}$ denote the triad $(X,Y,Z; P_{XY}^{\alpha},P_{XZ}^{\beta},P_{YZ}^{\gamma})$.  

We next define several auxiliary sets with the goal of controlling the behavior of the irregular triads of $\calP$.  We begin by defining the subsets of $\calP_1$ and $\calP_2$ consisting of elements which are $\e_1$-nontrivial.   In particular, we define
$$
\calP_1^{lg}=\{X\in \calP_1: |X|\geq \e_1 |V(H)|/3t\}\text{ and }\calP_2^{lg}=\{P_{XY}^{\alpha}: X,Y\in \calP_1^{lg}, |P_{XY}^{\alpha}|\geq \e_1 |X||Y|/\ell\}.
$$
We next set notation for set of the $\dev_{2,3}$-regular triads whose components come $\calP_1^{lg}$ and $\calP_2^{lg}$.
$$
\Omega=\{G\in \triads(\calP): \text{$G$ is $\e_1$-nontrivial and $\dev_{2,3}(\e_1,\e_2(\ell))$-regular with respect to $H$}\}.
$$
Then set 
\begin{align*}
\calS_{good}=\bigcup_{G\in \Omega}K_3(G)\text{ and }\calS_{bad}=V(H)^3\setminus \calS_{good}. 
\end{align*}
The goal of the next few paragraphs is to get good control over $\calS_{bad}$.  Observe that by Lemma \ref{lem:averaging} and since $\calP$ is a $\dev_{2,3}(\e_1,\e_2(\ell))$-regular with respect to $H$, 
\begin{align}\label{al:sbad}
|\calS_{bad}\cap (A\times B\times C)|\leq |\calS_{bad}|\leq 3\e_1|V(H)|^3=27\e_1|A||B||C|,
\end{align}
where the the last is because $|A|=|B|=|C|=n=|V(H)|/3$. Define now 
$$
\Sigma_{good}=\{(x,y)\in A\times B: |N_{\calS_{good}}(x,y)|\geq (1-3\e_1^{1/2})|C|\}\text{ and }\Sigma_{bad}=(A\times B)\setminus \Sigma_{good}.
$$
Since $|\calS_{bad}|\leq 27\e_1|A||B||C|$, we have $|\Sigma_{bad}|\leq 9\e_1^{1/2}|A||B|$ and thus 
\begin{align}\label{al:sigmagood1}
|\Sigma_{good}|\geq (1-9\e_1^{1/2})|A||B|.  
\end{align}
Since $\Sigma_{good}$ is contained in $\bigcup_{A_i,B_j\in \calP_1^{lg}}A_i\times B_j$,  inequality (\ref{al:sigmagood1}) and averaging implies there exist some $A_i,B_j\in \calP_1^{lg}$ so that 
\begin{align}\label{al:sigmagood}
 |\Sigma_{good}\cap (A_i\times B_j)|&\geq (1-9\e_1^{1/2})|A_i||B_j|.
\end{align}
This $A_i$ and $B_j$ will the ones from the statement of \ref{lem:LB2}.  Note conclusion (1) from Lemma \ref{lem:LB2} holds for this choice of $A_i,B_j$ by construction.  We now observe that by definition of $\Sigma_{good}$ and   inequality (\ref{al:sigmagood}), we have 
\begin{align}
|\calS_{good}\cap (A_i\times B_j\times C)|&\geq |\Sigma_{good}\cap (A_i\times B_j)|(1-3\e_1^{1/2})|C|\nonumber\\
&\geq (1-9\e_1^{1/2})(1-3\e_1^{1/2})|A_i||B_j||C|\nonumber\\ 
&\geq (1-12\e_1^{1/2})|A_i||B_j||C|.\label{al:aibjsgood}
\end{align}
We now define the set $\calZ$ from the statement of our desired result. In particular, we set 
\begin{align*}
\calZ&=\{\alpha\in [\ell]: P_{A_iB_j}^{\alpha}\in \calP_2^{lg}\text{ and } |P_{A_iB_j}^{\alpha}\cap \Sigma_{good}|\geq (1-3\e_1^{1/4})|P_{A_iB_j}^{\alpha}|\}.
\end{align*}
By  inequality (\ref{al:sigmagood}) and Lemma \ref{lem:averaging}, we have $|\bigcup_{\alpha\in \calZ}P_{A_iB_j}^{\alpha}|\geq (1-3\e_1^{1/4})|A_i||B_j|$ (so conclusion (2) of Lemma \ref{lem:LB2} holds for this choice of $\calZ$).  We now define a subset of $C$ consisting of elements who behave well with respect to $\calS_{good}$ and $A_i,B_j$.  In particular, set 
$$
C_{good}=\{c\in C: |N_{\calS_{bad}}(c)\cap (A_i\times B_j)|\leq 3\e_1^{1/4}|A_i||B_j|\}.
$$
By definition of $C_{good}$ and  inequality (\ref{al:aibjsgood}), we have  
  $$
  3\e_1^{1/4}|A_i||B_j||C\setminus C_{good}|\leq |(A_i\times B_j\times C)\setminus \calS_{good}|\leq 12\e_1^{1/2}|A_i||B_j||C|.
  $$
This implies 
\begin{align}\label{al:wgood}
|C_{good}|\geq (1-4\e_1^{1/4})|C|.
\end{align}
Now define the set of elements in $\calP^{lg}_1$ which are mostly contained in $C_{good}$. 
$$
\calW_{good}=\{W_k\in \calP^{lg}_1:  |W_k\cap C_{good}|\geq (1-2\e_1^{1/8})|W_k|\}.
$$
By Lemma \ref{lem:averaging} and  inequality (\ref{al:wgood}), $|\bigcup_{W_k\in \calW_{good}}W_k|\geq (1-2\e_1^{1/8})|C|$. 
We now define the set of indices $\alpha\in [\ell]$ so that $P_{ij}^{\alpha}\in \calP_2^{lg}$ and such that $P_{ij}^{\alpha}$ is well behaved with respect to a fixed $W_k$.  In particular, given $W_k\in \calW_{good}$, set
$$
\calN_{\Omega}(W_k)=\{\alpha\in [\ell]:\text{   there are }1\leq \beta,\gamma\leq \ell \text{ so that }G_{A_iB_jC_k}^{\alpha\beta\gamma}\in \Omega\}.
$$

We now turn to proving conclusions (3) and (4) of Lemma \ref{lem:LB2}.  To ease notation for this, given $v\in V$, we define  
$$
R(v,1)=\bigcup_{u\in N_{\calG}(v)}P_u\text{ and }R(v,0)=\bigcup_{u\in U\setminus N_{\calG}(v)}P_u.
$$
We next prove a claim which roughly shows that certain elements of $\calP_2$ cannot cross-cut sets of the form $R(v,1)$ and $R(v,0)$.

\begin{claim}\label{cl:wcase}
Suppose $W_k\in  \calW_{good}$ and $v\in V$. If $|W_k\cap C_{good}\cap C_v|\geq \e^{1/32}_1|W_k|$, then for any $\alpha\in \calN_{\Omega}(W_k)$,  either $|P_{A_iB_j}^{\alpha}\cap R(v,1)|<\e^{1/32}_1|P_{A_iB_j}^{\alpha}|$ or $|P_{ij}^{\alpha}\cap R(v,0)|<\e^{1/32}_1|P_{A_iB_j}^{\alpha}|$.  
\end{claim}
\begin{proof}
Fix $W_k\in \calW_{good}$ and $v\in V$ such that $|W_k\cap C_v\cap C_{good}|\geq\e^{1/32}_1|W_k|$.  We now define a few auxiliary sets. First, we set
$$
J(W_k)=\bigcup_{\alpha\in \calN_{\Omega}(W_k)}P_{A_iB_j}^{\alpha}.
$$
We  show $J(W_k)$ is large.  By definition,
\begin{align}\label{al:J}
|J(W_k)|\geq |\{(x,y)\in A_i\times B_j: \text{ for some }z\in W_k, (x,y,z)\in \calS_{good}\}|.
\end{align}
Since $W_k\in \calW_{good}$, 
\begin{align*}
(1-2\e_1^{1/8})(1-3\e_1^{1/4})|A_i||B_j||W_k|&\leq (1-3\e_1^{1/4})|A_i||B_j||W_k\cap C_{good}|\\
&\leq |\calS_{good}\cap (A_i\times B_j\times W_k)|.
\end{align*}
Consequently, $|\calS_{good}\cap (A_i\times B_j\times W_k)|\geq (1-5\e_1^{1/8})|A_i||B_j||W_k|$.  Clearly, this implies
$$
|\{(x,y)\in A_i\times B_j: \text{ for some }z\in W_k, (x,y,z)\in \calS_{good}\}|\geq (1-5\e_1^{1/8})|A_i||B_j|.
$$
Combining with (\ref{al:J}), this yields
\begin{align}\label{al:Jbound}
|J(W_k)|\geq (1-5\e_1^{1/8})|A_i||B_j|.
\end{align}
Observe that for all $u\in U$,  becauase $\e$ is sufficiently small compared to $\e_1/t$, because $A_i,B_j\in \calP_1^{lg}$, and by Lemma \ref{lem:sldev}, we have that the bigraph $(A_i,B_j; P_u\cap (A_i\times B_j))$ has $\dev_2(\sqrt{\e})$ with density $|U|^{-1}\pm 2\e$.  Define 
$$
\calU=\{u\in U: |P_u\cap J(W_k)|\geq (1-3\e_1^{1/16})|P_u\cap (A_i\times B_j)|\}.
$$
Since $J(W_k)\subseteq A_i\times B_j=(\bigcup_{u\in U}P_u)\cap (A_i\times B_j)$, we have by (\ref{al:Jbound}) and Lemma \ref{lem:averaging} that 
$$
|\bigcup_{u\in \calU}P_u|\geq (1-2\e_1^{1/16})|A_i||B_j|.
$$
Because all the bigraphs $(A_i,B_j;\overline{P}_u\cap (A_i\times B_j))$ have density $|U|^{-1}\pm 2\e$ in $A_i\times B_j$, and because $\e$ is sufficiently small, this implies 
$$
|\calU|\geq \frac{(1-2\e_1^{1/16})|A_i||B_j|}{(1+\e)|U|^{-1}|A_i||B_j|}\geq (1-3\e^{1/16})|U|.
$$

Now we suppose towards a contradiction there exists some $\alpha\in \calN_{\Omega}(W_k)$ so that 
\begin{align}\label{wkcont}
\min\{|P_{A_iB_j}^{\alpha}\cap R(v,1)|,|P_{A_iB_j}^{\alpha}\cap R(v,0)|\}\geq \e^{1/32}_1|P_{A_iB_j}^{\alpha}|.
\end{align}
By definition of $\calN_{\Omega}(W_k)$, there are some $1\leq \beta,\gamma\leq \ell$ so that $G_{A_iB_jW_k}^{\alpha\beta\gamma}\in \Omega$.  Note since $G_{A_iB_jW_k}^{\alpha\beta\gamma}$ is $\e_1$-nontrivial,
\begin{align}\label{dlb}
\min\{d_{ij}^{\alpha},d_{ik}^{\beta}, d_{jk}^{\gamma}\}\geq \frac{\e_1}{\ell},
\end{align}
and since it is $\dev_{2,3}(\e_1,\e_2(\ell))$-regular with respect to $H$, Proposition \ref{prop:counting} implies
\begin{align}\label{k3lb}
|K_3(G_{A_iB_jW_k}^{\alpha\beta\gamma})|=(d_{ij}^{\alpha}d_{ik}^{\beta}d_{jk}^{\gamma}\pm 4\e_2(\ell)^{1/4})|A_i||B_j||W_k|\leq \frac{6}{4}d_{ij}^{\alpha}(d_{ik}^{\beta}d_{jk}^{\gamma}-2\e_2(\ell)^{1/2000}) |A_i||B_j||W_k|,
\end{align}
where the last inequality uses (\ref{dlb}) and the fact that $\e_2(\ell)$ is sufficiently small compared to $\e_1/\ell$ (in particular, $(\e_1/\ell)^3\geq 8\e_2(\ell)^{1/2000}$).   We now consider the following two sub-trigraphs of $G_{A_iB_jW_k}^{\alpha\beta\gamma}$.
\begin{align*}
G_1&=(A_i,B_j,W_k\cap C_{good}\cap C_v; P_{A_iB_j}^{\alpha}\cap R(v,1), P_{A_iW_k}^{\beta},P_{B_jW_k}^{\gamma})\text{ and }\\
G_0&=(A_i,B_j,W_k\cap C_{good}\cap C_v; P_{A_iB_j}^{\alpha}\cap R(v,0), P_{A_iW_k}^{\beta},P_{B_jW_k}^{\gamma}).
\end{align*}
By definition of $H$, $R(v,1)$, and $R(v,0)$, we have 
$$
K_3(G_1)\subseteq \overline{E(H)}\text{ and }K_3(G_0)\cap \overline{E(H)}=\emptyset.
$$
By Lemma \ref{lem:sldev} and our choice of parameters, the bigraphs 
$$
(A_i,W_k\cap C_{good}\cap C_v; P_{A_iW_k}^{\beta})\text{ and }(B_j,W_k\cap C_{good}\cap C_v; P_{B_jW_k}^{\gamma})
$$
both satisfy $\dev_2(\e_2(\ell)^{1/20})$ with respective densities 
 \begin{align}\label{dlb2}
 d_{ik,\beta}'=d_{A_iW_k}^{\beta}\pm \e_2(\ell)^{1/20}\text{ and }d_{jk,\gamma}'=d_{B_jW_k}^{\gamma}\pm \e_2(\ell)^{1/20}.
 \end{align}
   Let
\begin{align*}
Y=\{(x,y)\in P_{A_iB_j}^{\alpha}: &|N_{P_{A_iW_k}^{\beta}}(x)\cap N_{P_{B_jW_k}^{\gamma}}(x)\cap W_k\cap C_{good}\cap C_v|\\
&=(d'_{ik,\beta}d'_{jk,\gamma}\pm \e_2(\ell)^{1/2000})|W_k\cap C_{good}\cap C_v|\}.
\end{align*}
By Lemma \ref{lem:standreg}, 
\begin{align}\label{ylb}
|Y|\geq (1-\e_2(\ell)^{1/2000})|P_{A_iB_j}^{\alpha}|.
\end{align}
We can now deduce the following lower bound for the size of $K_3(G_1)$. 
\begin{align*}
|K_3(G_1)|&\geq \sum_{(x,y)\in P_{A_iB_j}^{\alpha}\cap R(v,1)\cap Y}|N_{P_{A_iW_k}^{\beta}}(x)\cap N_{P_{B_jW_k}^{\gamma}}(x)\cap W_k\cap C_{good}\cap C_v|\\
&\geq |P_{A_iB_j}^{\alpha}\cap R(v,1)\cap Y|(d'_{ik,\beta}d'_{jk,\gamma}-  \e_2(\ell)^{1/2000})|W_k\cap C_{good}\cap C_v|\\
&\geq |P_{A_iB_j}^{\alpha}\cap R(v,1)\cap Y|(d_{A_iW_k}^{\beta}d_{B_jW_k}^{\gamma}- 2\e_2(\ell)^{1/2000})|W_k\cap C_{good}\cap C_v|\\
&\geq |P_{A_iB_j}^{\alpha}\cap R(v,1)\cap Y|(d_{A_iW_k}^{\beta}d_{B_jW_k}^{\gamma}- 2\e_2(\ell)^{1/2000})\e_1^{1/32}|W_k|,
\end{align*}
where the first inequality is by definition of $Y$, the second is by (\ref{dlb2}), and the third is by our assumption on the size of $W_k\cap C_{good}\cap C_v$.  Combining with our assumption (\ref{wkcont}) and with (\ref{ylb}), this implies 
\begin{align*}
|K_3(G_1)|&\geq (|P_{A_iB_j}^{\alpha}\cap R(v,1)|-|P_{A_iB_j}^{\alpha}\setminus Y|)(d_{A_iW_k}^{\beta}d_{B_jW_k}^{\gamma}-2\e_2(\ell)^{1/2000})\e_1^{1/32}|W_k|\\
&\geq \e_1^{1/32}|P_{ij}^{\alpha}|(d_{A_iW_k}^{\beta}d_{B_jW_k}^{\gamma}-2 \e_2(\ell)^{1/2000})\e_1^{1/32}|W_k|\\
&\geq \frac{3\e_1^{1/16}}{4}d_{A_iB_j}^{\alpha}(d_{A_iW_k}^{\beta}d_{B_jW_k}^{\gamma}-2 \e_2(\ell)^{1/2000})|A_i||B_j||W_k|\\ 
&\geq \frac{\e_1^{1/16}}{2}|K_3(G_{A_iB_jW_k}^{\alpha\beta\gamma})|,
\end{align*}
where the second inequality uses that $\e_2(\ell)$ is sufficiently small compared to $\e_1$ (in particular $\e_2(\ell)^{1/2000}<\e_1^{1/32}/4$), and where last inequality uses (\ref{k3lb}).   A similar argument shows 
$$
|K_3(G_0)|\geq \frac{\e_1^{1/16}}{2}|K_3(G_{A_iB_jW_k}^{\alpha\beta\gamma})|.
$$
Assume $d_H(G_{A_iB_jW_k}^{\alpha\beta\gamma})\leq 1/2$  (the case where $d_H(G_{A_iB_jW_k}^{\alpha\beta\gamma})>1/2$ is similar).  Then, we have
\begin{align*}
||E(H)\cap K_3(G_1)|-d_G(H)|K_3(G_1)||=|K_3(G_1)|-d_G(H)|K_3(G_1)|&\geq \frac{1}{2}|K_3(G_1)|\\
&\geq \frac{\e_1^{1/16}}{2}|K_3(G_{A_iB_jW_k}^{\alpha\beta\gamma})|\\
&\geq (2\e_1)^{1/8}|K_3(G_{A_iB_jW_k}^{\alpha\beta\gamma})|,
\end{align*}
where the last inequality uses that $\e_1$ is sufficiently small (in particular $\e_1<\e^*$).  However, since we have (\ref{dlb}) and $\e_2(\ell)\leq p(\e_1/\ell)$, this contradicts Theorem \ref{thm:translate} and the assumption that $G_{A_iB_jW_k}^{\alpha\beta\gamma}\in \Omega$.  This finishes the proof of Claim \ref{cl:wcase}.
\end{proof}

Fix $\alpha\in \calZ$.  We show conclusions (3) and (4) of Lemma \ref{lem:LB2} hold for this $\alpha$.  Define 
$$
\calN_{\Omega}(\alpha)=\{W_k\in \calW_{good}:  \alpha\in \calN_{\Omega}(W_k)\}.
$$
By definition of $\calZ$ and $\Sigma_{good}$,
\begin{align}\label{al:union}
\Big|\bigcup_{W_k\in \calN_{\Omega}(\alpha)}W_k\Big|\geq (1-3\e_1^{1/2})|P_{A_iB_j}^{\alpha}\cap \Sigma_{good}||C|\geq (1-3\e_1^{D/2})(1-3\e_1^{1/4})|C|\geq (1-6\e_1^{1/4})|C|.
\end{align}
Define now $C^{\alpha}_{good}=C_{good}\cap (\bigcup_{W_k\in \calN_{\Omega}(\alpha)}W_k)$.  By  inequalities (\ref{al:wgood}) and (\ref{al:union}), 
\begin{align}\label{al:calpha}
|C^{\alpha}_{good}|\geq |C_{good}|-6\e_1^{1/4}|C|\geq (1-9\e_1^{1/4})|C|.
\end{align}
We now consider the collection of sets of the form $C_v$ which are mostly covered by $C^{\alpha}_{good}$.  In particular, define
$$
\calC_{\alpha}=\{C_v: v\in V, |C_v\cap C^{\alpha}_{good}|\geq (1-3\e_1^{1/8})|C_v|\}.
$$
By inequality (\ref{al:calpha}) and Lemma \ref{lem:averaging}, and since all the $C_v$ have the same size, we have 
$$
|\calC_{\alpha}|\geq (1-3\e_1^{1/8})|V|,
$$
which shows conclusion (3) of Lemma \ref{lem:LB2} holds.  We just have left to check conclusion (4) for $\alpha$.  Fix some $C_v\in \calC_{\alpha}$.  We show there is some $W_k\in \calN_{\Omega}(\alpha)$ so that $|C_v\cap C_{good}\cap W_k|\geq \e_1^{1/32}|W_k|$.  Indeed, if this was not the case, we would have by inequality (\ref{al:union}) that
$$
\frac{|C|}{|V|}=|C_v|\leq \Big|C\setminus \Big(\bigcup_{W_k\in \calN_{\Omega}(\alpha)}W_k\Big)\Big|+\sum_{W_k\in \calN_{\Omega}(\alpha)}\e_1^{1/32}|W_k|\leq (6\e_1^{1/4}+\e_1^{1/32})|C|\leq 2\e_1^{1/32}|C|,
$$
where the last inequality is because $\e_1<\e^*$.  However, this contradicts our assumption that $\e_1<|V|^{-64}/2$.  Thus, for all $C_v\in \calC_{\alpha}$ there is some $W_k\in \calN_{\Omega}(\alpha)$ so that $|C_v\cap C_{good}\cap W_k|\geq \e_1^{1/32}|W_k|$.  By Claim \ref{cl:wcase}, this implies that for all $C_v\in \calC_{\alpha}$, either $|P_{A_iB_j}^{\alpha}\cap R(v,1)|<\e_1^{1/32}|P_{A_iB_j}^{\alpha}|$ or $|P_{A_iB_j}^{\alpha}\cap R(v,0)|<\e_1^{1/32}|P_{A_iB_j}^{\alpha}|$.  Thus, for all $C_v\in \calC_{\alpha}$, there is some $f(\alpha,v)\in \{0,1\}$ so that 
$$
\Big|P_{A_iB_j}^{\alpha}\setminus \Big(\bigcap_{v\in C_{\alpha}}R(v,f(Q,v))\Big)\Big|\leq |V|\e_1^{1/32}|P_{A_iB_j}^{\alpha}|\leq \e_1^{1/32}|P_{A_iB_j}^{\alpha}|,
$$
where the last inequality is because $\e_1<|V|^{-64}$.  This finishes the proof.
\end{proof}

We now use Lemma \ref{lem:LB2} to produce a general lower bound lemma, which will be used to produce our constant, polynomial, and exponential lower bound constructions. 

\begin{proposition}\label{prop:LB}
There are constants $K\geq 1$ and $\e^*>0$ and a polynomial $p(x,y)$ so that the following holds.  Suppose $\calG=(U, V;E)$ is an irreducible bigraph and 
$$
0<\e_1<\min\Big\{  \frac{1}{2|V|^{64}}, \Big(\frac{1}{3|V|^2}\Big)^8, \e^*\Big\}.
$$
For all integers $t\geq 1$ and all $\e_2:\mathbb{N}\rightarrow (0,1)$ satisfying $\e_2(x)<p(\e_1,x^{-1})$, there is $\e\in (0,1)$ so that for sufficiently large $n$, the following holds.

Suppose $\Gamma=(A, B; (P_u)_{u\in U})$ is a $U$-colored bigraph with $|A|=|B|=n$ such that for each $u\in U$, $(A,B; P_u)$ is $\e$-regular with density $|U|^{-1}\pm \e$.  Suppose $H$ is an $(n/|V|,\Gamma)$-blowup of $G$ and $\calP$ is a $\dev_{2,3}(\e^{K}_1,\e^{K}_2(\ell))$-regular $(t,\ell,\e^{K}_1,\e^{K}_2(\ell))$-decomposition of $H$.  Then $\ell \geq (1-2\e_1^{1/64})|U|$. 
\end{proposition}
\begin{proof}
Let $K_0\geq 1$, $q_0(x,y)$, and $\mu_0>0$ be as in Corollary \ref{cor:sldev23} for the constant $3$.  Let $\mu_1>0$ and $p_1(x,y)$ be as in Lemma \ref{lem:LB2}.  Set 
$$
K=48K_0(K_0+1),\text{   }\e^*= \min\{\mu_0,\mu_1,(1/3)^{64}\},\text{ and }p(x,y)=x^{48}q_0(x,y).
$$
 Fix an integer $t\geq 1$, $0<\e_1<\min\{ \frac{1}{2|V|^{64}}, (\frac{1}{3|V|^2})^8,\e^*\}$, and $\e_2:\mathbb{N}\rightarrow (0,1]$ satisfying $\e_2(x)\leq p(\e_1,x^{-1})$. Choose $\e>0$ sufficiently small compared to the previously mentioned parameters, including $t^{-1}$ and $|U|^{-1}$, and assume $n$ is sufficiently large.

Suppose $\Gamma=(A, B; (P_u)_{u\in U})$ is a $U$-colored bigraph with $|A|=|B|=n$ such that for each $u\in U$, $(A, B; P_u)$ is $\e$-regular with density $|U|^{-1}\pm \e$.  Suppose $H$ is an $(n/|V|,\Gamma)$-blowup of $G$.  Assume $\calP$ is a $\dev_{2,3}(\e^{K}_1,\e^{K}_2(\ell))$-regular $(t,\ell,\e^{K}_1,\e^{K}_2(\ell))$-decomposition of $H$. 

We may assume $H$ has vertex set $A\cup B\cup C$ and edge set $E(H)$, where $C=\bigcup_{v\in V}C_v$ and for each $v\in V$, $|C_v|=n/|V|$ and where
$$
E(H)\cap K_3[A,B,C]=\bigcup_{(u,v)\in E(\calG)}\{xyz: xy\in P_u, z\in C_v\}.
$$
Note $|A|=|B|=|C|=n$ and thus $|V(H)|=3n$.  Say $\calP_1=\{V_1,\ldots, V_t\}$. 

Let $\calQ_1=\{A_i,B_i,W_i: i\in [t]\}$, where for each $i\in [t]$, $A_i=V_i\cap A$, $B_i=V_i\cap B$ and $W_i=V_i\cap C$, and define
$$
\calQ_2=\{P\cap (X\times Y): P\in \calP_2, X,Y\in \calQ_1\}.
$$
We have that $\calQ:=(\calQ_1,\calQ_2)$ is a $(3t,\ell)$-decomposition of $V(H)$. Fix an enumeration of $\calQ_2$ as follows.
$$
\calQ_2=\{Q_{A_iB_j}^{\alpha}, Q_{A_iW_k}^{\alpha},Q_{B_jW_k}^{\alpha}: 1\leq i,j,k\leq t, 1\leq \alpha\leq \ell\}.
$$
By Corollary \ref{cor:sldev23}, since $\e_1<1/4$, by our definition of $K$, and by our assumption that $\e_2(\ell)\leq p(\e_1,\ell^{-1})$, we have that $\calQ$ is a $\dev_{2,3}(\e_1,\e_2(\ell))$-regular $(3t,\ell, \e_1,\e_2(\ell))$-decomposition for $H$.  By Lemma \ref{lem:LB2}, there are $A_i,B_j\in \calQ_1$, a set $\calZ\subseteq [\ell]$, and for all $\alpha\in \calZ$, a set $\calC_{\alpha}\subseteq V$ so that the following hold.
\begin{enumerate}[(i)]
\item $|A_i|\geq \e_1|A|/3t$ and $|B_j|\geq \e_1|B|/3t$,
\item  $|\bigcup_{\alpha\in \calZ}Q_{A_iB_j}^{\alpha}|\geq (1-\e_1^{1/4})|A_i||B_j|$,
\item For all $\alpha\in \calZ$, $|\calC_{\alpha}|\geq (1-3\e_1^{1/8})|V|$,
\item For all $\alpha\in \calZ$ and $v\in \calC_{\alpha}$, there is $f(\alpha,v)$ so that the following holds, where we let $N_{\calG}^1(v)=N_{\calG}(v)$ and $N_{\calG}^0(v)=U\setminus N_{\calG}(v)$.
\begin{align}\label{al:qbound}
|Q_{A_iB_j}^{\alpha}\cap \Big(\bigcap_{v\in \calC_{\alpha}}\Big(\bigcup_{u\in N_{\calG}^{f(\alpha,v)}(v)}P_u\Big)\Big)|\geq (1-\e^{1/64}_1)|Q_{A_iB_j}^{\alpha}|.
\end{align}
\end{enumerate}
Note condition (i) above, our choices of parameters, and Lemma \ref{lem:sldev} imply that for each $\alpha\in \calZ$, and each $u\in U$, $(A_i,B_j; \overline{P}_u)$ has $\dev_2(\sqrt{\e})$ and density $|U|^{-1}\pm 2\e$.  Combining with (\ref{al:qbound}), this implies that for all $\alpha\in \calZ$, 
\begin{align}\label{al:qabbound}
\nonumber|Q_{A_iB_j}^{\alpha}|&\leq \frac{1}{1-\e^{1/64}_1}\Big|(A_i\times B_j)\cap \Big(\bigcap_{v\in \calC_{\alpha}}\Big(\bigcup_{u\in N^{f(\alpha,v)}_{\calG}(v)}P_u\Big)\Big)\Big|\\
\nonumber&\leq \frac{1}{1-\e^{1/64}_1} \Big(\frac{1}{|U|}+ 2\e\Big)|\{u\in U: \text{ for each }v\in \calC_{\alpha}, u\in S^{f(v,\alpha)}\}||A_i||B_j|\\
&=\frac{|A_i||B_j|}{1-\e_1^{1/64}}\Big(\frac{1}{|U|}+2\e\Big)\Big|\bigcap_{v\in \calC_{\alpha}}N_{\calG}^{f(\alpha,v)}(v)\Big|.
\end{align}
Since $\e_1$ is sufficiently small compared to $|V|^{-1}$ (in particular, since $\e_1<(3|V|^2)^{-8}$),  condition (iii) implies that for all $\alpha\in \calZ$, $\calC_{\alpha}=V$.  Combining this with the fact that $\calG$ is irreducible, we have that  for all $\alpha\in \calZ$,
$$
\Big|\bigcap_{v\in \calC_{\alpha}}N_{\calG}^{f(\alpha,v)}(v)\Big|=\Big|\bigcap_{v\in V}N_{\calG}^{f(\alpha,v)}(v)\Big|\leq 1.
$$
Plugging this back into inequality (\ref{al:qabbound}), we have that for all $\alpha \in \calZ$,
\begin{align}\label{al:qabbound2}
|Q_{A_iB_j}^{\alpha}|\leq\frac{|A_i||B_j|}{1-\e_1^{1/64}} \Big(\frac{1}{|U|}+2\e\Big).
\end{align}
Combining inequality (\ref{al:qabbound2}) with conclusion (ii) and the fact that $|\calZ|\leq \ell$, we have the following.
$$
(1-\e_1^{1/4})|A_i||B_j|\leq \sum_{\alpha\in \calZ}|Q_{A_iB_j}^{\alpha}|\leq \ell \frac{|A_i||B_j|}{1-\e_1^{1/64}} \Big(\frac{1}{|U|}+2\e\Big).
$$
Rearranging, this yields 
$$
\ell \geq (1-\e_1^{1/4})(1-\e_1^{1/64})\frac{|U|}{1+2\e|U|}\geq (1-2\e_1^{1/64})|U|,
$$ 
where the last inequality uses that $\e$ is sufficiently small compared to $|U|^{-1}$ and $\e_1<\e^*$.  
\end{proof}

\section{Putting it all together}\label{sec:L2}

 This section contains the proof of Theorem \ref{thm:mainL2}.  We will show the gap between polynomial and exponential is characterized by $\VC_2$-dimension. In particular, we show $L_{\calH}$ grows at most polynomially if $\VC_2(\calH)<\infty$, and otherwise, $L_{\calH}$ is bounded below by an exponential.  On the other hand, the gap between constant and polynomial is related what blowups appear in $\calH$. Specifically, we will see that when $\calB_{\calH}$ is finite, $L_{\calH}$ is constant, and if $\calB_{\calH}$ is infinite, then $L_{\calH}$ is at least polynomial.

We begin by considering the constant range, where we  show the upper bound in Corollary \ref{cor:constantub} is tight.

\begin{corollary}\label{cor:LconstLB}
Suppose $\calH$ is a hereditary $3$-graph property and for some integer $L\geq 1$,
$$
L=\max\{K\in \mathbb{N}: \calB_{\calH}\cap Irr(K)\neq \emptyset\}.
$$
Then there is $\e_1^*>0$ and a polynomial $p(x,y)$ so that for all $0<\e_1<\e_1^*$ and all $\e_2:\mathbb{N}\rightarrow (0,1]$ satisfying $\e_2(x)\leq p(\e_1,x^{-1})$,   $L_{\calH}(\e_1,\e_2)=L$. 
\end{corollary}
\begin{proof}
Let $\mu_1>0$ and $p_1(x,y)$ be from   Corollary \ref{cor:constantub}.  Fix some $\calG=(U, V;E)$ from $ \calB_{\calH}\cap Irr(L)$.  Note that by definition of $\textrm{Irr}(L)$, $|U|=L$ and $|V|\leq 2^L$. Let $\mu_2$, $K$ and $p_2(x,y)$ be as in Proposition \ref{prop:LB}, and set
$$
 \e_1^*=\min\{\mu_1,\mu_2^K, \Big(\frac{1}{2|V|^{64}}\Big)^K, \Big(\frac{1}{3|V|^2}\Big)^{8K}, \Big(\frac{1}{2L}\Big)^{64K}\Big\}\text{ and }p(x,y)=p_1(x,y)p^{K}_2(x^K,y).
 $$
 Fix $0<\e_1<\e_1^*$ and all $\e_2:\mathbb{N}\rightarrow (0,1]$ satisfying $\e_2(x)\leq p(\e_1,x^{-1})$. Since $\e_1<\mu_1$ and $\e_2(x)\leq p_1(\e_1,x^{-1})$, Corollary \ref{cor:constantub} implies $L_{\calH}(\e_1,\e_2)\leq L$.  

On the other hand, by definition of $L_{\calH}$, there exists an integer $T$ so that $\psi(\e_1,\e_2,T,L_{\calH}(\e_1,\e_2),\calH)$ holds (we recall $\psi$ is defined in the introduction).  Now let $\e\in (0,1)$ be as in Proposition \ref{prop:LB}, and let $n$ then be sufficiently large. 

Fix $\Gamma=(A, B; (P_u)_{u\in U})$  a $U$-colored bigraph with $|A|=|B|=n$, such that for each $u\in U$, $(A, B; P_u)$ is $\e$-regular with density $|U|^{-1}\pm \e$ (this exists by Lemma \ref{lem:3.8easy}).  Since $\calG\in \calB_{\calH}$, there is some $H\in \calH$ which is an $(n/|V|,\Gamma)$-blowup of $G$.  By our choice of $T$, there exists $\calP$, an $\dev_{2,3}(\e_1,\e_2(\ell))$-regular $(t,\ell,\e_1,\e_2(\ell))$-decomposition of $H$ with  $t\leq T$ and $\ell\leq L_{\calH}(\e_1,\e_2)$.  Let $\e_1'>0$ and $\e_2':\mathbb{N}\rightarrow (0,1]$ be such that $\e_1=(\e_1')^K$ and $\e_2(\ell)=(\e_2'(\ell))^K$. Note $0<\e_1'<\min\{\mu_2, (2| V|)^{-64},(3|V|^2)^{-8}\}$ and  $\e_2'(x)\leq p_2(\e_1',x^{-1})$.  Thus, by Proposition \ref{prop:LB}, 
$$
\ell\geq (1-2(\e'_1)^{1/64})L=(1-2\e_1^{1/64K})L.
$$
 Since $\e_1<(2L)^{-64K}$ and $\ell$ is an integer, this implies $\ell\geq L$, and thus  $L\leq L_{\calH}(\e_1,\e_2)$. 

\end{proof}

As a corollary, we can now prove that a property $\calH$ has $L_{\calH}=1$ if and only if it falls into ranges (2)-(4) in Theorem \ref{thm:strong1}.  This proves Corollary \ref{cor:L1} from the introduction.

\begin{corollary}\label{cor:L1cor}
Suppose $\calH$ is a hereditary $3$-graph property.  Then the following are equivalent.
\begin{enumerate}[(a)]
\item For some $k\geq 1$, $k\otimes U(k)\notin \trip(\calH)$,
\item There exists $\e_1^*>0$ and a polynomial $p(x,y)$ so that for all $0<\e_1<\e_1^*$ and all $\e_2:\mathbb{N}\rightarrow (0,1]$ satisfying $\e_2(x)\leq p(\e_1,x^{-1})$,  $L_{\calH}(\e_1,\e_2)=1$,
\item There exists $\e_1^*,C>0$ and $\e^*_2:\mathbb{N}\rightarrow (0,1]$ so that for all $0<\e_1<\e_1^*$ and all $\e_2:\mathbb{N}\rightarrow (0,1]$ satisfying $\e_2(x)\leq \e_2^*(x)$,  $T_{\calH}(\e_1,\e_2)\leq 2^{\e_1^{-C}}$.
\end{enumerate}
\end{corollary}
\begin{proof}
That (a) and (c) are equivalent follows from  \cite{Terry.2024c}. In particular it is shown in \cite{Terry.2024c} that (a) implies $\calH$ satisfies (2) in Theorem \ref{thm:strong1}, and on the other hand, if (a) is false, then $\calH$ satisfies (1) in Theorem \ref{thm:strong1}. 

That (a) implies (b) is Corollary 5.8 in part 3 \cite{Terry.2024c}.  We have left to show (b) implies (a). Assume (b) holds. By Corollary \ref{cor:LconstLB}, 
$$
1=\max\{K\in \mathbb{N}: \calB_{\calH}\cap Irr(K)\neq \emptyset\}.
$$
 This clearly implies $\VC_2(\calH)<\infty$.  Say $\VC_2(\calH)=s$.  Let $D=D(s)$ be as in Proposition \ref{prop:suffvc2}.  Fix some $\e_1>0$ sufficiently small and some $\e_2:\mathbb{N}\rightarrow (0,1]$ tending to $0$ sufficiently fast, and let  $T$ be such that $\psi(\e_1,\e_2, 1,T,\calH)$ holds. Suppose now $H$ is a sufficiently large element of $\calH$, and assume $\calP$ is a    $\dev_{2,3}(\e_1,\e_2(\ell))$-regular $(t,\ell, \e_1,\e_2(\ell))$-decomposition of $H$ with $\ell=1$.  By Proposition \ref{prop:suffvc2}, $\calP$ is $\e_1^{1/D}$-homogeneous in the sense of Definition \ref{def:homdec}.  Since $\ell=1$, this immediately implies  the vertex partition $\calP_1$ is an $\e_1^{1/D}$-homogeneous partition in the sense of Definition 4.9 of \cite{Terry.2024c}.  We have now shown that all sufficiently large elements in $\calH$ admit $\e_1^{1/D}$-homogeneous partitions (in the sense of Definition 4.9 of \cite{Terry.2024c}) of size at most $T$.  By Theorem 4.14 in \cite{Terry.2024c}, there is some $k\geq 1$ so that $k\otimes U(k)\notin \trip(\calH)$, so (a) holds.
\end{proof}

We next prove a sufficient condition for $L_{\calH}$ to be bounded above and below by polynomials.  We shall see later it is in fact a characterization. 

\begin{corollary}\label{cor:poly}
Suppose $\calH$ has finite $\VC_2$-dimension and $\calB_{\calH}$ contains arbitrarily large bigraphs.  Then there are constants $K,K'>0$, $\e_1^*>0$ and a polynomial $p(x,y)$ so that for all $0<\e_1<\e_1^*$ and all $\e_2:\mathbb{N}\rightarrow(0,1]$ satisfying $\e_2(x)\leq p(\e_1,x^{-1})$, we have
$$
\e_1^{-K}\leq L_{\calH}(\e_1,\e_2)\leq \e_1^{-K'}.
$$
\end{corollary}
\begin{proof}
Let $k=\VC_2(\calH)$. Let $\mu_0$ and $p_0(x,y)$ be from Corollary \ref{cor:polyub} for $k$.  Let $K_1,\mu_1>0$ and $p_1(x,y)$ be as in Proposition \ref{prop:LB} for $k$. Let 
$$
\e_1^*=\min\{2^{-64K_1},\mu_1^{K_1},\mu_2\}\text{ and }p(x,y)=p_0(x,y)p_1^{K_1}(x,y).
$$

Fix $0<\e_1<\e_1^*$ and $\e_2:\mathbb{N}\rightarrow(0,1]$ satisfying $\e_2(x)\leq p(\e_1,x^{-1})$.  Since $\e_1<\mu_0$ and $\e_2(x)\leq p_0(\e_1,x^{-1})$, Corollary \ref{cor:polyub} implies $L_{\calH}(\e_1,\e_2)\leq \e_1^{-O_k(1)}$.

For the lower bound,  let $\e_1'=\e_1^{1/K_1}$ and let $\e_2':\mathbb{N}\rightarrow (0,1]$ be defined by setting $\e_2'(x)=\e_2(x)^{1/K_1}$.  Set $L=4^{8}\e_1^{-1/(64K_1)}$. Note that by definition, 
$$
\e_1'<\min\{\mu_1, (2L)^{-64}, (3L^2)^{-8}\}\text{ and }\e_2'(x)\leq p_1(\e_1',x^{-1}).
$$
 Since $\calB_{\calH}$ is infinite, Lemma \ref{lem:prime} implies that it contains some $\calG=(U,V;E)$ in the set $\{H(L),M(L),\overline{M}(L)\}$.  Note this implies by definition of $H(L), M(L), \overline{M}(L)$ that $|U|=L$ and $|V|=L$.   

  Let  $T$ be such $\psi(\e_1,\e_2,T,L_{\calH}(\e_1,\e_2),\calH)$ holds.  Now let $\e>0$ be as in Proposition \ref{prop:LB}, and let $n$ be sufficiently large.

Let $\Gamma=(A, B; (P_u)_{u\in U})$ be a bipartite $U$-colored graph with $|A|=|B|=n$ such that for each $u\in U$, $(A, B; P_u)$ is $\e$-regular with density $|U|^{-1}\pm \e$ (this exists by Lemma \ref{lem:3.8}). Since $\calG\in \calB_{\calH}$, there is some $H\in \calH$ which is an $(n/|V|,\Gamma)$-blowup of $\calG$.  By our choice of $T$,  there exists $\calP$, a $\dev_{2,3}(\e_1,\e_2(\ell))$-regular $(t,\ell,\e_1,\e_2(\ell))$-decomposition of $H$ with  $1\leq t\leq T$ and $1\leq \ell\leq L_{\calH}(\e_1,\e_2(\ell))$.   By Proposition \ref{prop:LB}, 
$$
\ell\geq (1-(\e_1')^{1/64})|U|=(1-\e_1^{1/(64K_1)})L\geq \frac{L}{2},  
$$
where the last inequality is because $\e_1<\e_1^*$.  We have now shown 
$$
4^7\e_1^{-1/(64K_1)}\leq L\leq  L_{\calH}(\e_1,\e_2).
$$
Since $K_1$ depended only only on $k$, this shows $L_{\calH}(\e_1,\e_2)\geq \e_1^{-O_k(1)}$, which finishes the proof.
\end{proof}

We now prove a sufficient condition for $L_{\calH}$ to be at least exponential.

\begin{corollary}\label{cor:exp}
Suppose $\calH$ is a hereditary $3$-graph property and $\VC_2(\calH)=\infty$.  Then there is a constant $C$,  $\e_1^*>0$, and a polynomial $p(x,y)$ so that for all $0<\e_1<\e_1^*$ and all $\e_2:\mathbb{N}\rightarrow (0,1]$ satisfying $\e_2(x)\leq p(\e_1,x^{-1})$, we have $L_{\calH}(\e_1,\e_2)\geq 2^{\e_1^{-C}}$.
\end{corollary}
\begin{proof}
Let $K_1\geq 1$, $\e^*>0$, and $p_1(x,y)$ be as in Proposition \ref{prop:LB}.  Let 
$$
\e_1^*=\min\{2^{-64K_1},\mu_1^{K_1} \}\text{ and }p(x,y)= p_1(x,y)^{K_1}.
$$
Fix $0<\e_1<\e^*$ and $\e_2:\mathbb{N}\rightarrow (0,1]$ satisfying $\e_2(x)\leq p(\e_1,x^{-1})$.  To ease notation, let $\e_1'=\e_1^{1/K_1}$ and let $\e_2':\mathbb{N}\rightarrow (0,1]$ be defined by setting $\e_2'(x)=\e_2(x)^{1/K_1}$.  Set
$$
 k:=\lfloor \e_1^{-1/(64K_1)}\rfloor =\lfloor (\e'_1)^{-1/64}\rfloor.
 $$
Now  let $\calG=U_{bg}(k)$.  To ease notation, let us say $\calG=(U, V; \calE)$.  Note by definition of $U_{bg}(k)$, $\calG$ is irreducible, $|V|=k$, and $|U|=2^k$ (see Definition \ref{def:ukbg}). Note that by construction,
$$
\e_1'<\min\{\mu_1, \Big(\frac{1}{2|V|}\Big)^{64}, \Big(\frac{1}{3|V|^2}\Big)^8\}.
$$

Let $T\geq 1$ be such that $\psi(\e_1,\e_2, L_{\calH}(\e_1,\e_2),\calH)$ holds.  Now let $\e>0$ be sufficiently small compared to all previously mentioned parameters, including $T^{-1}$ and $|U|^{-1}$.   Let $n$ be sufficiently large.

Let $\Gamma=(A,B; (P_u)_{u\in U})$ be a bipartite $U$-colored graph with $|A|=|B|=n$ such that for each $u\in U$, $(A, B; P_u)$ is $\e$-regular with density $|U|^{-1}\pm \e$ (this exists by Lemma \ref{lem:3.8}). Since $\calH$ has infinite $\VC_2$-dimension, it is an exercise to show there exists $H\in \calH$ so that $H$ is an $(n/|V|,\Gamma)$-blowup of $\calG$ (see Observation \ref{ob:universal2}).  By our choice of $T$,  there exists $\calP$, a $\dev_{2,3}(\e_1,\e_2(\ell))$-regular $(t,\ell,\e_1,\e_2(\ell))$-decomposition of $H$ with  $1\leq t\leq T$ and $1\leq \ell\leq L_{\calH}(\e_1,\e_2(\ell))$.   By Proposition \ref{prop:LB}, 
$$
\ell\geq (1-(\e_1')^{1/64})|U| \geq \frac{|U|}{2}\geq \frac{1}{2} 2^{\e_1^{-1/(64K_1)}},
$$  
where the second inequality is because $\e_1'$ is sufficiently small.  This shows $L_{\calH}(\e_1,\e_2)\geq 2^{\e_1^{-O(1)}}$, which finishes the proof.  
\end{proof}

We can now prove the main theorem about $L_{\calH}$.

\vspace{2mm}

\noindent{\bf Proof of Theorem \ref{thm:mainL2}}
Suppose $\calH$ is a hereditary $3$-graph property.  If $\VC_2(\calH)=\infty$, then the stated exponential lower bound  holds by Corollary \ref{cor:exp}.  If $\VC_2(\calH)<\infty$ and $\calB_{\calH}$ contains arbitrarily large elements, then the stated polynomial upper and lower bounds hold  by Corollary \ref{cor:poly}.  Finally, if  $\calB_{\calH}$ contains finitely many $3$-graphs, then the stated constant upper and lower bounds hold by Corollary \ref{cor:LconstLB}.
\qed

\appendix

 \section{ }
    
  Here we discuss Lemma \ref{lem:3.8} and how to deduce it from \cite{Frankl.2002}.  First, we state the following from \cite{Frankl.2002}.

\begin{lemma}\label{lem:3.81}
Fix $0<\e<1/2$,  $\rho\geq 2\e$, $0<p<\rho/2$, and $u=[1/p]$.  Suppose $G=(U\cup V,R)$ is an $\e$-regular bipartite graph, where $m=\min\{|U|, |V|\}\geq m_0(\e, u)$, and $d_G(U,V)=\rho$, and $\e\geq 10(p/m)^{1/5}$.  Then there exists a partition 
$$
R=E_0\cup \ldots \cup E_u,
$$
so that $|E_0|\leq \rho p(1+o(1))|U||V|$, and for each $1\leq \alpha\leq u$, $(U\cup V,E_{\alpha})$ is $\e$-regular  and has density $\rho p(1+o(1))$, where $o(1)\rightarrow 0$ as $m\rightarrow \infty$.  Moreover, if $1/p$ is an integer, then $E_0=\emptyset$.
\end{lemma}
 
Lemma \ref{lem:3.81} is a mild generalization of Lemma 3.8 of \cite{Frankl.2002}, the only difference being that above we allow $U$ and $V$ to have possibly distinct sizes, while \cite{Frankl.2002} assumes they have the same size.  It  is not difficult to see an identical proof to that in \cite{Frankl.2002} will also prove Lemma \ref{lem:3.81}.  We now deduce Lemma \ref{lem:3.8} from Lemma \ref{lem:3.81}.   

\vspace{2mm}

\noindent{\bf Proof of Lemma \ref{lem:3.8}.}
Fix $0<\e<(1/2)^{12}$,  $\rho\geq 2\e^{1/12}$, $0<p<\rho/2$, and $u=[1/p]$.  Let $m_0=m(\e^{1/12},u)$ be from Lemma \ref{lem:3.81}.    Suppose $G=(U,V;R)$ is a bigraph satisfying $\dev_2(\e)$, where $m=\min\{|U|, |V|\}\geq m_0$,  $d_G(U,V)=\rho$, and $\e\geq 10(p/m)^{1/5}$. Consider the bipartite graph $(A\cup B, E)$ where $A$ is a copy of $U$, $B$ is a copy of $V$, and $E$ is the set of the unordered pairs corresponding to the ordered pairs in $R$. By Theorem \ref{thm:equiv}, $(A\cup B, E)$ is $\e^{1/12}$-regular, and clearly its density is $\rho$.   Thus, we can apply Lemma \ref{lem:3.81} to $(A\cup B, E)$ to obtain a partition 
$$
E=E(0)\cup E(1)\cup \ldots \cup E(u),
$$
where $u=[1/p]$, $|E(0)|\leq (1+o(1))\rho p |U||V|$, and each of $ E(1),\ldots, E(u)$ are $\e^{1/12}$-regular with density $(1\pm o(1))\rho p$.  This yields a corresponding partition 
$$
R=R(0)\cup R(1)\cup \ldots \cup R(u),
$$
where $|R(0)|\leq (1+o(1))\rho p |U||V|$, and where for each $1\leq \alpha\leq u$, $(U,V; R(\alpha))$ satisfies $\dev_2(\e^{1/12})$ with density $(1\pm o(1))\rho p$  (here we are again using Theorem \ref{thm:equiv}).  
\qed

\vspace{2mm}

\bibliography{growth.bib}
\bibliographystyle{amsplain}

\end{document}